\documentclass[11pt]{article}

\usepackage{morefloats}
\usepackage{amsfonts}
\usepackage{color}
\usepackage{multirow}
\usepackage{latexsym}
\usepackage{amsmath,amssymb,amsthm}
\usepackage{dsfont}
\usepackage{extarrows}
\usepackage{hyperref}
\usepackage{diagbox}
\usepackage{booktabs}
\usepackage{tikz}       % Graph mit Latex Code aus R einfügen
\usepackage{float}
\usepackage{subcaption}
\usepackage[varg]{txfonts}
\usepackage{enumitem} % Package for handling of lists
\usepackage{array}
\usepackage[font=small,labelfont=bf,margin=1cm]{caption}

\numberwithin{equation}{section}

\newtheoremstyle{thm}% name
{9pt}%      Space above, empty = `usual value'
{9pt}%      Space below
{\itshape}% Body font
{}%         Indent amount (empty = no indent, \parindent = para indent)
{\bfseries}% Thm head font
{.}%        Punctuation after thm head
{ }% Space after thm head: \newline = linebreak
{}%         Thm head spec
\theoremstyle{thm}

\newtheorem{theorem}{Theorem}[section]
\newtheorem{lemma}[theorem]{Lemma}
\newtheorem{corollary}[theorem]{Corollary}
\newtheorem{prop}[theorem]{Proposition}

\newtheoremstyle{def}% name
{9pt}%      Space above, empty = `usual value'
{9pt}%      Space below
{}% Body font
{}%         Indent amount (empty = no indent, \parindent = para indent)
{\bfseries}% Thm head font
{.}%        Punctuation after thm head
{ }% Space after thm head: \newline = linebreak
{}%         Thm head spec
\theoremstyle{def}

\newtheorem{remark}[theorem]{Remark}
\newtheorem{example}[theorem]{Example}

\def\cd{\stackrel{\mathcal{D}}{\longrightarrow}}

\newcommand{\C}{\mathbb{C}} % komplexe
\newcommand{\R}{\mathbb{R}} % reelle
\newcommand{\N}{\mathbb{N}} % natuerliche
\newcommand{\Sp}[1][d-1]{\mathcal{S}^{#1}} % Einheitssphäre
\newcommand{\Pb}{\mathbb{P}} % Wahrscheinlichkeit 1
\newcommand{\E}{\mathbb{E}} % Erwartungswert
\newcommand{\standardsumme}[1][j]{\sum_{#1 = 1}^{n}}
\newcommand{\spheremax}{\underset{b \in \Sp}{\max}}
\newcommand{\skalarprodukt}[2]{{#1}^\top \vspace{-0.05cm} #2}
\newcommand{\unifsphere}{\mathcal{U}\left(\Sp\right)}
\newcommand{\harmonicspace}{\mathcal{H}_k(\Sp)}
\newcommand{\weakconv}{\overset{\mathcal{D}}{\longrightarrow}}

\newcommand{\skalarproduktLebesgue}[2]{\langle #1, #2 \rangle_{L^2}}

\newcommand{\limkappa}{\lim_{\kappa \rightarrow 0^+}}

\renewcommand{\footnoterule}{%
	\kern -3.5pt
	\hrule width \textwidth height 1pt
	\kern 3.5pt
}

\makeatletter
\def\blfootnote{\xdef\@thefnmark{}\@footnotetext}
\makeatother

\begin{document}

\title{\bf A general maximal projection approach to uniformity testing on the hypersphere}

%\titlerunning{Distributional Characterizations}        % if too long for running head

\author{Jaroslav Borodavka and Bruno Ebner}

\date{\today}
\maketitle


\begin{abstract}
We propose a novel approach to uniformity testing on the $d$-dimensional unit hypersphere $\mathcal{S}^{d-1}$ based on maximal projections. This approach gives a unifying view on the classical uniformity tests of Rayleigh and Bingham, and it links to measures of multivariate skewness and kurtosis. We derive the limiting distribution under the null hypothesis using limit theorems for Banach space valued stochastic processes and we present strategies to simulate the limiting processes by applying results on the theory of spherical harmonics. We examine the behavior under contiguous and fixed alternatives and show the consistency of the testing procedure for some classes of alternatives. For the first time in uniformity testing on the sphere, we derive local Bahadur efficiency statements. We evaluate the theoretical findings and empirical powers of the procedures in a broad competitive Monte Carlo simulation study and, finally,  apply the new tests to a data set on midpoints of large craters on the moon.
\end{abstract}

\section{Introduction}
\label{sec:Intro}
Testing uniformity on the circle, the sphere and the hypersphere $\Sp=\{x\in\R^d:\|x\|=1\}$, $d\in\N$, $d\ge2$, of $\R^d$, endowed with the Euclidean norm $\|x\|=\sqrt{x^\top x}$, are classical and still up-to-date research fields in directional statistics. Here and in the following, $^\top$ stands for the transpose of a matrix or a vector. We numerate just a small subset of fields, where data on the surface of the unit hypersphere $\Sp$ is applied: meteorology, geology, paleomagnetism, political sciences, text mining and wildfire orientation, for examples of such datasets, see \cite{LV:2019} and the contributions therein. The first step to serious statistical inference on $\Sp$ is to check whether or not a sample of unit vectors stems from the uniform law, since this distribution characterizes the absence of structure in directional data. To be specific, we model the observed data by independent identically distributed (iid.) column random vectors $U,U_1,\ldots,U_n$ taking values in $\Sp$. The testing problem of interest is whether or not the hypothesis
\begin{equation*}
H_0:\Pb^U=\unifsphere
\end{equation*}
holds, against general alternatives. Here, $\Pb^U$ stands for the distribution of $U$ and $\mathcal{U}(\cdot)$ denotes the uniform distribution. This problem has been extensively studied in the literature: Lord Rayleigh presented the first test of uniformity in \cite{R:1919} based on the norm of the arithmetic mean. Rayleigh's test was followed by circular tests based on the classical goodness-of-fit measures of Kolmogorov-Smirnov type in \cite{K:1960} and of Cram\'{e}r-von Mises type in \cite{W:1961}. Later, Bingham developed a test of uniformity in \cite{B:1974} based on the sample scatter matrix and Gin\'e, see \cite{G:1975}, introduced the so-called Sobolev-tests. We refer to \cite{JS:2001,MJ:2000} for more details on these tests and to \cite{JMV:2020} for some new developments. More recently, \cite{CCF:2009} proposed a Kolmogorov-Smirnov type test based on random projections, \cite{EHY:2018} suggest a procedure using powers of volumes of nearest-neighbor spheres, and \cite{GNC:2020} consider the Cram\'{e}r-von Mises counterpart to \cite{CCF:2009}. For details on this approach as well as more recent developments in uniformity testing of axial data see \cite{LV:2017}, chapter 6. The authors of the review article \cite{GV:2018} give an overview of uniformity tests on the hypersphere. Comparative Monte Carlo simulation studies are found in \cite{DFL:1985} for $d=3$ and for higher dimensions in \cite{F:2007}.

A well-known characterizing property of $\unifsphere$ is invariance with respect to rotations about the origin. Any test (say) $T_n$ of uniformity should therefore inherit this structure and as such be invariant under rotations, i.e.
\begin{equation}\label{eq:rotinv}
T_n(AU_1,\ldots,AU_n)=T_n(U_1,\ldots,U_n)\quad\mbox{holds for all}\;A\in \mbox{SO}(d),
\end{equation}
where $\mbox{SO}(d)$ is the $d$-dimensional rotation group, i.e., for $d\times d$-matrices $A\in\mbox{SO}(d)$ we have $AA^\top=A^\top A=I_d$ and $\mbox{det}(A)=1$. We denote the identity matrix by $I_d$, and $\mbox{det}(\cdot)$ is the notation for the determinant of a matrix. In the following, we call the property (\ref{eq:rotinv}) rotational invariance of the test statistic $T_n$.

We propose a novel class of statistics $T_{n,\, \beta}$ based on powers of maximal projections. In this spirit assume $U\sim\unifsphere$ and by \cite{B:1997}, we have using the rotational invariance of the uniform distribution and symmetry arguments for every $b\in\Sp$ and $\beta\in\N$
\begin{equation}\label{eq:ewert}
\psi_d(\beta)=E(b^\top U)^\beta=\left\{\begin{array}{cc}\displaystyle\frac{\Gamma\left((\beta+1)/2\right)\Gamma\left(d/2\right)}{\sqrt{\pi}\Gamma\left((\beta+d)/2\right)}, & \mbox{if}\,\beta\,\mbox{is even},\\ 0, & \mbox{if}\,\beta\,\mbox{is odd},\end{array}\right.
\end{equation}
where $\Gamma(\cdot)$ denotes the Gamma function. Hence $\psi_d(\beta)$ is independent of the choice of $b$, a property that likewise follows by the rotation invariance of the uniform law on the sphere. Next, we define the family of statistics
\begin{equation}\label{eq:test}
T_{n,\, \beta}=T_{n,\, \beta}(U_1,\ldots,U_n)=n\max_{b\in\Sp}\left(\frac1n\sum_{j=1}^n(b^\top U_j)^\beta-\psi_d(\beta)\right)^2,\quad \beta\in\N.
\end{equation}
It is obvious that $T_{n,\, \beta}$ is rotational invariant for every $\beta$ due to the rotational invariance of the maximum functional.

Interestingly, $T_{n,\, \beta}$ has close connections to well-known classical tests such as the Rayleigh test, the Bingham test and to measures of multivariate skewness and kurtosis by Malkovich and Afifi.
First, notice that with the sample mean of the observations $\overline{U}_n=\frac1n\sum_{j=1}^nU_j$ we have
\begin{equation*}
T_{n,1}=n\max_{b\in\Sp}\left(\frac1n\sum_{j=1}^n b^\top U_j\right)^2=n\|\overline{U}_n\|^2\max_{b\in\Sp}\left(b^\top\frac{\overline{U}_n}{\|\overline{U}_n\|}\right)^2=n\|\overline{U}_n\|^2,
\end{equation*}
since the scalar product in the maximum is the cosine of the angle between the two unit vectors, which takes its maximum for $b=\frac{\overline{U}_n}{\|\overline{U}_n\|}$. Hence we have an equivalent test as the classical Rayleigh test, see \cite{R:1919}, given by $R_n=2nd\|\overline{U}_n\|^2$.

Second, with the sample scatter matrix $S=\frac1n\sum_{j=1}^nU_jU_j^\top$ we have
\begin{equation*}
T_{n,2}=n\max_{b\in\Sp}\left(\frac1n\sum_{j=1}^n(b^\top U_j)^2-\frac1{d}\right)^2=n\max_{b\in\Sp}\left(b^\top S b-\frac1d\right)^2=n\max_{b\in\Sp}\left(b^\top \left(S-\frac1dI_d\right) b\right)^2,
\end{equation*}
Notice that $T_{n,2}$ is the squared spectral norm of $S-E(UU^\top)$ for $U\sim\unifsphere$, hence it compares the scatter matrix to the covariance matrix of $U$, which is in the same spirit as the Bingham test, see \cite{B:1974}. Note that by the Courant–Fischer–Weyl min-max principle from linear algebra, we have
 \begin{equation*}
T_{n,2}=n(\max(|\lambda_{min}|,|\lambda_{max}|))^2,
\end{equation*}
where $\lambda_{min}$ and $\lambda_{max}$ are the minimal and maximal eigenvalues of the symmetric matrix $S-\frac1dI_d$.

Third, we have
\begin{equation}\label{eq:Test3}
T_{n,3}=n\max_{b\in\Sp}\left(\frac1n\sum_{j=1}^n(b^\top U_j)^3\right)^2\quad \mbox{and}\quad T_{n,4}=n\max_{b\in\Sp}\left(\frac1n\sum_{j=1}^n(b^\top U_j)^4-\frac3{d(d+2)}\right)^2,
\end{equation}
which can be interpreted as analogs to the multivariate sample skewness and sample kurtosis by Malkovich and Afifi, for a definition see \cite{MA:1973}. For $T_{n,\, \beta}$, $\beta>2$, no explicit closed form and easy to calculate formula is known. The authors of \cite{MA:1973} suggest using the Newton-Raphson method to obtain a good approximation of the maximal value in (\ref{eq:test}). Since for such a numerical routine, the choice of some good start values is not straightforward, we suggest to use a random approach, see Section \ref{sec:Simu}, which is related to the idea of random projections as suggested in \cite{CCF:2009}.

The rest of the paper is organized as follows: We present asymptotic theory under the null hypothesis in Section \ref{sec:asy_null}. In Section \ref{sec:asy_cont} we derive the behaviour of $T_{n,\beta}$ for contiguous alternatives. We show consistency of the tests against some classes of fixed alternatives in Section \ref{sec:asy_fixed}. Afterwards, we establish local approximate and exact asymptotic relative efficiency statements in the Bahadur sense in Section \ref{sec:Bah}. We examine the theoretical findings by a Monte Carlo simulation study in Section \ref{sec:Simu} and provide a real data application to midpoints of large craters on the moon in Section \ref{sec:real_data}. Conclusions as well as an outlook are provided in Section \ref{sec:conclusions}. We finish the article by three Appendices \ref{sec:AppA}, \ref{sec:AppB} and \ref{sec:AppC} that contain facts on $d$-dimensional Legendre polynomials and spherical harmonics, as well as some technical Lemmas and proofs.

\section{Asymptotic null distribution of \texorpdfstring{$T_{n,\, \beta}$}{Lg}}
\label{sec:asy_null}
 Let $C(\Sp,\R)$ be the Banach space of continuous functions $f:\Sp\rightarrow\R$, equipped with the norm $\|f\|_\infty=\sup_{b\in\Sp}|f(b)|$. We introduce the stochastic process
\begin{equation*}
Z_{n,\beta}(b)=\sqrt{n}\left(\frac1n\sum_{j=1}^n(b^\top U_j)^\beta-\psi_d(\beta)\right),\quad b\in\Sp.
\end{equation*}
For the covariance structure in the following theorem, we write
\begin{align} \label{eta}
        \eta_\beta(b,c)
        = \E (\skalarprodukt{b}{U})^{\, \beta} (\skalarprodukt{c}{U})^{\, \beta}
        = \sum_{j=0}^\beta \frac{(c_{j,\, d}(\beta))^2}{\nu_d(j)} P_j^{\, d}(\skalarprodukt{b}{c}),\quad b,c\in\Sp.
\end{align}
Here, $P_j^{\, d}(\cdot)$ is the $d$-dimensional Legendre polynomial of order $j$, for a definition see \eqref{def_legendre}, $\nu_d(j)$ is the dimension of the space of $d$-dimensional spherical harmonics of order $j$, see \eqref{eq:dimsp}, and $c_{j,d}(\beta)$ are constants only depending on $j,d$ and $\beta$, compare with \eqref{potenzen_legendre} and Proposition \ref{c_j,d(m) allgemein}. An explicit way of calculation can be found in Appendix \ref{sec:AppB}.

\begin{theorem}\label{thm:asy_emp}
Let $U_1,\ldots,U_n$ be iid. with $U_1\sim\mathcal{U}\left(\Sp\right)$. For fixed $\beta\in\N$ there exists a centred Gaussian process $Z_\beta(b)$, $b\in\Sp$ with continuous sample paths and covariance kernel
\begin{equation}\label{eq:covk}
\rho_\beta(b,c)=\eta_\beta(b,c)-\psi_d^2(\beta),\quad b,c\in\Sp.
\end{equation}
Regarding $Z_\beta(\cdot)$ as a random element of $C(\Sp,\R)$, we have
\begin{equation*}
Z_{n,\beta}(\cdot)\stackrel{\mathcal{D}}{\longrightarrow}Z_{\beta}(\cdot).
\end{equation*}
\end{theorem}
\begin{remark} \label{rem:rho_beta_explicit}
For the special cases of Section \ref{sec:Intro} and some higher powers $\beta$, we have
\begin{eqnarray*}
\rho_1(b,c)&=&\frac{1}{d}(b^\top c),\\
\rho_2(b,c)&=&\frac{1}{d(d+2)}\left(2(b^\top c)^2+1\right)-\frac1{d^2},\\
\rho_3(b,c)&=&\frac{1}{d(d+2)(d+4)}\left(6(b^\top c)^3+9(b^\top c)\right),\\
\rho_4(b,c)&=&\left(24(b^\top c)^4+72(b^\top c)^2+9\right)\prod_{j=0}^3(d+2j)^{-1}-\frac9{d^2(d+2)^2},\\
\rho_5(b,c)&=&\left(120(b^\top c)^5+600(b^\top c)^3+225(b^\top c)\right)\prod_{j=0}^4(d+2j)^{-1},\\
\rho_6(b,c)&=&\left(720(b^\top c)^6+5400(b^\top c)^4+4050(b^\top c)^2+225\right)\prod_{j=0}^5(d+2j)^{-1}-\frac{225}{d^2(d+2)^2(d+4)^2},
\end{eqnarray*}
and thus explicit formulas for the covariance kernel in Theorem \ref{thm:asy_emp}.
\end{remark}
Note that the covariance kernel $\rho_\beta(b,c)$ solely depends on the scalar product $b^\top c$ and hence can be written as a function (say) $\rho_\beta(b,c)=Q(b^\top c)$, where $Q\in C([-1,1],\R)$ is a polynomial of degree $\beta$. Kernels of this particular structure are called zonal kernels, for an application of Gaussian processes with zonal covariance kernel in machine learning see \cite{DDH:2020}. The fact that $|\rho_\beta(b,c)|\le 1$ for all $\beta\in\N$ follows by the inequalities of Cauchy--Schwarz and Popoviciu, since the projections are bounded random variables.

Define the integral operators $K_\beta$ for $\beta\in\N$ given by
\begin{equation} \label{K_beta}
    K_\beta f(x) = \frac{1}{|\Sp|} \int_{\Sp} \rho_\beta(\omega, x) f(\omega) \, \text{d}\sigma(\omega), \quad x \in \Sp, \; f \in L^2(\Sp, \text{d}\sigma),
\end{equation}
where integration is with respect to the unique spherical Lebesgue measure $\sigma$ on $\Sp$. Since $\rho_\beta$ is continuous on a compact set of $\R^d$, the operator $K_\beta$ is compact from $L^2(\Sp, \text{d}\sigma)$ to $L^2(\Sp, \text{d}\sigma)$. Due to the zonal covariance structure we can even show that $K_\beta$ is a finite-rank operator, i.e., an operator whose range is finite-dimensional. The latter, and other properties, are presented and proved in the next proposition. In the following, we denote by $\mathcal{H}_k(\Sp)$ the space of $d$-dimensional spherical harmonic functions of order $k\in\N_0$, for a definition see \cite{G:1996}.

\begin{prop} \label{K_beta_properties}
Let $\beta \in \N$ and $K_\beta$ be defined as in \eqref{K_beta}.
\begin{enumerate}[label=\roman*)]
    \item For any spherical harmonic $\phi \in \mathcal{H}_k(\Sp)$ of order $k \in \N_0$, we have $K_\beta \phi = \lambda_k \phi$, where
    \begin{align} \label{eq:eigenvalues}
    \lambda_k
    = \begin{cases}
        \begin{alignedat}{2}
            & \hspace{0.5cm} \left( \frac{c_{k,\, d}(\beta)}{\nu_d(k)} \right)^2   &,\quad& \textrm{for} \; 0 < k \leq \beta,\\[0.25cm]
	        & \hspace{1.25cm} 0                                                     &,\quad& \textrm{for} \; k = 0 \textrm{ or } k > \beta,
    \end{alignedat}
	\end{cases}
    \end{align}
    with constants $c_{k, \, d}(\beta) \in \R$, depending only on $k, d$ and $\beta$, compare with Proposition \ref{c_j,d(m) allgemein}, and $\nu_d(k) = \textrm{dim}(\mathcal{H}_k(\Sp))$.
    \item $K_\beta$ is a finite-rank operator.
    \item The spectrum of $K_\beta$ consists of $0$ and the eigenvalues in \eqref{eq:eigenvalues}.
    \item $K_\beta$ is positive, i.e., we have $\left\langle K_\beta f, f \right\rangle_{L^2} \geq 0$ for all $f \in L^2(\Sp, d\sigma).$
\end{enumerate}
\end{prop}

In the spirit of \cite{BH:1991}, we thus have alternative representations of the limiting Gaussian process for our special cases.

\begin{prop} \label{pro:AltD}
Let $\nu_d(k)$ be the dimension of the space of $d$-dimensional spherical harmonics of order $k\in\N$, see \eqref{eq:dimsp}.
\begin{enumerate}[label=\roman*)]
\item If $\beta$ is odd, the limiting Gaussian process $Z_\beta(b)$, $b\in\Sp$, can be represented in the form
\begin{equation*}
Z_\beta(b) = \sqrt{|\Sp|} \sum_{\substack{ k=1 \\ k \textrm{ odd}}}^\beta \sqrt{\lambda_k} \sum_{j=1}^{\nu_d(k)} \phi_{k,j}(b) N_{k,j}, \quad b \in \Sp.
\end{equation*}
Here, $N_{k,j}$, $k=1,3,\ldots$ and $j=1,2,\ldots,\nu_d(k)$, is an array of independent unit normal variables, $\lambda_k$ is the eigenvalue in \eqref{eq:eigenvalues}, and $\varphi_{k,j}$ are $j=1,2,\ldots,\nu_d(k)$ linearly independent surface harmonics of degree $k$ being orthonormal with respect to $\sigma/|\Sp|$, compare with the proof of Proposition \ref{K_beta_properties}, ii).

\item If $\beta$ is even, the limiting Gaussian process $Z_\beta(b)$, $b\in\Sp$, can be represented in the form
\begin{equation*}
Z_\beta(b) = \sqrt{|\Sp|} \sum_{\substack{ k=1 \\ k \textrm{ even}}}^\beta \sqrt{\lambda_k} \sum_{j=1}^{\nu_d(k)} \phi_{k,j}(b) N_{k,j}, \quad b \in \Sp.
\end{equation*}
Here, $N_{k,j}$, $k=0,2,4,\ldots$ and $j=1,2,\ldots,\nu_d(k)$, is an array of independent unit normal variables, $\lambda_k$ is the eigenvalue in \eqref{eq:eigenvalues}, and $\varphi_{k,j}$ are $j=1,2,\ldots,\nu_d(k)$ linearly independent spherical harmonics of degree $k$ being orthonormal with respect to $\sigma/|\Sp|$, compare with the proof of Proposition \ref{K_beta_properties}, ii).
\end{enumerate}
\end{prop}
Proposition \ref{pro:AltD} shows an easy way to simulate Gaussian random processes on the sphere with a polynomial covariance kernel. What is essentially needed are three ingredients: the positive eigenvalues (which can be calculated explicitly), an array of independent unit normal variables and an implementation of spherical harmonics, see Section \ref{sec:Simu} for more details. For a generation method of a suitable basis of spherical harmonics, see \cite{AH:2012}, Section 2.11, or \cite{DX:2013}, Theorem 1.1.9. The package \texttt{HFT.m} in \texttt{Mathematica}, see \cite{Mathematica}, provides a direct way to calculate an orthonormal basis of spherical harmonics in any dimension $d$ and any order $k$ based on Theorem 5.25 in \cite{ABW:2001}.  Note that explicit versions of orthonormal systems up to order 4 in any dimensions can be found in \cite{MQ:2001}, Tables 1 and 2.
\begin{remark} \label{rem:AltD_explicit_cases}
\begin{itemize}
    \item Case $\beta = 1$: We have $\nu_d(1) = d$ and $u \mapsto u_k \in \mathcal{H}_1(\Sp)$ for all $k = 1,\ldots,d$, where $u = (u_1,\ldots,u_d) \in \Sp$. These functions form an orthogonal system of $\mathcal{H}_1(\Sp)$, see \cite{G:1996}, Lemma 3.2.3. Normalization w.r.t. $\sigma$ yields the orthonormal basis functions
    \begin{equation*}
        u \mapsto \sqrt{\frac{d}{|\Sp|}} u_k \in \mathcal{H}_1(\Sp), \quad k = 1,\ldots,d.
    \end{equation*}
    We have a single positive eigenvalue $\lambda_1 = 1/d^2$. With Proposition \ref{pro:AltD} it follows
    \begin{equation*}
        Z_1(u) = \frac{1}{\sqrt{d}} \sum_{j=1}^d u_j N_j, \quad u \in \Sp, \quad N_1,\ldots,N_d \overset{\text{uiv}}{\sim} \mathcal{N}(0,1).
    \end{equation*}
    Moreover, putting $N = (N_1,\ldots,N_d) \sim \mathcal{N}_d(0,I_d)$ the Cauchy--Schwarz Inequality yields
    \begin{equation*}
        Z_1^2(u) = \frac{1}{d} \sum_{j,k=1}^d u_j u_k N_j N_k = \frac{1}{d} (\skalarprodukt{u}{N})^2 \leq \frac{1}{d} \| N \|^2.
    \end{equation*}
    Hence we obtain
    \begin{equation*}
        d \max_{u \in \Sp} Z_1^2(u) = \| N \|^2 \sim \chi_{d}^2,
    \end{equation*}
    thus recovering the limit result of the Rayleigh-Test.
    \item Case $\beta = 2$: By \eqref{eq:dimsp} it is $\nu_d(2) = (d+2)(d-1)/2$. Straightforward calculations yield the single positive eigenvalue
    \begin{equation*}
        \lambda_2 = \left( \frac{2}{d(d+2)} \right)^2.
    \end{equation*}
    Let $\phi_{2,1},\ldots, \phi_{2,\nu_d(2)}$ be an orthonormal basis of $\mathcal{H}_2(\Sp)$, see \cite{MQ:2001}, Table 1, for an explicit representation, and set $ \phi_2(u) = ( \phi_{2,1}(u),\ldots, \phi_{2,\nu_d(2)}(u) )$, $u \in \Sp$. With Proposition \ref{def_legendre} and \ref{legendre_schranke} it follows that
    \begin{equation*}
        \| \phi_2(u) \|^2 = \sum_{i=1}^{\nu_d(2)} \phi_{2,i}(u) \phi_{2,i}(u) = \frac{\nu_d(2)}{|\Sp|} P_2^{\, d}(\skalarprodukt{u}{u}) = \frac{\nu_d(2)}{|\Sp|}, \quad u \in \Sp.
    \end{equation*}
    Therefore, putting $N = (N_1,\ldots,N_{\nu_d(2)}) \sim \mathcal{N}_{\nu_d(2)}\left( 0,I_{\nu_d(2)} \right)$, gives us again using the Cauchy--Schwarz Inequality
    \begin{equation*}
        Z_2^2(u) = \frac{4 |\Sp|}{d^2(d+2)^2} \left( \sum_{j=1}^d \phi_{2,j}(u) N_j \right)^2 \leq \frac{4\nu_d(2)}{d^2(d+2)^2} \| N \|^2 = \frac{2(d-1)}{d^2(d+2)} \| N \|^2.
    \end{equation*}
    Since $\| N \|^2 \sim \chi_{\nu_d(2)}^2$, a comparison of 95\% quantiles of $2(d-1)\| N \|^2/(d^2(d+2))$ with Table \ref{tab:table_crit_values} shows that this upper bound is only a good approximation for $d=2$.
\end{itemize}
In the following we give a list of the non-null eigenvalues $\lambda_k$, $k=0,\ldots,\beta$, in \eqref{eq:eigenvalues} corresponding to higher values of $\beta$.
\begin{itemize}
    \item Case $\beta = 3$: $\lambda_1 = \left( 3/(d(d+2)) \right)^2$ and $\lambda_3 = \left( 6/(d(d+2)(d+4)) \right)^2$.
    \item Case $\beta = 4$: $\lambda_2 = \left( 12/(d(d+2)(d+4)) \right)^2$ and $\lambda_4 = \left( 24/(d(d+2)(d+4)(d+6)) \right)^2$.
    \item Case $\beta = 5$: $\lambda_1 = \left( 15/(d(d+2)(d+4)) \right)^2$, $\lambda_3 = \left( 60/(d(d+2)(d+4)(d+6)) \right)^2$, and $\lambda_5 = \left( 120/(d(d+2)(d+4)(d+6)(d+8)) \right)^2$.
    \item Case $\beta = 6$: $\lambda_2 = \left( 90/(d(d+2)(d+4)(d+6)) \right)^2$, $\lambda_4 = \left( 360/(d(d+2)(d+4)(d+6)(d+8)) \right)^2$, as well as $\lambda_6 = \left( 720/(d(d+2)(d+4)(d+6)(d+8)(d+10)) \right)^2$.
\end{itemize}
\end{remark}
Since $\Sp$ is compact, a direct application of the continuous mapping theorem and Theorem \ref{thm:asy_emp} prove the following Corollary to Theorem \ref{thm:asy_emp}.
\begin{corollary}\label{cor:H0}
Let $U_1,\ldots,U_n$ be iid. with $U_1\sim\unifsphere$. Then we have
\begin{equation*}
T_{n,\, \beta}\weakconv\spheremax Z^2_\beta(b),
\end{equation*}
where $Z_\beta(\cdot)$ is the limiting Gaussian process of Theorem \ref{thm:asy_emp}.
\end{corollary}
The resulting limit random variable in Corollary \ref{cor:H0} is not of pure theoretic interest, since the distribution and hence the asymptotic critical value can be approximated, see Section \ref{sec:Simu}.

\section{Contiguous alternatives}
\label{sec:asy_cont}
In this section, we consider a triangular array $U_{n1},\ldots,U_{nn}$ of rowwise identically independent distributed random vectors on $\Sp$ having the density function $f_n(x)=\mu(x)\left(1+h(x)/\sqrt{n}\right),$ $x\in\Sp,$ where $\mu(\cdot)$ denotes the density of the uniform distribution with respect to the spherical Lebesgue measure $\sigma$, and $h$ is a bounded measurable function satisfying $\int_{\Sp}h(x)\mu(x)\,\mbox{d}\sigma(x)=0$. We consider $n$ large enough to assure the non-negativity of $f_n$.

First, define
\begin{eqnarray*}
\mathbb{P}^{(n)}=\bigotimes_{j=1}^n(\mu \; \sigma)&\text{and}& \mathbb{A}^{(n)}=\bigotimes_{j=1}^n(f_n \; \sigma)
\end{eqnarray*}
on the measurable space $(\mathfrak{X}_n,\mathfrak{B}_n)=\displaystyle{\bigotimes_{j=1}^n(\Sp,\mathcal{M})}$, where $\mathcal{M}$ denotes the class of subsets of $\Sp$ that are measurable with respect to $\sigma$. Further, denote the likelihood ratio with $L_n=\frac{d\mathbb{A}^{(n)}}{d\mathbb{P}^{(n)}}$ and write
\begin{equation*}
S_\beta: \Sp\times \Sp\rightarrow \R,\quad (a,b)\mapsto S_\beta(a,b)=(a^\top b)^{\, \beta}-\psi_d(\beta),
\end{equation*}
where $\psi_d(\cdot)$ is defined in (\ref{eq:ewert}).
\begin{theorem}\label{thm:ben_alt}
Under the standing assumptions we have for the triangular array $U_{n1},\ldots,U_{nn}$
\begin{equation*}
Z_{n,\beta}(\cdot)\stackrel{\mathcal{D}}{\longrightarrow}Z_\beta(\cdot)+S^*_\beta(\cdot) \quad \text{under} \; \; \mathbb{A}^{(n)}
\end{equation*}
in $C(\Sp,\R)$, where $Z_\beta(\cdot)$ is a centred Gaussian process in $C(\Sp,\R)$ having covariance kernel $\rho_\beta$ from Theorem \ref{thm:asy_emp}. The shift function $S^*_\beta(\cdot)$ is given by
\begin{equation}\label{ShiftparaC}
S^*_\beta(b)=\frac1{|\Sp|}\int_{\Sp}S_\beta(u,b)h(u)\;\mbox{d}\sigma(u),\;\;b\in\Sp.
\end{equation}
\end{theorem}
As a direct consequence of Theorem \ref{thm:ben_alt} and the continuous mapping theorem we have the following Corollary.
\begin{corollary}
Under the conditions of Theorem \ref{thm:ben_alt}, we have
\begin{equation*}
T_{n,\, \beta}\stackrel{\mathcal{D}}{\longrightarrow} \max_{b\in\Sp}\left(Z_\beta(b)+S^*_\beta(b)\right)^2 \quad \text{under} \; \; \mathbb{A}^{(n)}.
\end{equation*}
\end{corollary}

\begin{example} \label{bsp_cont_alt}
As an example we consider the alternatives where
\begin{equation*}
h(x)=h_{m,\theta}(x)=P_m^{\, d}(\theta^\top x),\quad x\in\Sp,\;m\ge1,
\end{equation*}
is the Legendre polynomial of degree $m$ and $\theta\in\Sp$ is fixed. Note that $x\mapsto P_m^{\, d}(\theta^\top x)$ is a spherical harmonic function of degree $m$ such that the orthogonality property of spherical harmonics, see \cite{G:1996}, Section 3.2, implies
\begin{equation*}
\int_{\Sp} h(x)\mbox{d}\sigma(x)=\int_{\Sp} P_m^{\, d}(\theta^\top x)P_0^{\, d}(\theta^\top x)\; \mbox{d}\sigma(x)=0,
\end{equation*} since $P_0^{\, d}(\theta^\top\cdot)=1$ is the spherical harmonic of degree 0. If $X$ follows the law given by the density $f_n$ we have for any orthogonal $d\times d$-matrix $A$ with $A\theta=\theta$ that the distribution of $AX$ is the same as the distribution of $X$, hence these types of alternatives are rotationally symmetric about $\theta$. An application of the Funk/Hecke-Theorem \ref{funk-hecke} shows for $b\in\Sp$
\begin{eqnarray*}
S^*_\beta(b)&=&\frac1{|\Sp|}\int_{\Sp}S_\beta(u,b)h(u)\;\mbox{d}\sigma(u)\\
&=&\frac1{|\Sp|}\int_{\Sp}\left((b^\top u)^\beta-\psi_d(\beta)\right)P_m^{\, d}(\theta^\top u)\;\mbox{d}\sigma(u)\\
&=&\lambda_d(\beta,m)P_m^{\, d}(\theta^\top b),
\end{eqnarray*}
where
\begin{equation*}
\lambda_d(\beta,m)=\frac{|\Sp[d-2]|}{|\Sp|}\int_{-1}^1P_m^{\, d}(t) \left(t^\beta-\psi_d(\beta)\right)(1-t^2)^{\frac{d-3}{2}}\mbox{d}t.
\end{equation*}
It follows with \eqref{potenzen_legendre} and Proposition \ref{legendre_orthogonal}
\begin{align*}
    \lambda_d(\beta, m)
    = \frac{|\Sp[d-2]|}{|\Sp|} \left[ \sum_{j=0}^\beta c_{j, \, d}(\beta) \langle P_m^{\, d}, P_j^{\, d} \rangle - \psi_d(\beta) \langle P_m^{\, d}, P_0^{\, d} \rangle \right]
    = \frac{c_{m, \, d}(\beta)}{\nu_d(m)},
\end{align*}
so that
\begin{equation} \label{shift_legendre}
    S_\beta^*(b) = \frac{c_{m, \, d}(\beta)}{\nu_d(m)} P_m^{\, d}(\skalarprodukt{\theta}{b}), \quad b \in \Sp.
\end{equation}
Note that $\lambda_d(\beta, m) = 0$, if $\beta + m$ is odd or $m > \beta$, because then the coefficients $c_{m, \, d}(\beta)$ in \eqref{potenzen_legendre} equal zero. Hence, in these cases we have the same asymptotic behaviour under contiguous alternatives as under the null hypothesis. We can conclude that the tests $T_{n,\, \beta}$ are not able to detect the alternatives for such a combination of $\beta$ and $m$. For the shift function we have $S_\beta^*(\theta) = c_{m, \, d}(\beta)/\nu_d(m)$, so that there is a non-negative shift in the limiting distribution under contiguous alternatives as long as we can show that the coefficients $c_{m, \, d}(\beta)$ are non-negative. We conjecture that this is indeed the case, as all examples after Proposition fulfill this property. That in turn means that, under the assumption of this conjecture, there is a positive shift if $\beta + m$ is even and $m \leq \beta$. Thus, $T_{n,\, \beta}$ is a family of testing procedures which is able to detect the contiguous alternatives. As indicated in \cite{CPV:2017}, Section 2, the famous von Mises--Fisher distribution (see \cite{MJ:2000}, Section 9.3, for a definition) with mean direction $\theta$ and concentration parameter $\kappa\ge0$ falls into a comparable class of contiguous alternatives. We expect to see matchable power performances of $T_{n,\, \beta}$ in the simulation study, see Section \ref{sec:Simu}.
\end{example}

\section{Consistency}
\label{sec:asy_fixed}
In this short section we consider spherical random vectors $U,U_1,\ldots,U_n$ with a distribution having a continuous density $f$ w.r.t. the spherical Lebesgue measure $\sigma$. We adopt the reasoning in \cite{BH:1991} to argue that the considered tests are consistent against a large class of alternatives. If for $\beta\in\N$ there is a unit vector (say) $b_0\in\Sp$ such that
\begin{equation*}
\zeta(b_0)=\left(\E(b_0^\top U)^\beta-\psi_d(\beta)\right)^2>0,
\end{equation*}
the strong law of large numbers shows
\begin{equation*}
\lim_{n\rightarrow\infty}\zeta_n(b_0)=\lim_{n\rightarrow\infty}\left(\frac1n\sum_{j=1}^n(b_0^\top U_j)^\beta-\psi_d(\beta)\right)^2=\zeta(b_0)\quad \mbox{a.s.}
\end{equation*}
and since $T_{n,\, \beta}/n\ge \zeta_n(b_0)$ we have
\begin{equation*}
\lim_{n\rightarrow\infty}T_{n,\, \beta}=\infty \quad \mbox{a.s.}
\end{equation*}
This reasoning shows that the tests $T_{n,\, \beta}$ are consistent against each such alternative. Nonetheless, as we have already seen in the last section, $T_{n, \, \beta}$ is not consistent against any arbitrary alternative class. For certain combinations of $\beta$ and $m$, the order of the Legendre polynomial, $T_{n, \, \beta}$ exhibits the same asymptotic behaviour under the alternatives as under the null hypothesis. Another indication for the inconsistency of $T_{n, \, \beta}$ can be seen in the case $\beta = 1$, which essentially concerns the Rayleigh test. The authors of \cite{PPV:2021} have shown that, in the rather general context of rotationally symmetric alternatives with a location and concentration parameter and a defining angular function, the Rayleigh test is blind against certain local alternatives. These local alternatives show polynomial decrease of the concentration parameter towards zero (hence yielding the null hypothesis) and the odd-order derivatives of their angular function vanish at zero. An example of such an alternative is the well-known Watson distribution.

\section{Bahadur efficiencies}
\label{sec:Bah}

In this section we present some interesting insights into the Bahadur asymptotic relative efficiencies (ARE) of the statistics $(T_{n, \, \beta})_{\beta \in \N}$. For an elaborate and comprehensive introduction to the concept of Bahadur efficiency we refer the reader to \cite{BAH:1967} and \cite{N:1995}.

We consider alternative classes whose defining density $f(\cdot \, | \, \kappa)$ w.r.t. $\sigma$ is parameterized through a non-negative number $\kappa \geq 0$, where the uniform distribution on $\Sp$ is only obtained for the limit case $\kappa \to 0^+$ in $L^1(\Sp, \text{d}\sigma)$, i.e.
\begin{equation} \label{density_L1_conv}
    \limkappa || f(\cdot \, | \, \kappa) - f(\cdot \, | \, 0) ||_{L^1} = \limkappa \int_{\Sp} \left| f(x \, | \, \kappa) - f(x \, | \, 0) \right| \, \text{d}\sigma(x) = 0.
\end{equation}
Hence, the testing problem can be reformulated as
\begin{equation} \label{null_hypothesis_kappa}
    H_0: \kappa = 0 \text{ against } H_1: \kappa > 0.
\end{equation}
In order to properly apply the Bahadur theory, we consider the family of equivalent test statistics $\widetilde{T}_{n, \, \beta} = \sqrt{T_{n, \, \beta}}$, $\beta \in \N$. In the following we mainly focus our attention to the local approximate and the local exact Bahadur ARE. For two statistics $T_n^{(1)}$ and $T_n^{(2)}$ these are defined as follows. \\
The local approximate Bahadur ARE is given by
\begin{equation*}
    \Lambda_{T^{(1)}, T^{(2)}}^{\text{a}} =
    \underset{\kappa \to0^+}{\lim} \frac{c_{T^{(1)}}^{\; a}(\kappa)}{c_{T^{(2)}}^{\; a}(\kappa)},
\end{equation*}
where $c_T^{\; a}$ denotes the approximate Bahadur slope of a statistic $T_n$, see \cite{N:1995}, page 10. \\
The local exact Bahadur ARE is defined by
\begin{equation*}
    \Lambda_{T^{(1)}, T^{(2)}}^{\text{ex}} =
    \underset{\kappa \to0^+}{\lim} \frac{c_{T^{(1)}}(\kappa)}{c_{T^{(2)}}(\kappa)},
\end{equation*}
where $c_T$ denotes the exact Bahadur slope of a statistic $T_n$, see \cite{N:1995}, Section 1.2. In many cases the approximate and exact Bahadur slopes coincide in the proximity of the null hypothesis, i.e. in the limit case. This can also be observed for $(\widetilde{T}_{n, \, \beta})_{\beta \in \N}$ in the next proposition.

\begin{prop} \label{T_n,beta_bahadur_steigung}
Let $\beta \in \N$. The approximate Bahadur slope of $\widetilde{T}_{n, \, \beta}$ is given by
\begin{equation*}
    c_{\widetilde{T}_\beta}^{\; a}(\kappa)
    = \frac{\spheremax \gamma_\kappa^2(b)}{\spheremax \rho_\beta(b,b)}
    = \frac{\spheremax \gamma_\kappa^2(b)}{\sum_{j=1}^\beta \lambda_j \nu_d(j)}, \quad \kappa > 0,
\end{equation*}
with the eigenvalues $\lambda_j$ from Proposition \ref{K_beta_properties}, $\nu_d(j)$ as in \eqref{eq:dimsp} and
\begin{equation*}
    \gamma_\kappa(b) = \E_\kappa (\skalarprodukt{b}{U})^{\, \beta} - \psi_d(\beta), \quad b \in \Sp,
\end{equation*}
where $U \sim f(\cdot \, | \, \kappa)$, $\kappa > 0$. Furthermore, the exact Bahadur slope of $\widetilde{T}_{n, \, \beta}$ is for sufficiently small $\kappa > 0$ given by
\begin{equation*}
    c_{\widetilde{T}_\beta}(\kappa)
    = \frac{\spheremax \gamma_\kappa^2(b)}{\sum_{j=1}^\beta \lambda_j \nu_d(j)} + o \left( \spheremax \gamma_\kappa^2(b) \right).
\end{equation*}
\end{prop}

The Bahadur slopes of $\widetilde{T}_{n, \, \beta}$ apparently coincide locally, so that we do not distinguish between them anymore. In the following, we consider some explicit alternative classes and determine the local Bahadur ARE of $\widetilde{T}_{n, \, \beta}$ w.r.t. $M_n = -2 \log(\Lambda_n)$, where $\Lambda_n$ is the likelihood-ratio test. It is well-known that the exact and approximate Bahadur slope of $M_n$ are the same and are given by
\begin{equation*}
    c_M^{\; a}(\kappa) = c_M(\kappa) = 2 \text{KL}(\kappa, 0), \quad \kappa > 0,
\end{equation*}
where
\begin{equation*}
    \text{KL}(\kappa, \kappa_0) = \E_\kappa \left[ \log \left( \frac{f(U \, | \, \kappa)}{f(U \, | \, \kappa_0} \right) \right], \quad \kappa, \kappa_0 \geq 0,
\end{equation*}
is the Kullback--Leibler information number for $U \sim f(\cdot \, | \, \kappa)$.
The proof for the following quite technical calculations can be found in Appendix \ref{sec:AppB}.

\begin{example} \label{bsp_bahadur_vMF}
A random vector $U$ with values in $\Sp$ has a von Mises--Fisher distribution
 $\text{vMF}(\theta, \kappa)$ with mean direction $\theta \in \Sp$ and concentration parameter $\kappa \geq 0$ if the density w.r.t. $\sigma$ is given by
\begin{equation} \label{von_mises_fisher_dichte}
    f(x \, | \, \kappa) = \frac{\left( \kappa/2 \right)^{d/2-1}}{2 \pi^{d/2} I_{\frac{d}{2}-1}(\kappa)} \exp(\kappa \skalarprodukt{x}{\theta}), \quad x \in \Sp,
\end{equation}
where $I_{\frac{d}{2}-1}$ is the modified Bessel function of the first kind and order $d/2-1$, see \eqref{bessel_def}. In case of the von Mises--Fisher alternative class $\text{vMF}(\theta, \kappa)$ with a fixed mean direction $\theta \in \Sp$ and $\kappa > 0$ we have
\begin{equation*}
    \lim_{\kappa \to0^+} \frac{\max_{b \in \Sp} \gamma_\kappa^2(b)}{2 \textsc{KL}(\kappa, 0)} = \lambda_1 \nu_d(1),
\end{equation*}
whereby the local Bahadur ARE is
\begin{equation*}
    \Lambda_{\widetilde{T}_\beta, M}^{\text{ex}}
    = \Lambda_{\widetilde{T}_\beta, M}^{\text{a}}
    = \underset{\kappa \to0^+}{\lim} \frac{c_{\widetilde{T}_\beta}^{\; a}(\kappa)}{c_M^{\; a}(\kappa)}
    = \frac{1}{\sum_{j=1}^\beta \lambda_j \nu_d(j)} \; \underset{\kappa \to0^+}{\lim} \frac{\max_{b \in \Sp} \gamma_\kappa^2(b)}{2 \text{KL}(\kappa, 0)}
    = \frac{\lambda_1 \nu_d(1)}{\sum_{j=1}^\beta \lambda_j \nu_d(j)}.
\end{equation*}
The special case of $\beta = 1$ yields the local asymptotic optimality of $\widetilde{T}_1$ in the Bahadur sense (see \cite{N:1995}, page 9, for this concept)
\begin{equation*}
    \Lambda_{\widetilde{T}_1, \text{M}}^{\text{ex}} = \Lambda_{\widetilde{T}_1, \text{M}}^{\text{a}} = 1.
\end{equation*}
This is not surprising since the Rayleigh test is exactly the likelihood-ratio test in the von Mises--Fisher model. On the contrary, if $\beta$ is even, then
\begin{equation*}
     \Lambda_{\widetilde{T}_\beta, \text{M}}^{\text{ex}} = \Lambda_{\widetilde{T}_\beta, \text{M}}^{\text{a}} = 0,
\end{equation*}
since the eigenvalue $\lambda_1$ equals zero in this case.
\end{example}

\begin{example} \label{bsp_bahadur_watson}
A random vector $U$ with values in $\Sp$ has a Watson distribution
 $\text{W}(\theta, \kappa)$ with mean direction $\theta \in \Sp$ and concentration parameter $\kappa \in \R$ if the density w.r.t. $\sigma$ is given by
\begin{equation} \label{watson_dichte}
    f(x \, | \, \kappa) = \frac{\Gamma(d/2)}{2 \pi^{d/2} M(1/2, d/2, \kappa)} \exp(\kappa (\skalarprodukt{x}{\theta})^2), \quad x \in \Sp,
\end{equation}
where $M(\cdot, \cdot, \cdot)$ is the Kummer function, see \ref{kummer_def}. In the following, we only consider the case $\kappa \geq 0$. In case of the Watson alternative class $\text{W}(\theta, \kappa)$ with a fixed mean direction $\theta \in \Sp$ and $\kappa > 0$ we have
\begin{equation*}
    \lim_{\kappa \to0^+} \frac{\max_{b \in \Sp} \gamma_\kappa^2(b)}{2 \textsc{KL}(\kappa, 0)} = \lambda_2 \nu_d(2).
\end{equation*}

\begin{table}[t] \label{table_bahadur}
\small
\centering
\begin{tabular}{>{\hspace{2mm}}c<{\hspace{2mm}}|r|rrrr>{\hspace{6mm}}c<{\hspace{2mm}}|r|rrrrr}
    \toprule[0.5mm]
    \midrule
        & \backslashbox{$\beta$}{$d$}
        & 2    & 3    & 5    & 10   &                   & \backslashbox{$\beta$}{$d$}                                                                                                                 & 2    & 3    & 5    & 10   \\
        \midrule
vMF     & 1     & 1.00 & 1.00 & 1.00 & 1.00 &   LP$_1$  & 1    & 1.00 & 1.00 & 1.00 & 1.00 \\
        & 3     & 0.90 & 0.84 & 0.77 & 0.70 &           & 3    & 0.90 & 0.84 & 0.77 & 0.70 \\
        & 5     & 0.79 & 0.67 & 0.54 & 0.41 &           & 5    & 0.79 & 0.67 & 0.54 & 0.41 \\
        \midrule
W       & 2     & 1.00 & 1.00 & 1.00 & 1.00 &   LP$_2$  & 2     & 1.00 & 1.00 & 1.00 & 1.00 \\
        & 4     & 0.94 & 0.92 & 0.89 & 0.84 &           & 4     & 0.94 & 0.92 & 0.89 & 0.84 \\
        & 6     & 0.86 & 0.80 & 0.72 & 0.61 &           & 6     & 0.86 & 0.80 & 0.72 & 0.61 \\
        \midrule
LP$_3$  & 3     & 0.10 & 0.16 & 0.23 & 0.30 &   LP$_4$  & 4     & 0.06 & 0.08 & 0.11 & 0.16 \\
        & 5     & 0.20 & 0.31 & 0.43 & 0.54 &           & 6     & 0.14 & 0.19 & 0.27 & 0.37 \\
        \midrule
LP$_5$  & 5     & 0.01 & 0.02 & 0.03 & 0.06 &   LP$_6$  & 6     & 0.004 & 0.01 & 0.01 & 0.02 \\
    \midrule
    \bottomrule[0.5mm]
\end{tabular}
\caption{Non-trivial local Bahadur ARE of $\widetilde{T}_{n,\, \beta}$ w.r.t. $M_n$ for the alternative classes of the examples \ref{bsp_bahadur_vMF} to \ref{bsp_bahadur_LP} with dimension $d \in \{ 2, 3, 5, 10 \}$ and order $m \in \{ 1, \ldots, 6 \}$ of the LP alternative class.}
\end{table}
\noindent
Thus the local Bahadur ARE equals
\begin{equation*}
    \Lambda_{\widetilde{T}_\beta, M}^{\text{ex}}
    = \Lambda_{\widetilde{T}_\beta, M}^{\text{a}}
    = \frac{\lambda_2 \nu_d(2)}{\sum_{j=1}^\beta \lambda_j \nu_d(j)}.
\end{equation*}
This time the special case of $\beta = 2$ yields the local asymptotic optimality of $\widetilde{T}_2$ in the Bahadur sense
\begin{equation*}
    \Lambda_{\widetilde{T}_2, \text{M}}^{\text{ex}} = \Lambda_{\widetilde{T}_2, \text{M}}^{\text{a}} = 1,
\end{equation*}
whereas if $\beta$ is odd, then
\begin{equation*}
     \Lambda_{\widetilde{T}_\beta, \text{M}}^{\text{ex}} = \Lambda_{\widetilde{T}_\beta, \text{M}}^{\text{a}} = 0.
\end{equation*}
\end{example}

\begin{example} \label{bsp_bahadur_LP}
In concordance with Example \ref{bsp_cont_alt} we shall define the alternative class
$\text{LP}_m(\theta, \kappa)$ of order $m \in \N$ with direction $\theta \in \Sp$ and $\kappa \in [0,1]$. LP stands for Legendre polynomial in this context. Let this class be given by the density
\begin{equation} \label{LP_dichte}
    f(x \, | \, \kappa) = \frac{1}{|\Sp|} \left( 1 + \kappa P_m^{\, d}(\skalarprodukt{\theta}{x}) \right), \quad x \in \Sp.
\end{equation}
For fixed order $m$ and direction $\theta$ we have
\begin{equation*}
    \lim_{\kappa \to0^+} \frac{\max_{b \in \Sp} \gamma_\kappa^2(b)}{2 \textsc{KL}(\kappa, 0)} = \lambda_m \nu_d(m),
\end{equation*}
so that the local Bahadur ARE equals
\begin{equation*}
    \Lambda_{\widetilde{T}_\beta, \text{M}}^{\text{ex}} = \Lambda_{\widetilde{T}_\beta, \text{M}}^{\text{a}} = \frac{\lambda_m \nu_d(m)}{\sum_{j=1}^\beta \lambda_j \nu_d(j)}.
\end{equation*}
We obtain non-trivial local Bahadur AREs only for combinations of $\beta$ and $m$, where $m \leq \beta$ and $\beta + m$ is even. In particular, the special case $\beta = m = 1$ or $\beta = m = 2$ gives the local asymptotic optimality of $\widetilde{T}_1$ or $\widetilde{T}_2$, respectively, in the Bahadur sense.
\end{example}

\section{Simulations}
\label{sec:Simu}

We present a competitive Monte-Carlo simulation study, that was implemented and performed in the statistical computing environment \texttt{R}, see \cite{rco:2019}. The maximum on the hypersphere in (\ref{eq:test}) cannot be calculated analytically, and therefore one has to approximate it with a computationally fast method. We suggest to use a uniform random cover of the hypersphere: Simulate a large number $m$ of uniformly distributed points on $\Sp$, $B_1,\ldots,B_m$ (say), evaluate the so chosen centered and squared projections $\left(\frac1n\sum_{j=1}^n(B_k^\top U_j)^{\, \beta}-\psi_d(\beta)\right)^2$, $k=1,\ldots,m$, and approximate the maximum value in (\ref{eq:test}) by the discrete maximum over all $k$. Critical values for $T_{n,\, \beta}$ under $H_0$ have been simulated with 20000 replications and a random cover of $m=5000$ points for $d=2,3$ and with 20000 replications and $m=20000$ points for $d=5,10$, see Table \ref{tab:table_crit_values}.

The critical values in the rows in Table \ref{tab:table_crit_values} denoted by "$\infty$" and "$\infty^*$" represent approximations of the limit random element $\max_{b \in \Sp} Z^2_\beta(b)$ in Corollary \ref{cor:H0} via two methods. The first method, which corresponds to the rows with "$\infty$", simulates the same random cover of the sphere as above, and it considers a large number (say) $\ell$ of random variables $Z_j=\max(X_j^2)$, $j=1,\ldots,\ell$, with iid. $X_j\sim \mbox{N}_m(0,\Sigma_\beta)$, where $\mbox{N}_m$ is the $m$-variate normal distribution and  $\Sigma_\beta=\left(\rho_\beta(B_{k_1},B_{k_2})\right)_{k_1,k_2\in\{1,\ldots,m\}}$ is a singular $m\times m$-covariance matrix for $m\ge d$ and $\rho_\beta$ is the covariance kernel in (\ref{eq:covk}) for which we have already summarized explicit formulas in Remark \ref{rem:rho_beta_explicit}. Here $x^2$ is shorthand for the vector of squared components of $x$. Next, we calculate the empirical $95\%$ quantile of $Z_1,\ldots,Z_\ell$, where each approximation was simulated with $\ell=100000$ and $m=1000$ for $d=2,3$ as well as $\ell=10000$ and $m=5000$ for $d=5,10$.
The second method utilizes the alternative representation of the Gaussian process from Proposition \ref{pro:AltD}. For this purpose, we need orthonormal bases of the spaces $\harmonicspace$ for $k=0,\ldots,\beta$ and the corresponding eigenvalues $\lambda_k$. We have already presented an explicit list of these eigenvalues for $\beta=1,\ldots,6$ at the end of section \ref{sec:asy_null}. In each replication step of the Monte-Carlo simulation we generate  an array $N_{k,j}$, $j \in \{ 1,\ldots,\nu_d(k) \}$ of independent unit normal random variables, cover $\Sp$, once again, with $m$ uniformly distributed points $B_1,\ldots,B_m$ and calculate $Y_i = \sqrt{|\Sp|} \sum_{k=0}^\beta \sqrt{\lambda_k} \sum_{j=1}^{\nu_d(k)} \phi_{k,j}(b) N_{k,j}, i = 1,\ldots,m$.
Repeating this step for the number of set replications $\ell$ yields an approximation of the limit distribution of $T_{n, \, \beta}$ in the same fashion as before. However, so far there is no library with a stable implementation of orthonormal spherical harmonics in higher dimensions and orders, which is why we restricted the simulation with this method to the case of $d=2$. We used the package \texttt{HFT.m} in \texttt{Mathematica} in order to implement an orthonormal basis in \texttt{R}. Each approximation was performed with $\ell = 20000$ and $m=2500$.

Table \ref{tab:table_crit_values} shows empirical and approximated 0.95 quantiles of $T_{n, \, \beta}$ under the null hypothesis. It is interesting to compare the approximated critical values with the 0.95 quantiles of $\chi^2_d/d$ for $T_{n,1}$ (respectively the Rayleigh test), which are $2.996$ for $d=2$, $2.605$ for $d=3$, $2.214$ for $d=5$, and $1.830$ for $d=10$. Evidently, the approximation with the random covering and the limiting process is close to the theoretical asymptotic critical values for the dimensions $d=2,3,5$, but it gets less accurate for dimensions greater than 5. This behaviour can be explained by the curse of dimensionality, indicating that more points on the unit sphere in the random covering should be considered to increase the accuracy of the approximation. A similar behaviour can be observed for $\beta = 2$ and $2(d-1)/(d^2(d+2)) \, \chi^2_{\nu_d(2)}$ with $\nu_d(2) = (d-1)(d+2)/2$, where, of course, the latter random variable is only an upper bound for the limit distribution of $T_{n, 2}$ as we have seen in Remark \ref{rem:AltD_explicit_cases}. Nevertheless, the numerical results support the theoretical findings of Section \ref{sec:asy_null}.

\begin{table}
\small
\centering
\begin{tabular}{>{\hspace{2mm}}c<{\hspace{2mm}}|r|rrrrrr}
    \toprule[0.5mm]
    \midrule
                    & \backslashbox{$n$}{$\beta$}
                                    & 1     & 2     & 3     & 4     & 5     & 6     \\
                    \midrule
    $d=2$           & 20            & 2.906 & 0.746 & 2.037 & 0.917 & 1.761 & 0.941 \\
                    & 50            & 2.968 & 0.730 & 2.051 & 0.914 & 1.734 & 0.934 \\
                    & 100           & 3.004 & 0.752 & 2.031 & 0.906 & 1.730 & 0.936 \\
                    & 500           & 3.033 & 0.735 & 2.081 & 0.923 & 1.724 & 0.939 \\
                    & $\infty$      & 2.986 & 0.750 & 2.050 & 0.924 & 1.729 & 0.944 \\
                    & $\infty^*$    & 2.944 & 0.753 & 2.047 & 0.923 & 1.735 & 0.945 \\
                    \midrule
    $d=3$           & 20            & 2.582 & 0.864 & 1.306 & 0.794 & 0.994 & 0.738 \\
                    & 50            & 2.578 & 0.875 & 1.319 & 0.751 & 0.932 & 0.666 \\
                    & 100           & 2.562 & 0.881 & 1.293 & 0.734 & 0.922 & 0.632 \\
                    & 500           & 2.585 & 0.869 & 1.298 & 0.733 & 0.901 & 0.614 \\
                    & $\infty$      & 2.605 & 0.866 & 1.283 & 0.724 & 0.895 & 0.606 \\
                    \midrule
    $d=5$           & 20            & 2.183 & 0.736 & 0.729 & 0.519 & 0.444 & 0.399 \\
                    & 50            & 2.157 & 0.695 & 0.674 & 0.451 & 0.378 & 0.318 \\
                    & 100           & 2.203 & 0.685 & 0.654 & 0.407 & 0.352 & 0.280 \\
                    & 500           & 2.173 & 0.663 & 0.632 & 0.363 & 0.324 & 0.232 \\
                    & $\infty$      & 2.161 & 0.638 & 0.620 & 0.340 & 0.306 & 0.206 \\
                    \midrule
    $d=10$          & 20            & 1.567 & 0.419 & 0.252 & 0.165 & 0.107 & 0.090 \\
                    & 50            & 1.580 & 0.368 & 0.209 & 0.124 & 0.074 & 0.060 \\
                    & 100           & 1.586 & 0.342 & 0.190 & 0.105 & 0.059 & 0.046 \\
                    & 500           & 1.605 & 0.314 & 0.173 & 0.080 & 0.045 & 0.030 \\
                    & $\infty$      & 1.485 & 0.277 & 0.153 & 0.062 & 0.036 & 0.019 \\
                    \midrule
    \bottomrule[0.5mm]
\end{tabular}
\caption{Empirical and approximated 0.95 quantiles of $T_{n, \, \beta}$ under $H_0$ for dimensions $d \in \{2, 3, 5, 10\}$, sample sizes $n \in \{ 20, 50, 100, 500 \}$ and $\beta \in \{ 1, \ldots, 6 \}$. Here, $\infty$ denotes the approximation of the limit distribution of $T_{n, \, \beta}$ via covariance kernel and $\infty^*$ the approximation via spherical harmonics.} \label{tab:table_crit_values}
\end{table}

We consider testing for uniformity on the unit circle $\mathcal{S}^1$, on the unit sphere $\mathcal{S}^2$ and on the hypersphere $\mathcal{S}^5$, and we divide the presentation of the simulation study into two parts, since different competing tests are considered in these cases.

Generating uniformly distributed random numbers on $\Sp$ can be done efficiently, since for a random vector $N \sim \mathcal{N}_{d}(0, I_d)$, where $\mathcal{N}_{d}$ stands for the $d$-variate normal distribution on $\mathbb{R}^d$, we have
\begin{equation*}
    \frac{N}{|| N ||} \sim \unifsphere.
\end{equation*}
This property is merely a consequence of the rotational invariance of $N$ and the fact that the uniform distribution is the only rotationally invariant distribution on $\Sp$. We consider the following alternatives to the uniform distribution:
\begin{itemize}
    \item von Mises--Fisher distribution: \\
    This alternative class was already introduced in Example \ref{bsp_bahadur_vMF}. The density is given by
    \begin{equation*}
        \Sp \ni x \mapsto \frac{\left( \kappa/2 \right)^{d/2-1}}{2 \pi^{d/2} I_{\frac{d}{2}-1}(\kappa)} \exp(\kappa x^\top \theta),
    \end{equation*}
    where $I_{\frac{d}{2}-1}$ is the modified Bessel function of the first kind and order $d/2-1$. This class is denoted with $\text{vMF}(\theta, \kappa)$.
    \item Mix of von Mises--Fisher distributions with two centers: \\
    Let $U$ be uniformly distributed on $(0,1)$, $p \in (0,1)$ and $Y_i \sim \text{vMF}(\theta_i, \kappa_i)$ with corresponding location and concentration parameters for $i=1,2$. Let  $U$, $Y_1$ and $Y_2$ be stochastically independent. Then we generate a random sample $X$ according to
    \begin{equation*}
        X = Y_1 \mathds{1}_{\{ U < p \} } + Y_2 \mathds{1}_{\{ U \geq p \} }.
    \end{equation*}
    We denote this alternative class with $\text{Mix-vMF}(p, \theta_1, \theta_2, \kappa_1, \kappa_2)$.
    \item Mix of von Mises--Fisher distributions with three centers: \\
     Let $U$ be uniformly distributed on $(0,1)$, $p \in (0,1/2)$ and $Y_i \sim \text{vMF}(\theta_i, \kappa_i)$ with corresponding location and concentration parameters for $i=1,2,3$. Let $U$, $Y_1$, $Y_2$ and $Y_3$ be stochastically independent. Then we generate a random sample $X$ according to
    \begin{equation*}
        X = Y_1 \mathds{1}_{\{ U < p \} } + Y_2 \mathds{1}_{\{ p \leq U < 2p \} } + Y_3 \mathds{1}_{\{ U \geq 2p \} }.
    \end{equation*}
    We denote this alternative class with $\text{Mix-vMF}(p, \theta_1, \theta_2, \theta_3, \kappa_1, \kappa_2, \kappa_3)$.
    \item Bingham distribution: \\
    The density
    \begin{equation*}
        \Sp \ni x \mapsto \frac{1}{c(d,A)} \exp(x^\top A x)
    \end{equation*}
    with a symmetric $d \times d$-matrix A and a normalizing constant $c(d, A)$ yields the Bingham model. In the following this alternative class is denoted with $\text{Bing(A)}$.
    \item Legendre polynomial distribution: \\
    We defined the Legendre polynomial alternative class $\text{LP}_m(\theta, \kappa)$ with order $m \in \N$, direction $\theta \in \Sp$ and $\kappa \in [0,1]$ in Example \ref{bsp_bahadur_LP}.  This class is given by the density
    \begin{equation*}
        f(x \, | \, \kappa) = \frac{1}{|\Sp|} \left( 1 + \kappa P_m^{\, d}(\skalarprodukt{\theta}{x}) \right), \quad x \in \Sp.
    \end{equation*}
    Due to $f(x \, | \, \kappa) \leq (1 + \kappa)|\Sp|^{-1}$, $x \in \Sp$,
    the acceptance rejection algorithm gives us a simple method to generate random numbers of this alternative class, see \cite{KE:2013} for a description of this algorithm. Simulations on $\Sp[1]$ with $\kappa = 1$ revealed that the order $m$ dictates the formation of exactly $m$ clusters, which spread out equidistantly from the direction $\theta$ over the unit circle. So, numerically, this class generalizes uni- and multipolar distributions, among them the von Mises--Fisher and Watson distribution.
\end{itemize}

Further details and properties of the presented distributions may be found in \cite{MJ:2000}, Section 9.3 and 10.3. Throughout the whole chapter the nominal level of significance is set to $0.05$. Empirical critical values for the competing statistics have been computed with a replication number of $l=20000$ as well. The empirical powers in all tables are based on $5000$ replications of a testing decision and are, in reality, rejection frequencies. Next, we specify the considered alternatives further due to limited space in the tables.
Let
\begin{align*}
    &\theta_1 = (1, 0, \ldots, 0)^\top, \quad \theta_2 = (-1, \ldots, -1)^\top, \quad
    \theta_3 = (-1, 1, \ldots, 1)^\top, \\[0.25cm]
    &A_1 = \text{diag}(1,2, \ldots, d), \quad A_2 = \text{diag}(-d, 0, \ldots, 0, d).
\end{align*}
If directions do not lie on $\Sp$, then the functions in \texttt{R} use the orthogonal projection onto $\Sp$ in order to generate random numbers of that alternative. Then let
\begin{itemize}
    \item vMF$_1(\kappa) = \text{vMF}(\theta_1, \kappa)$,
    \item Mix-vMF$_1(p) = \text{Mix-vMF}(p, -\theta_1, \theta_1, 1, 1)$ and  \\[0.1cm]        Mix-vMF$_2(p) = \text{Mix-vMF}(p, -\theta_1, \theta_1, 1, 4)$,
    \item Mix-vMF$_3(p) = \text{Mix-vMF}(p, \theta_2, \theta_3, \theta_1, 2, 3, 3)$ and \\[0.1cm] Mix-vMF$_4(p) = \text{Mix-vMF}(p, \theta_2, \theta_3, \theta_1, 2, 3, 4)$,
    \item Bing$_1(\kappa) =\text{Bing}(\kappa A_1)$ and Bing$_2(\kappa) =\text{Bing}(\kappa A_2)$, and
    \item LP$_m(\kappa) = \text{LP}_m(\theta_1, \kappa)$.
\end{itemize}

\subsection{Unit circle \texorpdfstring{$\Sp[1]$}{Lg}}

In this subsection the sample of random vectors on $\Sp$ is expressed in polar coordinates, i.e. for $U_i \in \Sp$, $i=1,\ldots,n$, we write $U_i = (\cos(\vartheta_i), \sin(\vartheta_i))^\top$ with a random angle $\vartheta_i \in [0,2\pi)$. We consider the subsequent competing test procedures on the unit sphere. The references next to them refer to further information about the respective statistic. Note that large values are significant for all presented test statistics except for the statistic by Cuesta-Albertos et al. in \cite{CCF:2009}.
\begin{itemize}
    \item Kuiper test, \cite{K:1960}: \\
    Let $F_n(\vartheta) = \frac{1}{n} \sum_{i=1}^n \mathds{1}_{\{ \vartheta_i \leq \vartheta \} }$ be the empirical distribution function based on the angles $\vartheta_1,\ldots,\vartheta_n$ and $F(\vartheta) = \frac{\vartheta}{2 \pi} \mathds{1}_{[0,2 \pi)}(\vartheta) + \mathds{1}_{[2 \pi, \infty)}(\vartheta)$ the distribution function of the uniform distribution on $\Sp[1]$ for $\vartheta \in \R$. The Kuiper test considers the quantities
    \begin{align*}
        &D_n^+
        = \sqrt{n} \underset{\vartheta \in [0,2 \pi)}{\sup} \left( F_n(\vartheta) - F(\vartheta) \right)
        = \sqrt{n} \underset{i=1,\ldots,n}{\max} \left( \frac{i}{n} - X_i \right), \\
        &D_n^-
        = \sqrt{n} \underset{\vartheta \in [0,2 \pi)}{\sup} \left( F(\vartheta) - F_n(\vartheta) \right)
        = \sqrt{n} \underset{i=1,\ldots,n}{\max} \left( X_i - \frac{i-1}{n} \right),
    \end{align*}
    with $X_i = \frac{\vartheta_{(i)}}{2 \pi}$ for $i=1,\ldots,n$, where $\vartheta_{(1)},\ldots,\vartheta_{(n)}$ is the ordered sample of the random angles. Then the test utilizes the statistic
    \begin{equation*}
        V_n = D_n^+ + D_n^-.
    \end{equation*}
    \item Watson test, \cite{W:1961}: \\
    According to the previous definitions, the Watson test is given by
    \begin{align*}
        U_n^2
        &= n \int_0^{2 \pi} \left[ F_n(\vartheta) - F(\vartheta) - \int_0^{2 \pi} F_n(\omega) - F(\omega) \; \text{d}F(\omega) \right]^2 \; \text{d}F(\vartheta) \\
        &= \sum_{i=1}^n \left[ \left( X_i - \frac{i-1/2}{n} \right) - \left( \Bar{X} - \frac{1}{2} \right) \right]^2 + \frac{1}{12n},
    \end{align*}
    where $\Bar{X} = \frac{1}{n} \sum_{i=1}^n X_i$.
    \item Ajne test, \cite{AJ:1968}: \\
    The Ajne test on $\Sp[1]$ is based on the statistic
    \begin{equation*}
        A_n = \frac{1}{2 \pi n} \int_0^{2 \pi} \left[ N(\alpha) - \frac{n}{2} \right]^2 \; \text{d}\alpha = \frac{n}{4} - \frac{1}{n \pi} \sum_{1 \leq i < j \leq n} d_c(\vartheta_i, \vartheta_j).
    \end{equation*}
    Here,
    \begin{equation*}
        N(\alpha) = \# \{ \vartheta_1,\ldots,\vartheta_n \, | \, d_c(\alpha,\vartheta_i) < \pi/2, \; i = 1,\ldots,n \}, \quad \alpha \in [0,2 \pi),
    \end{equation*}
    with $d_c(\alpha,\vartheta) = \min(|\alpha - \vartheta|, \, 2 \pi - |\alpha - \vartheta|)$, $\vartheta \in [0,2 \pi)$.
    \item Rayleigh test, \cite{R:1919}: \\
    The classical Rayleigh test on $\Sp[1]$ is given by
    \begin{equation*}
        R_n = 2 n \| \Bar{U} \|^2 = \frac{2}{n} \left[ \left( \sum_{i=1}^n \cos(\vartheta_i) \right)^2 + \left( \sum_{i=1}^n \sin(\vartheta_i) \right)^2 \right].
    \end{equation*}
    However, for the simulation we will use the modification as in \cite{MJ:2000}, Section 6.3.1, which permits an improved $\chi_2^2$ approximation, i.e., we use the statistic
    \begin{equation*}
        R_n^{\text{ mod}} = \left( 1- \frac{1}{2n} \right) R_n + \frac{n \| \Bar{U} \|^4}{2}.
    \end{equation*}
    \item test of Cuesta-Albertos et al., \cite{CCF:2009}: \\
    The test by Cuesta-Albertos et al. is based on random projections of the sample $U_1,\ldots,U_n$. For this, one chooses, independently of $U_1,\ldots,U_n$, a direction $H \sim \unifsphere$ and considers the projected random variables $Y_1 = \skalarprodukt{U_1}{H}, \ldots, Y_n = \skalarprodukt{U_n}{H}$. In \cite{CCF:2009}, Theorem 2.2, it has been proved that the distribution of $Y_1$ characterizes, with probability 1, the distribution of $U_1$. To that effect, $H_0$ is almost surely equivalent to
    \begin{equation*}
        H_0^*: \skalarprodukt{U}{H} \sim F_1,
    \end{equation*}
    where $U$ is a random vector with values in $\Sp[1]$ and $F_1$ is the distribution of the random projection $Y_1 = \skalarprodukt{U_1}{H}$. In the case $d=2$
    \begin{equation*}
        F_1(y) = \left( 1 - \frac{1}{\pi}\arccos(y) \right)\mathds{1}_{[-1,1]}(y) + \mathds{1}_{(1, \infty)}(y), \quad y \in \R.
    \end{equation*}
    The test proceeds as follows:
    \begin{enumerate}[label=\roman*)]
        \item Choose $q \in \N$ random projections $H_1,\ldots,H_q \sim \unifsphere$.
        \item For $m = 1,\ldots,q$ calculate the p-values $p_m$ of the Kolmogorov-Smirnov statistics
        \begin{equation*}
            K_{n,m} = \underset{y \in [-1,1]}{\sup} | F_{n,m}(y) - F_1(y) |,
        \end{equation*}
        where $F_{n,m}$ is the empirical distribution function based on $\skalarprodukt{U_1}{H_m},\ldots,\skalarprodukt{U_n}{H_m}$.
        \item Reject $H_0^*$, and thus $H_0$, for small values of the aggregated test statistic
        \begin{equation*}
            CA_n^q = \min\{p_1,\ldots,p_q\}.
        \end{equation*}
    \end{enumerate}
    This test procedure is a modification of the test that only chooses a single random direction in order to mitigate poor power due to a possibly unfavorable choice of said random direction.
\end{itemize}

\begin{table}[H]
    \footnotesize
    \centering
    \begin{tabular}{l|rrrrrrrrrrr}
        \toprule[0.5mm]
        \midrule
        \backslashbox{\scriptsize{Alternative}}{\scriptsize{Test}} & $T_{n,1}$ & $T_{n,2}$ & $T_{n,3}$ & $T_{n,4}$ & $T_{n,5}$ & $T_{n,6}$ & $K_n$ & $U_n^2$ & $A_n$ & $R_n^{\text{ mod}}$ & $CA_n^{25}$ \\
        \midrule
        $\unifsphere$ & 5 & 4 & 5 & 5 & 5 & 6 & 5 & 5 & 5 & 5 & 5 \\
        vMF$_1$(0.05) & \textbf{6} & 5 & \textbf{6} & 5 & \textbf{6} & 5 & 5 & 5 & 5 & 5 & 5 \\
        vMF$_1$(0.1) & \textbf{9} & 5 & \textbf{9} & 5 & 8 & 6 & 7 & 7 & \textbf{9} & \textbf{9} & 8 \\
        vMF$_1$(0.25) & 32 & 5 & 31 & 6 & 27 & 5 & 29 & 32 & 32 & \textbf{33} & 29 \\
        vMF$_1$(0.5) & 88 & 6 & 86 & 6 & 82 & 7 & 84 & 88 & 88 & \textbf{89} & 84 \\
        vMF$_1$(0.75) & \textbf{100} & 12 & \textbf{100} & 12 & 99 & 11 & 99 & \textbf{100} & \textbf{100} & \textbf{100} & 99 \\
        vMF$_1$(1) & \textbf{100} & 25 & \textbf{100} & 25 & \textbf{100} & 23 & \textbf{100} & \textbf{100} & \textbf{100} & \textbf{100} & \textbf{100} \\
        vMF$_1$(2) & \textbf{100} & 98 & \textbf{100} & 98 & \textbf{100} & 98 & \textbf{100} & \textbf{100} & \textbf{100} & \textbf{100} & \textbf{100} \\
        Mix-vMF$_1$(0.25) & 81 & 25 & 80 & 26 & 75 & 24 & 80 & \textbf{82} & 81 & 81 & 79 \\
        Mix-vMF$_1$(0.5) & 5 & \textbf{26} & 6 & 25 & 5 & 25 & 8 & 7 & 5 & 5 & 8 \\
        Mix-vMF$_2$(0.5) & 72 & \textbf{100} & 80 & 99 & 82 & 99 & 97 & 96 & 74 & 73 & 97 \\
        Mix-vMF$_2$(0.75) & 29 & \textbf{82} & 30 & 81 & 29 & 81 & 56 & 52 & 32 & 31 & 57 \\
        Mix-vMF$_3$(0.25) & 39 & \textbf{85} & 54 & \textbf{85} & 59 & 83 & 73 & 66 & 45 & 40 & 73 \\
        Mix-vMF$_4$(0.33) & 14 & 69 & 23 & \textbf{70} & 32 & 67 & 39 & 34 & 17 & 14 & 39 \\
        Bing$_1$(0.25) & 5 & \textbf{12} & 5 & 11 & 5 & 11 & 6 & 6 & 5 & 5 & 6 \\
        Bing$_1$(0.5) & 5 & \textbf{33} & 6 & 32 & 6 & 30 & 10 & 8 & 5 & 5 & 10 \\
        Bing$_1$(1) & 5 & \textbf{88} & 6 & 87 & 6 & 86 & 34 & 28 & 5 & 6 & 34 \\
        Bing$_2$(0.1) & 5 & \textbf{22} & 6 & \textbf{22} & 5 & \textbf{22} & 9 & 7 & 5 & 5 & 8 \\
        Bing$_2$(0.25) & 5 & \textbf{88} & 6 & \textbf{88} & 6 & 87 & 34 & 28 & 6 & 5 & 35 \\
        LP$_3$(0.1) & 5 & 5 & \textbf{6} & \textbf{6} & \textbf{6} & 5 & 5 & 4 & 5 & 5 & 5 \\
        LP$_4$(0.1) & 5 & \textbf{6} & 5 & 5 & 5 & \textbf{6} & 5 & 5 & 5 & 5 & 5 \\
        LP$_3$(0.5) & 5 & 5 & 25 & 6 & \textbf{43} & 5 & 18 & 10 & 11 & 4 & 18 \\
        LP$_4$(0.5) & 5 & 5 & 5 & 18 & 6 & \textbf{31} & 12 & 8 & 5 & 5 & 13 \\
        LP$_3$(1) & 5 & 6 & 90 & 5 & \textbf{100} & 5 & 75 & 69 & 75 & 5 & 74 \\
        LP$_4$(1) & 5 & 6 & 6 & 58 & 6 & \textbf{96} & 47 & 25 & 5 & 6 & 46 \\
        \midrule
        \bottomrule[0.5mm]
    \end{tabular}
    \caption{Empirical power for $n=100$ and $d=2$.} \label{table_power_100_2}
\end{table}

Table \ref{table_power_100_2} presents the empirical power of all considered test statistics for a sample size of $n=100$. Numbers in bold indicate the highest power to a given alternative. One can immediately see that the significance level is maintained or as in the case of $T_{n, 6}$ only slightly exceeded. The classical tests and $T_{n, \, \beta}$ for $\beta$ odd perform better than other tests with the unipolar von Mises--Fisher alternative. In the case of the mixtures of von Mises--Fisher distributions does $T_{n, \, \beta}$ for $\beta$ even show higher power. The same is true for the Bingham alternative. In any case, $T_{n, \, 2}$ obviously performs best here.
With the LP$_m$ alternative class it is interesting to observe that for $m = 3$ the test statistics $T_{n, \, 3}$ and $T_{n, \, 5}$ and for $m = 4$ the test statistics $T_{n, \, 4}$ and $T_{n, \, 6}$ dominate. It seems as though a higher exponentiation via $\beta$, i.e. higher moments in the definition of $T_{n, \, \beta}$, yields better results with multipolar distributions as long as the number of clusters and $\beta$ share the same parity. This observation is conform with the theoretical findings of Chapter \ref{sec:asy_fixed} and \ref{sec:Bah}.

\subsection{Unit sphere \texorpdfstring{$\Sp[2]$}{Lg} and unit hypersphere \texorpdfstring{$\Sp[5]$}{Lg}}

The Ajne and Rayleigh test statistic can be extended to $\Sp$, $d \geq 2$, in a straightforward way, since they are special cases of the fruitful Sobolev test class, see \cite{GV:2018}, Section 3. In addition, we consider three other uniformity tests. Except for the test by Cuesta-Albertos et al. in \cite{CCF:2009}, is $H_0$, once more, rejected for large values of the respective statistic.
\begingroup
\allowdisplaybreaks
\begin{itemize}
    \item Ajne test: \\
    The higher-dimensional extension of the Ajne test happens via
    \begin{equation*}
        A_n = \frac{\Gamma(d/2-1)}{2 n \pi^{d/2} } \int_{\Sp} \left[ N(\omega) - \frac{n}{2} \right]^2 \, \text{d}\sigma(\omega) = \frac{n}{4} - \frac{1}{n \pi} \sum_{1 \leq i < j \leq n} \arccos(\skalarprodukt{U_i}{U_j}),
    \end{equation*}
    where
    \begin{equation*}
        N(\omega) = \# \{ U_1,\ldots,U_n \, | \, \skalarprodukt{\omega}{U_i} \geq 0, \; i = 1,\ldots,n \}, \quad \omega \in \Sp.
    \end{equation*}
    \item Rayleigh test: \\
    The Rayleigh test on $\Sp$ is given by
    \begin{equation*}
        R_n = d n \| \Bar{U} \|^2
    \end{equation*}
    and the modification in \cite{MJ:2000}, Section 10.4.1, is
    \begin{equation*}
        R_n^{\text{ mod}} = \left( 1- \frac{1}{2n} \right) R_n + \frac{1}{2n(d+2)} R_n^2.
    \end{equation*}
    \item Bingham test, \cite{B:1974}: \\
    The Bingham test uses the empirical covariance matrix of the sample $S = \frac{1}{n} \standardsumme U_j U_j^\top$ and is based on the quantity
    \begin{equation*}
        B_n = \frac{nd(d+2)}{2} \left( \text{trace}(S^2) - \frac{1}{d} \right).
    \end{equation*}
    \item Gin\'{e}'s Sobolev test, \cite{G:1975}: \\
    We consider Giné's $G_n$ test, which is given by the statistic
    \begin{equation*}
        G_n = \frac{n}{2} - \frac{(d-1)\Gamma(d/2-1)^2}{2 n \Gamma(d/2)^2 } \sum_{1 \leq i < j \leq n} \sin(\arccos(\skalarprodukt{U_i}{U_j})).
    \end{equation*}
    \item test of Cuesta-Albertos et al., \cite{CCF:2009}: \\
    The main idea and test procedure is the same as in the case $d=2$. The only quantity that changes in higher dimensions is the distribution function $F_1$ of the random projection. For the general case, this distribution function may be obtained from the density of the projection as in \cite{MJ:2000}, Section 9.3.1, according to
    \begin{equation} \label{verteilungsfunktion_projektion}
        \begin{aligned}
        F_{d-1}(y)
        &= B \left( \frac{1}{2}, \frac{d-1}{2} \right)^{-1} \int_{-1}^y (1-t^2)^{\frac{d-3}{2}} \, \text{d}t \, \mathds{1}_{[-1,1]}(y) + \mathds{1}_{(1, \infty)}(y) \\[0.25cm]
        &= \frac{1}{2} \left[ 1 + \text{sign}(y) B_{y^2} \left( \frac{1}{2}, \frac{d-1}{2} \right) \right] \mathds{1}_{[-1,1]}(y) + \mathds{1}_{(1, \infty)}(y), \quad y \in \R,
        \end{aligned}
    \end{equation}
    where $B_x(a,b) = B (a,b)^{-1} \int_{0}^x t^{a-1} (1-t)^{b-1} \, \text{d}t$, $x \geq 0, a,b>0$ is the regularized, incomplete Beta function and sign$(\cdot)$ the usual sign function on $\R$.
    \item Cram\'{e}r-von Mises type test, \cite{GNC:2020}: \\
    Lastly, we present another test that is based on a projection approach and comes from Garc\'{i}a-Portugu\'{e}s et al., see \cite{GNC:2020}. It is, to a certain extent, the Cram\'{e}r-von Mises counterpart to the test by Cuesta-Albertos et al. and, thus, considers the expected value
    \begin{equation*}
        \text{CvM}_n = n \, \E_H \left[ \int_{-1}^1 | F_{n,H}(y) - F_{d-1}(y) |^2 \; \text{d}F_{d-1}(y) \right].
    \end{equation*}
    Here, $H \sim \unifsphere$, $F_{n,H}$ is the empirical distribution function based on $\skalarprodukt{U_1}{H},\ldots, \linebreak \skalarprodukt{U_n}{H}$, and $F_{d-1}$ is the distribution function in \eqref{verteilungsfunktion_projektion}. This test statistic can be written as a U-statistic for a practical implementation in \texttt{R} according to
    \begin{equation*}
        \text{CvM}_n = \frac{2}{n} \sum_{1 \leq i < j \leq n} \zeta_{d-1}(\arccos(\skalarprodukt{U_i}{U_j})) + \frac{3n-2}{6},
    \end{equation*}
    where for $\vartheta \in [0, \pi]$
    \begin{align*}
    \zeta_{d-1}(\vartheta)
    = \begin{cases}
        \begin{alignedat}{2}
            &  \frac{1}{2} + \frac{\vartheta}{2 \pi} \left( \frac{\vartheta}{2 \pi} -1 \right)                &,\quad& \textrm{for} \; d=2,\\[0.25cm]
	        &  \frac{1}{2} - \frac{1}{4} \sin \left( \frac{\vartheta}{2} \right)                            &,\quad& \textrm{for} \; d=3,\\[0.25cm]
	        &  \zeta_1(\vartheta) + \frac{1}{4 \pi^2} \left( (\pi -\vartheta) \tan \left( \frac{\vartheta}{2} \right) - 2 \sin^2 \left( \frac{\vartheta}{2} \right) \right)                                                             &,\quad& \textrm{for} \; d=4,\\[0.25cm]
	        \begin{split}
	            &- 4 \int_0^{\, \cos(\vartheta/2)} F_{d-1}(y)  F_{d-2} \left( \frac{y \tan(\vartheta/2)}{\sqrt{1-y^2}} \right)  \; \text{d}F_{d-1}(y) \\
	            &-\frac{3}{4} + \frac{\vartheta}{2 \pi} + 2 F_{d-1}^2 \left(\cos \left( \frac{\vartheta}{2} \right) \right)
	        \end{split}     &,\quad& \textrm{for} \; d \geq 5.
    \end{alignedat}
	\end{cases}
    \end{align*}
\end{itemize}
\endgroup
Tables \ref{table_power_100_3}, \ref{table_power_100_5} and \ref{table_power_100_10} show the empirical power for higher dimensions and for a sample size of $n=100$. The significance level is maintained in all cases. Here we can essentially observe similar patterns as in the case $d=2$; the power is merely lower and tends to decrease with increasing dimension. The tests $CA_n^{100}$ and $\text{CvM}_n$ show relatively high power for certain alternatives. We also note that the Bingham and Giné test perform better than other tests in the case of a Bingham alternative. However, it is especially striking that almost all tests fail to recognize the Legendre polynomial alternative class as the dimension grows.

\begin{table}[H]
\footnotesize
\centering
\begin{tabular}{l|rrrrrrrrrrrr}
    \toprule[0.5mm]
    \midrule
    \backslashbox{\scriptsize{Alternative}}{\scriptsize{Test}} & $T_{n,1}$ & $T_{n,2}$ & $T_{n,3}$ & $T_{n,4}$ & $T_{n,5}$ & $T_{n,6}$ & $A_n$ & $R_n^{\text{ mod}}$ & $B_n$ & $G_n$ & $CA_n^{100}$ & $\text{CvM}_n$ \\
    \midrule
    $\unifsphere$ & 5 & 5 & 5 & 5 & 5 & 5 & 4 & 5 & 5 & 5 & 5 & 5 \\
    vMF$_1$(0.05) & \textbf{7} & 5 & 5 & 6 & 5 & 5 & 4 & 4 & 5 & 5 & 4 & 5 \\
    vMF$_1$(0.1) & \textbf{8} & 5 & 7 & 6 & 7 & 5 & 6 & 7 & 5 & 5 & 5 & 6 \\
    vMF$_1$(0.25) & \textbf{21} & 5 & 18 & 5 & 15 & 6 & 18 & 19 & 4 & 5 & 16 & 19 \\
    vMF$_1$(0.5) & \textbf{68} & 6 & 61 & 6 & 53 & 6 & 65 & 67 & 5 & 5 & 59 & 64 \\
    vMF$_1$(0.75) & \textbf{96} & 8 & 93 & 8 & 90 & 7 & \textbf{96} & \textbf{96} & 8 & 7 & 94 & \textbf{96} \\
    vMF$_1$(1) & \textbf{100} & 15 & \textbf{100} & 14 & 99 & 14 & \textbf{100} & \textbf{100} & 15 & 14 & \textbf{100} & \textbf{100} \\
    vMF$_1$(2) & \textbf{100} & 93 & \textbf{100} & 91 & \textbf{100} & 87 & \textbf{100} & \textbf{100} & 92 & 92 & \textbf{100} & \textbf{100} \\
    Mix-vMF$_1$(0.25) & \textbf{61} & 14 & 55 & 15 & 51 & 14 & 59 & \textbf{61} & 15 & 14 & 57 & \textbf{61} \\
    Mix-vMF$_1$(0.5) & 6 & \textbf{15} & 5 & \textbf{15} & 5 & 13 & 5 & 5 & 14 & 14 & 6 & 5 \\
     Mix-vMF$_2$(0.5) & 86 & \textbf{100} & 91 & 99 & 92 & 98 & 86 & 86 & 99 & 99 & 97 & 95 \\
     Mix-vMF$_2$(0.75) & 10 & \textbf{75} & 11 & \textbf{75} & 12 & 70 & 8 & 11 & 74 & 73 & 22 & 18 \\
     Mix-vMF$_3$(0.25) & 64 & 90 & 73 & 88 & 73 & 82 & 65 & 65 & \textbf{92} & 91 & 78 & 77 \\
     Mix-vMF$_4$(0.33) & 9 & 90 & 21 & 87 & 29 & 83 & 10 & 8 & \textbf{93} & 91 & 27 & 23 \\
    Bing$_1$(0.25) & 6 & 13 & 6 & 13 & 5 & 12 & 5 & 6 & \textbf{14} & \textbf{14} & 6 & 6 \\
    Bing$_1$(0.5) & 5 & 45 & 6 & 44 & 6 & 39 & 5 & 5 & \textbf{48} & 47 & 10 & 7 \\
    Bing$_1$(1) & 6 & 98 & 8 & 98 & 9 & 96 & 5 & 6 & \textbf{99} & \textbf{99} & 37 & 26 \\
    Bing$_2$(0.1) & 5 & \textbf{17} & 6 & \textbf{17} & 5 & 16 & 5 & 5 & \textbf{17} & \textbf{17} & 6 & 6 \\
    Bing$_2$(0.25) & 6 & 84 & 7 & 82 & 7 & 76 & 5 & 6 & \textbf{86} & 85 & 19 & 13 \\
    LP$_3$(0.1) & 5 & 5 & 5 & \textbf{6} & 5 & 5 & 5 & 5 & 5 & 5 & 5 & 5 \\
    LP$_4$(0.1) & 4 & \textbf{5} & \textbf{5} & \textbf{5} & 4 & \textbf{5} & \textbf{5} & \textbf{5} & \textbf{5} & \textbf{5} & \textbf{5} & \textbf{5} \\
    LP$_3$(0.5) & 5 & 4 & 8 & 6 & \textbf{10} & 5 & 5 & 5 & 5 & 5 & 6 & 6 \\
    LP$_4$(0.5) & 6 & 5 & 6 & \textbf{7} & 6 & \textbf{7} & 5 & 5 & 5 & 6 & 5 & 5 \\
    LP$_3$(1) & 5 & 5 & 20 & 5 & \textbf{33} & 5 & 7 & 5 & 5 & 5 & 10 & 7 \\
    LP$_4$(1) & 5 & 4 & 6 & 10 & 6 & \textbf{17} & 5 & 6 & 5 & 9 & 7 & 5 \\
    \midrule
    \bottomrule[0.5mm]
\end{tabular}
\caption{Empirical power for $n=100$ and $d=3$.} \label{table_power_100_3}
\end{table}

\begin{table}[H]
\footnotesize
\centering
\begin{tabular}{l|rrrrrrrrrrrr}
    \toprule[0.5mm]
    \midrule
    \backslashbox{\scriptsize{Alternative}}{\scriptsize{Test}} & $T_{n,1}$ & $T_{n,2}$ & $T_{n,3}$ & $T_{n,4}$ & $T_{n,5}$ & $T_{n,6}$ & $A_n$ & $R_n^{\text{ mod}}$ & $B_n$ & $G_n$ & $CA_n^{100}$ & $\text{CvM}_n$ \\
    \midrule
    $\unifsphere$ & 5 & 5 & 5 & 5 & 5 & 5 & 5 & 5 & 5 & 5 & 5 & 5 \\
    vMF$_1$(0.05) & 5 & 4 & \textbf{6} & 4 & 5 & 5 & 4 & 4 & \textbf{6} & \textbf{6} & 5 & 5 \\
    vMF$_1$(0.1) & \textbf{6} & 4 & \textbf{6} & 4 & \textbf{6} & 4 & 5 & 5 & \textbf{6} & \textbf{6} & \textbf{6} & 5 \\
    vMF$_1$(0.25) & \textbf{11} & 4 & \textbf{11} & 5 & 8 & 5 & 10 & \textbf{11} & 6 & 6 & 9 & 10 \\
    vMF$_1$(0.5) & \textbf{37} & 5 & 31 & 6 & 22 & 4 & 34 & 33 & 6 & 6 & 28 & 33 \\
    vMF$_1$(0.75) & \textbf{73} & 5 & 63 & 5 & 49 & 5 & 71 & 72 & 7 & 7 & 60 & 71 \\
    vMF$_1$(1) & \textbf{94} & 7 & 89 & 7 & 78 & 7 & \textbf{94} & \textbf{94} & 8 & 9 & 88 & \textbf{94} \\
    vMF$_1$(2) & \textbf{100} & 66 & \textbf{100} & 58 & \textbf{100} & 50 & \textbf{100} & \textbf{100} & 58 & 57 & \textbf{100} & \textbf{100} \\
    Mix-vMF$_1$(0.25) & 33 & 7 & 30 & 8 & 24 & 8 & 33 & \textbf{34} & 8 & 8 & 28 & \textbf{34} \\
    Mix-vMF$_1$(0.5) & 4 & \textbf{7} & 5 & \textbf{7} & 5 & \textbf{7} & 5 & 6 & \textbf{7} & \textbf{7} & 4 & 4 \\
     Mix-vMF$_2$(0.5) & 89 & \textbf{96} & 94 & 95 & 92 & 92 & 90 & 90 & 94 & 94 & 92 & 93 \\
     Mix-vMF$_2$(0.75) & 6 & \textbf{54} & 7 & 53 & 10 & 46 & 5 & 5 & 50 & 51 & 7 & 7 \\
     Mix-vMF$_3$(0.25) & 72 & 65 & \textbf{77} & 61 & 72 & 52 & 73 & 72 & 72 & 71 & 69 & \textbf{77} \\
     Mix-vMF$_4$(0.33) & 23 & 71 & 36 & 66 & 40 & 59 & 24 & 23 & 81 & \textbf{82} & 27 & 32 \\
    Bing$_1$(0.25) & 5 & 14 & 5 & 14 & 6 & 13 & 5 & 5 & 16 & \textbf{17} & 5 & 5 \\
    Bing$_1$(0.5) & 5 & 55 & 7 & 48 & 8 & 42 & 5 & 5 & \textbf{65} & \textbf{65} & 7 & 10 \\
    Bing$_1$(1) & 5 & \textbf{100} & 12 & \textbf{100} & 18 & 99 & 6 & 6 & \textbf{100} & \textbf{100} & 25 & 25 \\
    Bing$_2$(0.1) & 4 & 13 & 6 & 11 & 5 & 10 & 5 & 5 & 14 & \textbf{15} & 5 & 6 \\
    Bing$_2$(0.25) & 6 & 76 & 8 & 70 & 10 & 62 & 6 & 5 & 79 & \textbf{80} & 9 & 9 \\
    LP$_3$(0.1) & \textbf{5} & \textbf{5} & \textbf{5} & \textbf{5} & \textbf{5} & \textbf{5} & \textbf{5} & \textbf{5} & \textbf{5} & \textbf{5} & \textbf{5} & \textbf{5} \\
    LP$_4$(0.1) & \textbf{5} & 4 & \textbf{5} & 4 & 4 & \textbf{5} & \textbf{5} & \textbf{5} & \textbf{5} & \textbf{5} & \textbf{5} & \textbf{5} \\
    LP$_3$(0.5) & \textbf{5} & 4 & \textbf{5} & \textbf{5} & \textbf{5} & \textbf{5} & \textbf{5} & \textbf{5} & \textbf{5} & \textbf{5} & \textbf{5} & \textbf{5} \\
    LP$_4$(0.5) & \textbf{5} & \textbf{5} & \textbf{5} & \textbf{5} & \textbf{5} & \textbf{5} & \textbf{5} & \textbf{5} & \textbf{5} & \textbf{5} & \textbf{5} & \textbf{5} \\
    LP$_3$(1) & 5 & 5 & \textbf{7} & 5 & \textbf{7} & 5 & 5 & 5 & 5 & 5 & 5 & 5 \\
    LP$_4$(1) & 5 & 4 & 5 & 5 & 5 & \textbf{6} & 5 & 5 & 4 & 4 & 5 & 5 \\
    \midrule
    \bottomrule[0.5mm]
\end{tabular}
\caption{Empirical power for $n=100$ and $d=5$.} \label{table_power_100_5}
\end{table}

\begin{table}[H]
\footnotesize
\centering
\begin{tabular}{l|rrrrrrrrrrrr}
    \toprule[0.5mm]
    \midrule
    \backslashbox{\scriptsize{Alternative}}{\scriptsize{Test}} & $T_{n,1}$ & $T_{n,2}$ & $T_{n,3}$ & $T_{n,4}$ & $T_{n,5}$ & $T_{n,6}$ & $A_n$ & $R_n^{\text{ mod}}$ & $B_n$ & $G_n$ & $CA_n^{100}$ & $\text{CvM}_n$ \\
    \midrule
    $\unifsphere$ & 5 & 5 & 6 & 5 & 5 & 5 & 5 & 5 & 5 & 6 & 4 & 5 \\
    vMF$_1$(0.05) & 5 & 5 & \textbf{6} & \textbf{6} & 5 & \textbf{6} & 4 & \textbf{6} & 4 & 4 & 3 & 4 \\
    vMF$_1$(0.1) & 5 & 5 & 5 & \textbf{6} & 5 & \textbf{6} & 5 & \textbf{6} & 4 & 5 & 4 & 5 \\
    vMF$_1$(0.25) & \textbf{7} & 6 & \textbf{7} & 6 & 6 & 6 & 6 & \textbf{7} & 4 & 4 & 5 & 6 \\
    vMF$_1$(0.5) & 14 & 5 & 12 & 6 & 9 & 6 & 13 & \textbf{15} & 4 & 4 & 9 & 12 \\
    vMF$_1$(0.75) & 28 & 6 & 23 & 6 & 15 & 6 & 29 & \textbf{30} & 4 & 4 & 18 & 27 \\
    vMF$_1$(1) & 53 & 5 & 40 & 6 & 25 & 7 & 51 & \textbf{54} & 5 & 5 & 33 & 51 \\
    vMF$_1$(2) & \textbf{100} & 15 & 98 & 13 & 88 & 11 & \textbf{100} & \textbf{100} & 12 & 12 & 96 & \textbf{100} \\
    Mix-vMF$_1$(0.25) & 14 & 5 & 12 & 5 & 9 & 6 & 12 & \textbf{15} & 6 & 6 & 10 & 14 \\
    Mix-vMF$_1$(0.5) & 4 & 4 & 5 & 5 & 5 & 5 & 4 & 5 & 5 & \textbf{6} & 4 & 4 \\
     Mix-vMF$_2$(0.5) & 77 & 56 & \textbf{80} & 47 & 68 & 33 & 76 & 77 & 47 & 47 & 61 & 78 \\
     Mix-vMF$_2$(0.75) & 6 & \textbf{16} & 8 & 14 & 9 & 11 & 5 & 7 & 14 & 14 & 5 & 6 \\
     Mix-vMF$_3$(0.25) & 57 & 21 & 56 & 18 & 40 & 14 & 56 & \textbf{60} & 19 & 19 & 39 & 58 \\
     Mix-vMF$_4$(0.33) & 29 & 29 & 32 & 24 & 26 & 19 & 28 & \textbf{33} & 29 & 28 & 20 & 30 \\
    Bing$_1$(0.25) & 5 & 14 & 7 & 13 & 6 & 11 & 5 & 6 & \textbf{20} & \textbf{20} & 5 & 6 \\
    Bing$_1$(0.5) & 5 & 61 & 10 & 55 & 12 & 42 & 6 & 5 & \textbf{84} & \textbf{84} & 6 & 8 \\
    Bing$_1$(1) & 6 & \textbf{100} & 26 & \textbf{100} & 38 & \textbf{100} & 6 & 6 & \textbf{100} & \textbf{100} & 10 & 21 \\
    Bing$_2$(0.1) & 5 & 9 & 5 & 8 & 6 & 8 & 5 & 6 & \textbf{10} & \textbf{10} & 5 & 6 \\
    Bing$_2$(0.25) & 5 & 65 & 9 & 56 & 11 & 43 & 5 & 5 & \textbf{66} & 65 & 5 & 7 \\
    LP$_3$(0.1) & 5 & \textbf{6} & 5 & 5 & 5 & 5 & 5 & 5 & \textbf{6} & 5 & 4 & 5 \\
    LP$_4$(0.1) & 5 & 5 & 5 & 4 & 5 & 5 & 5 & 5 & 5 & 5 & 4 & \textbf{6} \\
    LP$_3$(0.5) & \textbf{5} & \textbf{5} & \textbf{5} & \textbf{5} & \textbf{5} & \textbf{5} & \textbf{5} & \textbf{5} & \textbf{5} & \textbf{5} & 4 & \textbf{5} \\
    LP$_4$(0.5) & \textbf{5} & \textbf{5} & \textbf{5} & \textbf{5} & \textbf{5} & \textbf{5} & \textbf{5} & \textbf{5} & \textbf{5} & 4 & \textbf{5} & \textbf{5} \\
    LP$_3$(1) & 5 & 5 & \textbf{6} & 5 & 5 & \textbf{6} & \textbf{6} & 5 & 5 & 5 & 5 & 5 \\
    LP$_4$(1) & \textbf{5} & \textbf{5} & \textbf{5} & \textbf{5} & \textbf{5} & \textbf{5} & 4 & \textbf{5} & \textbf{5} & \textbf{5} & \textbf{5} & \textbf{5} \\
    \midrule
    \bottomrule[0.5mm]
\end{tabular}
\caption{Empirical power for $n=100$ and $d=10$.} \label{table_power_100_10}
\end{table}

\section{Data example}\label{sec:real_data}
As a real data example we consider the midpoints of craters on the surface of the moon. The analysed data is contained in the \textit{Moon Crater Database v1 Salamuni{\'c}car} provided at
\begin{center}
\url{https://astrogeology.usgs.gov/search/map/Moon/Research/Craters/GoranSalamuniccar_MoonCraters},
\end{center}
see \cite{SLM:2012,Setal:2014} as well as the webpage for more information on the provenance of the data set. The considered LU78287GT catalogue of 78287 craters is currently the most complete catalogue of Lunar impact craters. This catalogue is globally complete up to diagonal larger than 8 km, and each crater provides at least latitude, longitude, and diameter. Since this large amount of data clearly uniformly leads to rejection of the hypothesis, we consider a subset of midpoints of the craters by restricting the size of the craters to diameters of at least 150km. The resulting data set consists of 119 data points, see Figure \ref{fig:moon1}.
\begin{figure}[t]
\centering
\includegraphics[scale=0.25]{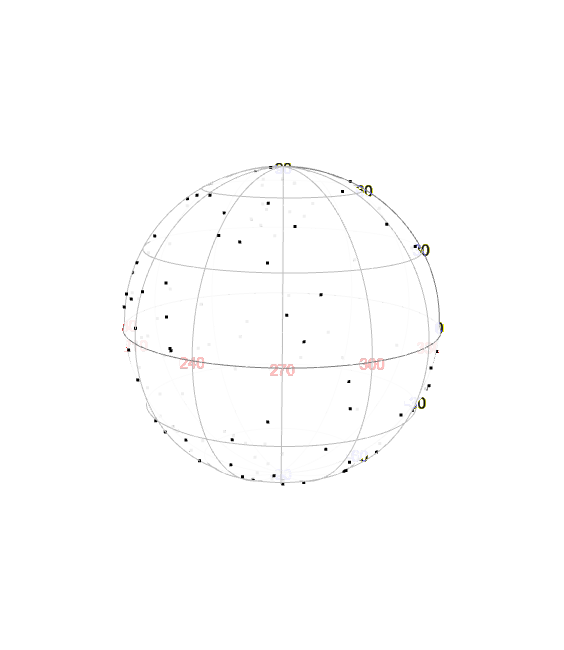}
\caption{Midpoints of craters with diameter length of at least 150km}\label{fig:moon1}
\end{figure}
We performed the tests presented in \eqref{eq:test} with the same procedure as in Section \ref{sec:Simu} with $m=5000$ points on the sphere. In order to provide $p$-values, we simulated 1000 times with the same sample size of uniformly distributed data on the sphere and calculated the relative frequency of times $T_{n,\beta}$ was below the value of the test for a simulation run. This gives an estimate for the $p$-value of the test. The calculated values are tabulated in Table \ref{tab:moon1}. As we can see, the tests for uneven values of $\beta$ clearly reject the hypothesis of uniformity on a $1\%$ significance level, while the even values fail to do so. In view of the simulation results in Sections \ref{sec:Bah} and \ref{sec:Simu}, it seems that a unimodal law like the vMF distribution might be a suitable model to consider.

\begin{table}[t]
\centering
\begin{tabular}{l|rrrrrr}
\texttt{Test} & $T_{n,1}$ & $T_{n,2}$ & $T_{n,3}$ & $T_{n,4}$ & $T_{n,5}$ & $T_{n,6}$ \\ \hline
\texttt{p.values} & 0.001 & 0.364 & 0.001 & 0.272 & 0.004 & 0.230
\end{tabular}
\caption{Empirical $p$-values of the test statistics $T_{n,\beta}$ for several parameters $\beta$ (1000 Monte Carlo runs). }\label{tab:moon1}
\end{table}

\section{Conclusion and outlook}\label{sec:conclusions}
We introduced a family of tests of uniformity on the $d$-dimensional hypersphere depending on powers of maximal projections, which include classical uniformity tests of Rayleigh and Bingham as well as measures of multivariate skewness and kurtosis in the sense of Malkovich and Afifi. We proved the weak convergence of the tests to maxima of a squared centred Gaussian process with zonal covariance structure in the space of continuous functions. We derived the eigenvalues of the connected integral operator and applied the largest eigenvalue to derive local Bahadur efficiencies. We proved consistency, as well as the behaviour under contiguous alternatives of the tests in dependence of the parity of the power parameter. Simulations illustrate that the new tests are serious competitors to classical procedures.

We finish the article by pointing out open questions and new directions for further research. The work at hand leaves the topic about the limit distribution of $(T_{n, \, \beta})_{\beta \in \N}$ under familiar alternatives, like the von Mises--Fisher distribution or the Watson distribution, untouched. In this context only local Bahadur efficiencies have been derived, which have the major drawback that trivial local Bahadur efficiencies, as in the case of $\widetilde{T}_{n, \, 2}$ and a von Mises--Fisher distribution, do not suggest that the statistic is blind against the respective alternative. Essentially, this is because the local Bahadur efficiency does not make a statement about rates of consistency under which a statistic might nevertheless recognize the alternative. For example, the authors of \cite{PPV:2021} prove that the Bingham test, which has structural resemblance with $T_{n, \, 2}$, reliably recognizes a von Mises--Fisher distribution under certain rates of consistency. Hence, it might be interesting to explore the problem of rates of consistency, possibly even for very general alternatives like rotationally symmetric distributions. Another direction for future research could be the problem of local asymptotic normality in the sense of Le Cam. Statistical models with this property allow the construction of optimal tests of a particular kind. This has already been done for the Rayleigh test for certain rotationally symmetric distributions, see \cite{CPV:2017}. Since we establish some results of the Le Cam theory for the statistical model in \ref{sec:asy_cont}, one might continue from there. In Section \ref{sec:Simu} we propose a new parametric family of spherical distributions, namely the Legendre polynomial distribution, which has shown some interesting properties in the numerical simulations. Due to its straightforward parameterization it is relatively easy to work with and to simulate. Nonetheless, at the given moment there is not enough substantial knowledge about the properties and practicability of this distribution, especially in higher dimensions, so that this is another option for further research. High-dimensional results for testing uniformity against monotone rotationally symmetric alternatives in \cite{CPV:2017} give first results for $T_{n,\,1}$, since this test is equivalent to the Rayleigh test. Similar results for $\beta=2$ may be related to \cite{CPV:2022}, and the question for larger $\beta$ is still open.

\section*{Acknowledgements}
The authors thank Norbert Henze for fruitful discussions and numerous suggestions that led to an improvement of the article. Furthermore, the authors are grateful for the suggestion of Michael A. Klatt to analyse the locations of moon craters and for providing the necessary references.

%Open Questions:
%\begin{itemize}
%\item Optimality? LAN property?
%\item limit distribution?
%\item Legendre distribution?
%\item Consistency rates?
%\item high dimensional results?
%\item Bahadur Optimality?
%\item closed form formulas of the test statistics?
%\end{itemize}

\bibliography{lit-DU}   % name your BibTeX data base
\bibliographystyle{abbrv}      % mathematics and physical sciences
\begin{appendix}
\section{Some facts on \texorpdfstring{$d$}{Lg}-dimensional Legendre polynomials and spherical harmonics}\label{sec:AppA}
A $d$-dimensional spherical harmonic function of order $k\in\N_0$ is the restriction of a $d$-dimensional harmonic polynomial of order $k$ to $\Sp$. Let $\harmonicspace$ be the space of $d$-dimensional spherical harmonics of order $k\in\N_0$, and let $\nu_d(k)=\dim(\harmonicspace)$. We have
\begin{equation}\label{eq:dimsp}
\nu_d(k) = \binom{d+k-1}{k} - \binom{d+k-3}{k-2} = \frac{d+2k-2}{d+k-2}\binom{d+k-2}{d-2},
\end{equation}
where the second binomial coefficient after the first the equation is $0$, if $k-2<0$, and the fraction after the second equation is $1$, if $d=2$ and $k=0$.
A couple of noteworthy special cases of $\nu_d(k)$ are
\begin{equation}
    \nu_d(0) = 1, \quad \nu_d(1) = d, \quad \nu_d(2) = \frac{(d+2)(d-1)}{2}, \quad \nu_2(k) = 2 \, (k>0), \quad \nu_3(k) = 2k+1,
\end{equation}
see \cite{G:1996}, Section 3, for details on spherical harmonics. Two important theoretical results in the theory of spherical harmonics are the density of finite linear combinations of spherical harmonics in $L^2(\Sp, \text{d}\sigma)$ and the orthogonality property of spherical harmonics of different order, which can be phrased as follows.
\begin{theorem}[\cite{S:1971}, Chapter 4, Corollary 2.3] \label{harmonic_dense}
The set of finite linear combinations of elements of $\bigcup_{k=0}^\infty \harmonicspace$ is dense in $L^2(\Sp, \text{d}\sigma)$.
\end{theorem}
\begin{prop}[\cite{S:1971}, Chapter 4, Corollary 2.4] \label{harmonic_orthogonal}
For $\Phi \in \harmonicspace, \Psi \in \mathcal{H}_{l}(\Sp)$ with $k \neq l$, we have
\begin{equation*}
    \skalarproduktLebesgue{\Phi}{\Psi} = \int_{\Sp} \Phi(\omega) \Psi(\omega) \,
    \text{d}\sigma(\omega) = 0.
\end{equation*}
\end{prop}
With these statements, it is possible to represent a function $f \in L^2(\Sp, \text{d}\sigma)$ uniquely as a series of spherical harmonics. For this purpose, we consider an orthonormal basis $\{\phi_{k,1},\ldots,\phi_{k,\nu_d(k)}\}$ of $\harmonicspace$ for each $k \in \N_0$. Then $\bigcup_{k=0}^\infty \{\phi_{k,1},\ldots,\phi_{k,\nu_d(k)}\}$ is an orthonormal basis of $L^2(\Sp, \text{d}\sigma)$, and we can write
\begin{equation} \label{L^2_harmonic_expansion}
        f \overset{L^2}{=} \sum_{k=0}^\infty \Psi_k
    \end{equation}
with $\Psi_k = \sum_{j=1}^{\nu_d(k)} \skalarproduktLebesgue{f}{\phi_{k,j}} \phi_{k,j} \in \harmonicspace$. The following Theorem is called the Funk--Hecke-Theorem and is used frequently.
\begin{theorem}[Funk--Hecke-Theorem, \cite{G:1996}, Theorem 3.4.1] \label{funk-hecke}
Let $k \in \N_0$ and $u \in \Sp$. If $\Lambda$ is a bounded, integrable function on $[-1,1]$ and $\phi \in \harmonicspace$, then the function $\Lambda(\skalarprodukt{u}{\cdot}) \colon \Sp \to \R; \; x \mapsto \Lambda(\skalarprodukt{u}{x})$ is integrable and
\begin{equation*}
    \int_{\Sp} \Lambda(\skalarprodukt{u}{x}) \, \phi(x) \, \text{d}\sigma(x) = \lambda_k \phi(u)
\end{equation*}
with
\begin{equation*}
    \lambda_k = |\Sp[d-2]| \int_{-1}^1 P_k^{\, d}(t) \Lambda(t) (1-t^2)^{(d-3)/2} \, \text{d}t,
\end{equation*}
where $P_k^{\, d}$ is the $d$-dimensional Legendre polynomial of order $k$.
\end{theorem}
The existence and uniqueness of higher dimensional Legendre polynomials are stated in the following Theorem.
\begin{theorem}[\cite{G:1996}, Theorem 3.3.3] \label{def_legendre}
For each $k \in \N_0$, there is exactly one polynomial $P_k^{\, d}$ on $[-1,1]$ with the property: If $\{ \phi_1,\ldots,\phi_{\nu_d(k)} \}$ is an orthonormal basis of  $\harmonicspace$, then
\begin{equation*}
    \sum_{i=1}^{\nu_d(k)} \phi_i(u) \phi_i(v) = \frac{\nu_d(k)}{|\Sp|} P_k^{\, d}(\skalarprodukt{u}{v}), \quad u, v \in \Sp.
\end{equation*}
The degree of $P_k^{\, d}$ is $k$, and the function $P_k^{\, d}(\skalarprodukt{u}{\cdot}) \colon \Sp \to \R ; \; v \mapsto P_k^{\, d}(\skalarprodukt{u}{v})$ is for fixed $u \in \Sp$ a $d$-dimensional spherical harmonic of order $k$. Furthermore, $P_k^{\, d}$ is an even function, whenever $k$ is even, and an odd function, whenever $k$ is odd.
\end{theorem}
The polynomial $P_k^{\, d}$ is called the $d$-dimensional Legendre polynomial of order $k$. Legendre polynomials of different orders fulfill certain orthogonality properties. To state these properties, we introduce the weighted scalar product
    \begin{equation} \label{legendre_scalarproduct}
        \left\langle f, g \right\rangle = \int_{-1}^1 f(t) g(t) (1-t^2)^{(d-3)/2} \, \text{d}t
    \end{equation}
for bounded and integrable functions $f, g$ on $[-1,1]$. The following proposition shows that Legendre polynomials are orthogonal w.r.t. this scalar product.
\begin{prop}[\cite{G:1996}, Proposition 3.3.6] \label{legendre_orthogonal}
Let $k, l \in \N_0$. If $P_k^{\, d}$ and $P_l^{\, d}$ are $d$-dimensional Legendre polynomials of order $k$ and $l$, respectively, then
\begin{equation*}
    \left\langle P_k^{\, d}, P_l^{\, d} \right\rangle
    = \delta_{kl} \frac{|\Sp|}{\nu_d(k) \, |\Sp[d-2]|}
    = \delta_{kl} \frac{\sqrt{\pi} \, \Gamma((d-1)/2)}{\nu_d(k) \, \Gamma(d/2)}.
\end{equation*}
\end{prop}

\begin{remark}[\cite{G:1996}, Lemma 3.3.5] \label{legendre_schranke}
Let $k \in \N_0$. For the $d$-dimensional Legendre polynomial of order $k$, we have $|P_k^{\, d}(t)| \leq 1$ for all $t \in [-1,1]$ and $P_k^{\, d}(1)=1$.
\end{remark}
The next result gives two explicit formulas for the calculation of Legendre polynomials of arbitrary dimension and order.
\begin{prop}[\cite{M:1998}, Section 1.2, S.16 and Lemma 1.6.1] \label{legendre_explizit}
For $k \in \N_0$, we have
\begin{equation*}
    P_k^{\, d}(t)
    = k! \, \Gamma\left(\frac{d-1}{2}\right) \sum_{l=0}^{\lfloor \frac{k}{2} \rfloor} \left( -\frac{1}{4} \right)^l \frac{(1-t^2)^l t^{\, k-2l}}{l! \, (k-2l)! \, \Gamma(l+\frac{d-1}{2})}
    = \sum_{l=0}^{\lfloor \frac{k}{2} \rfloor} a_{2l,k} t^{k-2l}, \quad t \in [-1,1],
\end{equation*}
with
\begin{align*}
    &a_{0, 0} = 1, \\
    &a_{2l,k} = \left( - \frac{1}{4} \right)^l \frac{\Gamma(d-1)}{\Gamma(d/2)} \frac{2^{k-1} k!}{(k+d-3)!} \frac{\Gamma(k-l+(d-2)/2)}{l! (k-2l)!}, \quad l \in \{ 0,\ldots,\lfloor k/2 \rfloor \}, \; k > 0,
\end{align*}
where $\lfloor \cdot \rfloor$ is the lower Gauss bracket.
\end{prop}
With these formulas, it is straightforward to obtain the first seven Legendre polynomials, which are given by
\begingroup
\allowdisplaybreaks
\begin{align*}
    &P_0^{\, d}(t) = 1, \quad P_1^{\, d}(t) = t, \\[0.25cm]
    &P_2^{\, d}(t) = \frac{1}{d-1}[d t^2 - 1], \\[0.25cm]
    &P_3^{\, d}(t) = \frac{1}{d-1}[(d+2) t^3 - 3t], \\[0.25cm]
    &P_4^{\, d}(t) = \frac{1}{(d-1) (d+1)}[(d+2)(d+4) t^4 - 6(d+2)t^2 + 3], \\[0.25cm]
    &P_5^{\, d}(t) = \frac{1}{(d-1) (d+1)}[(d+4)(d+6) t^5 - 10(d+4)t^3 + 15t], \\[0.25cm]
    &P_6^{\, d}(t) = \frac{1}{(d-1) (d+1) (d+3)}[(d+4)(d+6)(d+8) t^6 - 15(d+4)(d+6)t^4 + 45(d+4)t^2 - 15].
\end{align*}
\endgroup
In a reverse conclusion, we can write any $m$th power, $m \in \N_0$, of a number $t \in [-1,1]$ as a linear combination of Legendre polynomials
\begin{equation} \label{potenzen_legendre}
    t^m = \sum_{j=0}^m c_{j, \, d}(m) P_j^{\, d}(t), \quad t \in [-1,1],
\end{equation}
with coefficients $c_{j, \, d}(m)$, which only depend on $j$, $d$ and $m$. An elaborate calculation discloses the explicit form of these coefficients.
\begin{prop} \label{c_j,d(m) allgemein}
Let $k,l \in \N_0$ with $k \geq 2l$ and $[l] = {1,\ldots,l}$. Then, with the coefficients $a_{2l,k}$ from Proposition \ref{legendre_explizit}, we have
\begin{align*} \label{eigenvalues}
    c_{k-2l, \, d}(k) =
    \begin{cases}
        \begin{alignedat}{2}
        & \hspace{4cm} \frac{1}{a_{0,k}}          &,\quad& \textrm{for} \; l = 0, \\[0.25cm]
        &\sum_{r=1}^l \sum_{\substack{ (l_1,\ldots,l_r) \in [l]^r \\ l_1 + \cdots + l_r = l}} (-1)^r \frac{a_{2l_1,k} a_{2l_2,k-2l_1} \cdots a_{2l_r,k-2(l_1+\cdots+l_{r-1})}}{a_{0,k} a_{0,k-2l_1} \cdots a_{0,k-2l}}          &,\quad& \textrm{for} \; l > 0,
        \end{alignedat}
    \end{cases}
    \end{align*}
where $l_0 = 0$. These are the only nonzero coefficients in \eqref{potenzen_legendre} for given $k \in \N_0$. By the proof of Proposition \ref{K_beta_properties} i), we have in particular, $c_{0,\, d}(\beta) = \psi_d(\beta)$ for all $\beta \in \N$.
\end{prop}
\begin{proof}
The claim is obvious for $k=0$, hence let $k \in \N$. If $l = 0$, the second representation in Proposition \ref{legendre_explizit} yields
\begin{equation} \label{monom_darstellung}
    t^k = \frac{1}{a_{0,k}} P_k^{\, d}(t) - \frac{a_{2,k}}{a_{0,k}} t^{k-2} - \cdots -  \frac{a_{2g(k),k}}{a_{0,k}},
\end{equation}
where $g(k) = \lfloor k/2 \rfloor$. It follows that $c_{k,\, d}(k) = 1/a_{0,k}$, since monomials of lower degree do not contain Legendre polynomials of order $k$ by  \eqref{potenzen_legendre}. Let the claim hold true for all $j \in \{0,1,...,l-1\}$ for some $l \in \N$ with $k \geq 2l$. Next, we will prove the claim for this $l$ and conclude the statement by means of the principle of strong induction. The dots $\cdots$ are occasionally used for the sake of readability.

We repeatedly insert equation \eqref{monom_darstellung} for lower powers into \eqref{monom_darstellung}. The induction hypothesis gives
\begin{align*}
        t^k
        = \hspace{0.125cm} &c_{k,\, d}(k) P_k^{\, d}(t) - \frac{a_{2,k}}{a_{0,k}} \left[ \frac{1}{a_{0,k-2}} P_{k-2}^{\, d}(t) - \frac{a_{2,k-2}}{a_{0,k-2}} t^{k-4} - \cdots - \frac{a_{2(l-1),k-2}}{a_{0,k-2}} t^{k-2l} - \cdots - \frac{a_{2g(k-2),k-2}}{a_{0,k-2}} \right] \\[0.25cm]
        &- \frac{a_{4,k}}{a_{0,k}} t^{k-4} - \cdots - \frac{a_{2l,k}}{a_{0,k}} t^{k-2l} - \cdots - \frac{a_{2g(k),k}}{a_{0,k}} \\
        = \hspace{0.125cm} &c_{k,\, d}(k) P_k^{\, d}(t) + c_{k-2,\, d}(k) P_{k-2}^{\, d}(t) - \left( \frac{a_{4,k}}{a_{0,k}} + c_{k-2,\, d}(k) a_{2,k-2} \right) t^{k-4} \\[0.25cm] &- \cdots - \left( \frac{a_{2l,k}}{a_{0,k}} + c_{k-2,\, d}(k) a_{2(l-1),k-2} \right) t^{k-2l} - \cdots,
\end{align*}
since
\begin{equation*}
    c_{k-2,\, d}(k)
    = \sum_{r=1}^1 \sum_{\substack{ (l_1,\ldots,l_r) \in [1]^r \\ l_1 + \cdots + l_r = 1}} (-1)^r \frac{a_{2l_1,k} a_{2l_2,k-2l_1} \cdots a_{2l_r,k-2(l_1+\cdots+l_{r-1})}}{a_{0,k} a_{0,k-2l_1} \cdots a_{0,k-2l}} = - \frac{a_{2,k}}{a_{0,k} a_{0,k-2}}.
\end{equation*}
Furthermore, we have
\begin{align*}
    c_{k-4,\, d}(k)
    &= \sum_{r=1}^2 \sum_{\substack{ (l_1,\ldots,l_r) \in [2]^r \\ l_1 + \cdots + l_r = 2}} (-1)^r \frac{a_{2l_1,k} a_{2l_2,k-2l_1} \cdots a_{2l_r,k-2(l_1+\cdots+l_{r-1})}}{a_{0,k} a_{0,k-2l_1} \cdots a_{0,k-2l}} = - \frac{a_{4,k}}{a_{0,k} a_{0,k-4}} +  \frac{a_{2,k} a_{2,k-2}}{a_{0,k} a_{0,k-2} a_{0,k-4}} \\
    &= - \left( \frac{a_{4,k}}{a_{0,k} } + c_{k-2,d}(k) a_{2,k-2} \right) \frac{1}{a_{0,k-4}},
\end{align*}
so that
\begin{align*}
    t^k
    = \hspace{0.125cm} &c_{k,\, d}(k) P_k^{\, d}(t) + c_{k-2,\, d}(k) P_{k-2}^{\, d}(t) + c_{k-4,\, d}(k) P_{k-4}^{\, d}(t) \\[0.25cm]
    &- \cdots - \left( \frac{a_{2l,k}}{a_{0,k}} + c_{k-2,\, d}(k) a_{2(l-1),k-2} + c_{k-4,\, d}(k) a_{2(l-2),k-4} \right) t^{k-2l} - \cdots
\end{align*}
If we continue this procedure iteratively, then we arrive at
\begin{align*}
    t^k
    = c_{k,\, d}(k) P_k^{\, d}(t) + \cdots - \left( \frac{a_{2l,k}}{a_{0,k} a_{0,k-2l}} + \sum_{i=1}^{l-1} c_{k-2i,\, d}(k) \frac{a_{2(l-i),k-2i}}{a_{0,k-2l}} \right) P_{k-2l}^{\, d}(t) - \cdots.
\end{align*}
For the following, we abbreviate the quotient from the Proposition as
\begin{equation*}
    q_l(l_1,\ldots,l_r) = \frac{a_{2l_1,k} a_{2l_2,k-2l_1} \cdots a_{2l_r,k-2(l_1+\cdots+l_{r-1})}}{a_{0,k} a_{0,k-2l_1} \cdots a_{0,k-2l}}.
\end{equation*}
Once again, with the induction hypothesis it follows for each $i \in \N$ with $i \leq l-1$
\begin{align*}
    c_{k-2i,\, d}(k) \frac{a_{2(l-i),k-2i}}{a_{0,k-2l}}
    &= \sum_{s=1}^i \sum_{\substack{ (i_1,\ldots,i_s) \in [i]^s \\ i_1 + \cdots + i_s = i}} (-1)^s q_i(i_1,\ldots,i_s) \frac{a_{2(l-i),k-2i}}{a_{0,k-2l}} \\[0.25cm]
    &= \sum_{s=1}^i \sum_{\substack{ (i_1,\ldots,i_s) \in [i]^s \times \{ l-i \} \\ i_1 + \cdots + i_{s+1} = l}} (-1)^s q_l(i_1,\ldots,i_{s+1})
\end{align*}
and, thus, finally
\begin{align*}
    c_{k-2l,\, d}(k)
    &= - \left( \frac{a_{2l,k}}{a_{0,k} a_{0,k-2l}} + \sum_{i=1}^{l-1} c_{k-2i,d}(k) \frac{a_{2(l-i),k-2i}}{a_{0,k-2l}} \right) \\[0.25cm]
    &= - \left( \frac{a_{2l,k}}{a_{0,k} a_{0,k-2l}} + \sum_{i=1}^{l-1} \sum_{s=1}^i \sum_{\substack{ (i_1,\ldots,i_s) \in [i]^s \times \{ l-i \} \\ i_1 + \cdots + i_{s+1} = l}} (-1)^s q_l(i_1,\ldots,i_{s+1}) \right) \\[0.25cm]
    &= \sum_{r=1}^l \sum_{\substack{ (l_1,\ldots,l_r) \in [l]^r \\ l_1 + \cdots + l_r = l}} (-1)^r q_l(l_1,\ldots,l_r),
\end{align*}
since the first summand satisfies
\begin{equation*}
    - \frac{a_{2l,k}}{a_{0,k} a_{0,k-2l}} = \sum_{\substack{ (l_1) \in [l]^1 \\ l_1 = l}} (-1)^1 q_l(l_1).
\end{equation*}
 In the second summand we sum over all tuples $(i_1,\ldots,i_{s+1}) \in [i]^s \times \{ l-i \}$ satisfying $i_1 + \cdots + i_{s+1} = l$ for $s = 1,\ldots,i; i = 1,\ldots,l-1$. This coincides exactly with the sum over all tuples $(l_1,\ldots,l_r) \in [l]^r$ satisfying $l_1 + \cdots + l_r = l$ for $r = 2,\ldots,l$.
\end{proof}
For the first couple of powers, we have
\begingroup
\allowdisplaybreaks
\begin{align*}
    &t^0 = P_0^{\, d}(t), \quad t = P_1^{\, d}(t), \\[0.25cm]
    &t^2 = \frac{d-1}{d} P_2^{\, d}(t) + \frac{1}{d}, \\[0.25cm]
    &t^3 = \frac{d-1}{d+2} P_3^{\, d}(t) + \frac{3}{d+2} P_1^{\, d}(t), \\[0.25cm]
    &t^4 = \frac{(d-1) (d+1)}{(d+2) (d+4)}P_4^{\, d}(t) + \frac{6(d-1)}{d (d+4)} P_2^{\, d}(t) + \frac{3}{d (d+2)}, \\[0.25cm]
    &t^5 = \frac{(d-1)(d+1)}{(d+4)(d+6)} P_5^{\, d}(t) + \frac{10(d-1)}{(d+2) (d+6)} P_3^{\, d}(t) + \frac{15}{(d+2) (d+4)} P_1^{\, d}(t), \\[0.25cm]
    &t^6 = \frac{(d-1)(d+1)(d+3)}{(d+4) (d+6) (d+8)} P_6^{\, d}(t) + \frac{15(d-1)(d+1)}{(d+2) (d+4) (d+8)} P_4^{\, d}(t) \\[0.25cm]
    &\qquad+ \frac{45(d-1)}{d (d+4) (d+6)} P_2^{\, d}(t) + \frac{15}{d (d+2)(d+4)}. \\
\end{align*}
\endgroup

\section{Technical Lemmas and Proofs} \label{sec:AppB}

\begin{lemma} \label{lem:mean}
If $U \sim \unifsphere$, then
\begin{align*}
    \begin{aligned}
        \eta_\beta(b,c)
        &= \E (\skalarprodukt{b}{U})^{\, \beta} (\skalarprodukt{c}{U})^{\, \beta}
        = \sum_{j=0}^\beta \frac{(c_{j,\, d}(\beta))^2}{\nu_d(j)} P_j^{\, d}(\skalarprodukt{b}{c}), \quad b,c \in \Sp, \beta \in \N.
    \end{aligned}
\end{align*}
\end{lemma}

\begin{proof}
It follows by the Funk--Hecke-Theorem, Theorem \ref{funk-hecke}, applied to the spherical harmonic $P_j^{\, d}(\skalarprodukt{c}{\cdot}) \colon \Sp \to \R$, $c \in \Sp$, and the representation of a monomial via \eqref{potenzen_legendre}
\begin{align*}
    \eta_\beta(b,c)
    &= \E (\skalarprodukt{b}{U})^{\, \beta} (\skalarprodukt{c}{U})^{\, \beta}
    = \frac{1}{|\Sp|} \int_{\Sp} (\skalarprodukt{b}{\omega})^{\, \beta} (\skalarprodukt{c}{\omega})^{\, \beta} \, \text{d}\sigma(\omega) \\[0.25cm]
    &= \frac{1}{|\Sp|} \int_{\Sp} (\skalarprodukt{b}{\omega})^{\, \beta} \sum_{j=0}^\beta c_{j,\, d}(\beta) \, P_j^{\, d}(\skalarprodukt{\omega}{c}) \, \text{d}\sigma(\omega) \\[0.25cm]
    &= \frac{|\Sp[d-2]|}{|\Sp|} \sum_{j=0}^\beta c_{j,\, d}(\beta) \, P_j^{\, d}(\skalarprodukt{b}{c}) \int_{-1}^1 t^{\, \beta} P_j^{\, d}(t) (1-t^2)^{(d-3)/2} \, \text{d}t \\[0.25cm]
    &= \frac{|\Sp[d-2]|}{|\Sp|} \sum_{j=0}^\beta \sum_{l=0}^\beta c_{j,\, d}(\beta) \, c_{l,\, d}(\beta) \, P_j^{\, d}(\skalarprodukt{b}{c}) \int_{-1}^1 P_l^{\, d}(t) P_j^{\, d}(t) (1-t^2)^{(d-3)/2} \, \text{d}t.
\end{align*}
Due to the orthogonality properties of the $d$-dimensional Legendre polynomials, see Proposition \ref{legendre_orthogonal}, we can conclude
\begin{align*}
    \eta_\beta(b,c)
    &= \frac{|\Sp[d-2]|}{|\Sp|} \sum_{j=0}^\beta \sum_{l=0}^\beta c_{j,\, d}(\beta) \, c_{l,\, d}(\beta) \, P_j^{\, d}(\skalarprodukt{b}{c}) \, \left\langle P_l^{\, d}, P_j^{\, d} \right\rangle \\[0.25cm]
    &= \sum_{j=0}^\beta \sum_{l=0}^\beta c_{j,\, d}(\beta) \, c_{l,\, d}(\beta) \, P_j^{\, d}(\skalarprodukt{b}{c}) \, \frac{\delta_{jl}}{\nu_d(j)}
    = \sum_{j=0}^\beta \frac{(c_{j,\, d}(\beta))^2}{\nu_d(j)} P_j^{\, d}(\skalarprodukt{b}{c}).
\end{align*}
\end{proof}

\begin{lemma}\label{lem:DQVK}
Under $(\mathbb{P}^{(n)})_{n\in\N}$ we have for $n\rightarrow\infty$
\begin{eqnarray*}
\log L_n \cd \mbox{N}\left(-\frac{\tau^2}{2},\tau^2\right),
\end{eqnarray*}
where $\tau^2=\int_{\Sp} h^2(x)\sigma(x)\;\text{d}x<\infty$.
\end{lemma}
\begin{proof}
Using a Taylor expansion of the logarithm around $x_0=1$ we have
\begin{eqnarray*}\label{DQA1}
\log L_n(U_{n1},\ldots,U_{nn})&=&\sum_{j=1}^n\log\left(1+\frac{h(U_{nj})}{\sqrt{n}}\right)\\
&=&\sum_{j=1}^n\left\{\frac{h(U_{nj})}{\sqrt{n}}-\frac{h^2(U_{nj})}{2n}+R_{n,j}\right\},
\end{eqnarray*}
with
\begin{equation}\label{Restgliedab}
R_{n,j}=\frac{1}{3!}\frac{2}{\left(1+\eta_{n,j}\right)^3}\frac{h^3(U_{n,j})}{n^{\frac{3}{2}}},
\end{equation}
where $|\eta_{n,j}|\le\frac{\left|h\left(U_{n,j}\right)\right|}{\sqrt{n}}$. Since $h$ is bounded, we have $\displaystyle{\sum_{j=1}^nR_{n,j}=o_{\mathbb{P}^{(n)}}(1)}$, and
\begin{equation*}
E\left(\frac{1}{n}\sum_{j=1}^nh^2\left(U_{n,j}\right)\right)\longrightarrow \tau^2\;\mbox{as well as}\;V\left(\frac{1}{n}\sum_{j=1}^nh^2\left(U_{n,j}\right)\right)\longrightarrow 0
\end{equation*}
for $n\rightarrow\infty$. The claim is a consequence of the Lindeberg--Feller CLT.
\end{proof}

The next Proposition deals with the technical calculations of Examples \ref{bsp_bahadur_vMF}, \ref{bsp_bahadur_watson} and \ref{bsp_bahadur_LP}. Before we provide the proof, we present some special functions and their properties, compare with \cite{MJ:2000}, Appendix 1.
\begin{enumerate}[label={}]
    \item Modified Bessel function of the first kind and order $p \geq 0$:
    \begin{equation} \label{bessel_def}
        I_p(\kappa) = \frac{(\kappa/2)^p}{\Gamma(p+1/2) \Gamma(1/2)} \int_{-1}^1 e^{\kappa t} (1-t^2)^{p-1/2} \, \text{d}t, \quad \kappa > 0.
    \end{equation}
    \item The function $I_p$ satisfies
    \begin{equation} \label{bessel_reihe}
        I_p(\kappa) = \sum_{r=0}^\infty \frac{1}{\Gamma(p+r+1) \Gamma(r+1)} \left( \frac{\kappa}{2} \right)^{2r+p}, \quad p \geq 0, \; \kappa > 0.
    \end{equation}
    \item For $\kappa > 0$ and $p \geq 1$ we have
    \begin{equation} \label{bessel_ableitung}
        \kappa I_p'(\kappa) = p I_p(\kappa) + \kappa I_{p+1}(\kappa), \quad
        I_0'(\kappa) = I_1(\kappa).
    \end{equation}
    \item For $\kappa > 0$ let
    \begin{equation} \label{A_d}
        A_d(\kappa) = \frac{I_{d/2}(\kappa)}{I_{d/2-1}(\kappa)}.
    \end{equation}
    \item For sufficiently small $\kappa > 0$ we have
    \begin{equation} \label{A_d_kappa_klein}
        A_d(\kappa) = \frac{\kappa}{d} - \frac{\kappa^3}{d^2(d+2)} + O(\kappa^5).
    \end{equation}
    \item For $\kappa > 0$ we have
    \begin{equation} \label{A_d_ableitung}
        A_d'(\kappa) = 1 - A_d^2(\kappa) - \frac{d-1}{\kappa} A_d(\kappa).
    \end{equation}
    \item For $\kappa > 0$ let
    \begin{equation} \label{a_d}
        a_d(\kappa) = 2 \pi^{d/2} \left( \frac{\kappa}{2} \right)^{1-d/2} I_{d/2-1}(\kappa).
    \end{equation}
    \item With the series expansion in \eqref{bessel_reihe} we obtain
    \begin{equation} \label{a_d_grenzwert}
        a_d(\kappa)
        = 2 \pi^{d/2} \left( \frac{\kappa}{2} \right)^{1-d/2} \left[ \frac{1}{\Gamma(d/2) \Gamma(1)} \left( \frac{\kappa}{2} \right)^{d/2-1} + O(\kappa^{d/2+1}) \right] \overset{\kappa \to0^+}{\rightarrow} |\Sp|^+.
    \end{equation}
    \item An application of \eqref{bessel_ableitung} yields for $\kappa > 0$
    \begin{equation} \label{log_a_d_ableitung}
        \partial_\kappa [ \log a_d(\kappa) ] = \frac{a_d'(\kappa)}{a_d(\kappa)} = A_d(\kappa).
    \end{equation}
    \item Kummer function for non-negative $\kappa$:
    \begin{equation} \label{kummer_def}
        M(a,b,\kappa) = \frac{1}{B(a, b-a)} \int_{-1}^1 e^{\kappa t^2} t^{2a-1} (1-t^2)^{b-a-1} \, \text{d}t, \quad b>a>0, \; \kappa \geq 0.
    \end{equation}
    \item We have
    \begin{equation} \label{kummer_reihe}
        M(a,b,\kappa) = \sum_{r=0}^\infty \frac{\Gamma(a+r) \Gamma(b)}{\Gamma(a) \Gamma(b+r)} \frac{\kappa^r}{r!}, \quad b>a>0, \; \kappa \geq 0.
    \end{equation}
    \item For $\kappa \geq 0$ and $b > a > 0$ we have
    \begin{equation} \label{kummer_ableitung}
        M'(a,b,\kappa) = \frac{a}{b} M(a+1,b+1,\kappa).
    \end{equation}
    \item With the series expansion in \eqref{kummer_reihe} we obtain for $b > a > 0$
    \begin{equation} \label{kummer_grenzwert}
        M(a,b,\kappa) = 1 + O(\kappa) \overset{\kappa \to0^+}{\rightarrow} 1^+.
    \end{equation}
    \item For $\kappa \geq 0$ let
    \begin{equation} \label{D_d}
        D_d(\kappa) = \frac{M(3/2,d/2+1,\kappa)}{d M(1/2,d/2,\kappa)}.
    \end{equation}
    \item For sufficiently small $\kappa \geq 0$ we have
    \begin{equation} \label{D_d_kappa_klein}
        D_d(\kappa) = \frac{1}{d} + \frac{2(d-1)}{d^2(d+2)} \kappa + O(\kappa^2).
    \end{equation}
    \item For $\kappa > 0$ we have
    \begin{equation} \label{D_d_ableitung}
        D_d'(\kappa) = \frac{\frac{3}{d+2} M(5/2,d/2+2,\kappa) M(1/2,d/2,\kappa) - \frac{1}{d} M(3/2,d/2+1,\kappa)}{d M^2(1/2,d/2,\kappa)}.
    \end{equation}
    \item With \eqref{kummer_grenzwert} we obtain
    \begin{equation} \label{D_d_ableitung_grenzwert}
        \limkappa D_d'(\kappa) = \frac{2(d-1)}{d^2(d+2)}
    \end{equation}
    \item For $\kappa \geq 0$ let
    \begin{equation} \label{d_d}
        d_d(\kappa) = \frac{2 \pi^{d/2}}{\Gamma(d/2)} M(1/2,d/2,\kappa).
    \end{equation}
    \item Due to the series expansion in \eqref{kummer_reihe} we have for $\kappa > 0$
    \begin{equation} \label{d_d_negativ}
        1 - M(1/2,d/2,\kappa)
        = - \sum_{r=1}^\infty \frac{\Gamma(1/2+r) \Gamma(d/2)}{\Gamma(1/2) \Gamma(d/2+r)} \frac{\kappa^r}{r!} < 0.
    \end{equation}
    \item An application of \eqref{kummer_ableitung} immediately gives for $\kappa > 0$
    \begin{equation} \label{log_d_d_ableitung}
        \partial_\kappa [ \log d_d(\kappa) ] = \frac{d_d'(\kappa)}{d_d(\kappa)} = D_d(\kappa).
    \end{equation}
\end{enumerate}
\noindent
We come to the announced Proposition. The notation corresponds as far as possible to the one of Chapter \ref{sec:Bah}.
\begin{prop} \label{bahadur_ARE_berechnungen}
Let $\kappa >0$, and let $f(\cdot \, | \, \kappa)$ be a continuous density w.r.t the surface measure $\sigma$, which is parameterized via $\kappa$. The limit case $\kappa \to0^+$ is assumed to yield the uniform distribution on $\Sp$. Further, let $U$ be a random vector, which is distributed according to $f(\cdot \, | \, \kappa)$. Then
\begin{enumerate}[label=\roman*)]
    \item For the von Mises--Fisher alternative \textsc{vMF}$(\theta, \kappa)$, with $\theta \in \Sp$ fixed:
    \begin{enumerate}
        \item
        \begin{equation*}
            \textsc{KL}(\kappa, 0) = A_d(\kappa) \kappa - \log a_d(\kappa) +  \log |\Sp|,
        \end{equation*}
        \item
        \begin{equation*}
            \gamma_\kappa(b) = \frac{|\Sp[d-2]|}{a_d(\kappa)} \sum_{l=0}^\infty \frac{\kappa^l}{l!} \sum_{j=0}^\beta c_{j, \, d}(\beta) \Delta_j(l) P_j^d(\skalarprodukt{\theta}{b}) - \psi_d(\beta), \quad b \in \Sp,
        \end{equation*}
        \\
        \begin{equation*}
            \Delta_j(l) = \int_{-1}^1 P_j^d(t) t^l (1-t^2)^{(d-3)/2} \, \text{d}t, \quad j = 1,\ldots,\beta, \; l \in \N_0.
        \end{equation*}
        \item
        \begin{equation*}
            \lim_{\kappa \to0^+} \frac{\max_{b \in \Sp} \gamma_\kappa^2(b)}{2 \textsc{KL}(\kappa, 0)} = \lambda_1 \nu_d(1).
        \end{equation*}
    \end{enumerate}
    \item For the Watson alternative \textsc{W}$(\theta, \kappa)$, with $\theta \in \Sp$ fixed:
    \begin{enumerate}
        \item
        \begin{equation*}
            \textsc{KL}(\kappa, 0) = D_d(\kappa) \kappa - \log d_d(\kappa) +  \log |\Sp|,
        \end{equation*}
        \item
        \begin{equation*}
            \gamma_\kappa(b) = \frac{|\Sp[d-2]|}{d_d(\kappa)} \sum_{l=0}^\infty \frac{\kappa^l}{l!} \sum_{j=0}^\beta c_{j, \, d}(\beta) \Delta_j(2l) P_j^d(\skalarprodukt{\theta}{b}) - \psi_d(\beta), \quad b \in \Sp,
        \end{equation*}
        \item
        \begin{equation*}
            \lim_{\kappa \to0^+} \frac{\max_{b \in \Sp} \gamma_\kappa^2(b)}{2 \textsc{KL}(\kappa, 0)} = \lambda_2 \nu_d(2).
        \end{equation*}
    \end{enumerate}
    \item For the Legendre polynomial alternative \textsc{LP}$_m(\theta, \kappa)$, with $m \in \N$ and $\theta \in \Sp$ fixed and $\kappa \in [0,1]$:
    \begin{enumerate}
        \item
        \begin{equation*}
            \textsc{KL}(\kappa, 0) = \frac{\kappa^2}{\nu_d(m)} \left( 1 - \frac{1}{2 (1 + \xi_\kappa)^2} \right) -  \frac{\kappa^3}{2(1+\xi_\kappa)^2} \frac{|\Sp[d-2]|}{|\Sp|} \int_{-1}^1 (P_m^d(t))^3 (1-t^2)^{(d-3)/2} \, \text{d}t,
        \end{equation*}
        with an intermediate point $\xi_\kappa$ satisfying $|\xi_\kappa| \leq \kappa$.
        \item
        \begin{equation*}
            \gamma_\kappa(b) = \kappa^2 \frac{\lambda_m}{\nu_d(m)} P_m^d(\skalarprodukt{\theta}{b}), \quad b \in \Sp,
        \end{equation*}
        \item
        \begin{equation*}
            \lim_{\kappa \to0^+} \frac{\max_{b \in \Sp} \gamma_\kappa^2(b)}{2 \textsc{KL}(\kappa, 0)} = \lambda_m \nu_m(2).
        \end{equation*}
    \end{enumerate}
\end{enumerate}
\end{prop}

\begin{proof}
\begin{enumerate}[label=\roman*)]
    \item
    By definition of the von Mises--Fisher distribution, see \eqref{von_mises_fisher_dichte}, we have for fixed $\theta \in \Sp$
    \begin{equation*}
        f(x \, | \, \kappa) = \frac{1}{a_d(\kappa)} \exp(\kappa \skalarprodukt{\theta}{x}), \quad x \in \Sp.
    \end{equation*}
    We compute for $\kappa > 0$
    \begin{align*}
        \text{KL}(\kappa, 0)
        &= \E_\kappa \left[ \log \left( \frac{f(U \, | \, \kappa)}{f(U \, | \, 0} \right) \right]
        = \int_{\Sp} \log \left( |\Sp| f(x \, | \, \kappa) \right) f(x \, | \, \kappa) \, d\sigma(x) \\[0.25cm]
        &= \frac{\kappa}{a_d(\kappa)} \int_{\Sp} \skalarprodukt{\theta}{x} \exp(\kappa \, \skalarprodukt{\theta}{x}) \, d\sigma(x) - \log a_d(\kappa) + \log |\Sp| \\[0.25cm]
        &= \kappa \frac{|\Sp[d-2]|}{a_d(\kappa)} \int_{\Sp} t e^{\kappa t} (1-t^2)^{(d-3)/2} \, dt - \log a_d(\kappa) + \log |\Sp|,
    \end{align*}
    where the last equality follows from the Funk--Hecke-Theorem for $P_0^d \equiv 1$. Moreover, we obtain by virtue of $\Gamma((d+1)/2) =  \Gamma((d-1)/2)(d-1)/2$ and an integration by parts
    \begin{align*}
        I_{d/2}(\kappa)
        &= \frac{(\kappa/2)^{d/2}}{\Gamma((d+1)/2) \Gamma(1/2)} \int_{-1}^1 e^{\kappa t} (1-t^2)^{(d-1)/2} \, dt \\[0.25cm]
        &= \frac{(\kappa/2)^{d/2-1}}{\Gamma((d-1)/2) \Gamma(1/2)} \int_{-1}^1 t e^{\kappa t} (1-t^2)^{(d-3)/2} \, dt,
    \end{align*}
    so that statement a) follows from
    \begin{align*}
        \text{KL}(\kappa, 0)
        &= \frac{|\Sp[d-2]|}{a_d(\kappa)} \frac{\Gamma((d-1)/2) \Gamma(1/2)}{(\kappa/2)^{d/2-1}} I_{d/2}(\kappa) \kappa - \log a_d(\kappa) + \log |\Sp| \\[0.25cm]
        &=  \frac{2 \pi^{d/2} \Gamma((d-1)/2)^{-1}}{2 \pi^{d/2} (\kappa/2)^{1-d/2} I_{d/2-1}(\kappa)} \frac{\Gamma((d-1)/2)}{(\kappa/2)^{d/2-1}} I_{d/2}(\kappa) \kappa - \log a_d(\kappa) + \log |\Sp| \\[0.25cm]
        &= \frac{I_{d/2}(\kappa)}{I_{d/2-1}(\kappa)} \kappa - \log a_d(\kappa) + \log |\Sp| \\[0.25cm]
        &= A_d(\kappa) \kappa - \log a_d(\kappa) +  \log |\Sp|.
    \end{align*}
    For statement b) we observe for $\kappa > 0$ and $b \in \Sp$
    \begin{align*}
        \E_\kappa (\skalarprodukt{b}{U})^{\, \beta}
        &= \frac{1}{a_d(\kappa)} \int_{\Sp} (\skalarprodukt{b}{x})^{\, \beta} \exp(\kappa \, \skalarprodukt{\theta}{x}) \, \text{d}\sigma(x) \\[0.25cm]
        &= \frac{1}{a_d(\kappa)} \sum_{j=0}^\beta c_{j,\, d}(\beta) \int_{\Sp} P_j^d(\skalarprodukt{b}{x}) \exp(\kappa \, \skalarprodukt{\theta}{x}) \, \text{d}\sigma(x),
    \end{align*}
    where the last equality is due to \eqref{potenzen_legendre}. The function $\Sp \ni x \mapsto P_j^d(\skalarprodukt{b}{x})$ is a $d$-dimensional spherical harmonic of order $j$, so that, once again, the Funk--Hecke-Theorem yields
    \begin{align*}
        \E_\kappa (\skalarprodukt{b}{U})^{\, \beta}
        &= \frac{|\Sp[d-2]|}{a_d(\kappa)} \sum_{j=0}^\beta c_{j,\, d}(\beta) P_j^d(\skalarprodukt{\theta}{b}) \int_{-1}^1 P_j^d(t) e^{\kappa t} (1-t^2)^{(d-3)/2} \, \text{d}t \\[0.25cm]
        &= \frac{|\Sp[d-2]|}{a_d(\kappa)} \sum_{j=0}^\beta c_{j,\, d}(\beta) P_j^d(\skalarprodukt{\theta}{b}) \sum_{l=0}^\infty \frac{\kappa^l}{l!} \Delta_j(l) \\[0.25cm]
        &= \frac{|\Sp[d-2]|}{a_d(\kappa)} \sum_{l=0}^\infty \frac{\kappa^l}{l!} \sum_{j=0}^\beta c_{j, \, d}(\beta) \Delta_j(l) P_j^d(\skalarprodukt{\theta}{b}).
    \end{align*}
    Since $\gamma_\kappa(b) = \E_\kappa (\skalarprodukt{b}{U})^{\, \beta} - \psi_d(\beta)$, formula b) follows. In particular, Proposition \ref{legendre_orthogonal} and the examples after Proposition \ref{c_j,d(m) allgemein} for $j = 1, \ldots , \beta$ yield
    \begin{align} \label{Delta_0_1_2}
        &\Delta_j(0)
        = \left\langle P_j^{\, d}, P_0^{\, d} \right\rangle
        = \delta_{j \, 0} \frac{|\Sp|}{\nu_d(0) \, |\Sp[d-2]|}
        = \delta_{j \, 0} \frac{|\Sp|}{|\Sp[d-2]|}, \\[0.25cm]
        &\Delta_j(1)
        = \left\langle P_j^{\, d}, P_1^{\, d} \right\rangle
        = \delta_{j \, 1} \frac{|\Sp|}{\nu_d(1) \, |\Sp[d-2]|}, \\[0.25cm]
        &\Delta_j(2)
        = \frac{1}{d} \left\langle P_j^{\, d}, P_0^{\, d} \right\rangle + \frac{d-1}{d} \left\langle P_j^{\, d}, P_2^{\, d} \right\rangle
        = \delta_{j \, 0} \frac{1}{d} \frac{|\Sp|}{|\Sp[d-2]|} + \delta_{j \, 2} \frac{d-1}{d} \frac{|\Sp|}{\nu_d(2) \, |\Sp[d-2]|}.
    \end{align}
    With the fact that $c_{0, \, d}(\beta) = \psi_d(\beta)$, see proof of Proposition \ref{K_beta_properties} i), a brief calculation gives
    \begin{equation} \label{gamma_kappa_split}
        \begin{aligned}
            \gamma_\kappa(b)
            &= \left( \frac{|\Sp|}{a_d(\kappa)} - 1 \right) \psi_d(\beta)
            + \kappa \frac{|\Sp|}{a_d(\kappa)} \frac{c_{1, \, d}(\beta)}{\nu_d(1)} P_1^d(\skalarprodukt{\theta}{b}) \\[0.25cm]
            &\hspace{0.5cm} + \frac{|\Sp[d-2]|}{a_d(\kappa)} \sum_{l=2}^\infty \frac{\kappa^l}{l!} \sum_{j=0}^\beta c_{j, \, d}(\beta) \Delta_j(l) P_j^d(\skalarprodukt{\theta}{b}).
        \end{aligned}
    \end{equation}
     In order to prove the last statement in i) we assume that $\kappa \in (0, R)$ for some $R>0$. This is no severe restrictions, since we are solely interested in the asymptotics near $0$. As the initial step we want to determine
    \begin{equation*}
        \limkappa \frac{\gamma_\kappa(b)}{\sqrt{2 \text{KL}(\kappa, 0)}}, \quad b \in \Sp.
    \end{equation*}
    For this purpose, we distinguish the cases $l=0$, $l=1$ and $l \geq 2$ as in equation \eqref{gamma_kappa_split}.

    Consider $l=0$ and define $\alpha_d(\kappa) = \left( \frac{|\Sp|}{a_d(\kappa)} - 1 \right)$. With \eqref{a_d_grenzwert} we realize that $\limkappa \alpha_d(\kappa) = 0$. With an application of de L'Hospital's rule, we calculate
    \begin{align*}
        \limkappa \frac{2 \text{KL}(\kappa,0)}{\alpha_d^2(\kappa)}
        &= \limkappa \frac{2 \left[ A_d(\kappa) \kappa -\log a_d(\kappa) + \log |\Sp| \right]}{\alpha_d^2(\kappa)}
        = \limkappa \frac{2 [ A_d'(\kappa)\kappa + A_d(\kappa) - A_d(\kappa) ]}{2 \alpha_d'(\kappa) \alpha_d(\kappa)} \\[0.25cm]
        &= \frac{1}{|\Sp|} \limkappa \frac{A_d'(\kappa) \kappa}{- \cfrac{ a_d'(\kappa)}{a_d^2(\kappa)}  \alpha_d(\kappa)} =
        \frac{1}{|\Sp|} \limkappa \frac{A_d'(\kappa) \kappa}{\cfrac{A_d(\kappa)}{a_d^2(\kappa)} \left[ a_d(\kappa) - |\Sp| \right]} \\[0.25cm]
        &= \frac{1}{|\Sp|} \limkappa \frac{a_d^2(\kappa) \left[ \cfrac{\kappa}{A_d(\kappa)} - A_d(\kappa) \kappa - d + 1 \right]}{ a_d(\kappa) - |\Sp| } = \infty.
    \end{align*}
    Here, the second and the second to last equality follow from \eqref{log_a_d_ableitung}, and the last equation follows from \eqref{A_d_ableitung}. Then, the enumerator converges to $|\Sp|^2$ due to $ \limkappa A_d(\kappa) \kappa = 0$ and with \eqref{A_d_kappa_klein} due to
    \begin{equation*}
        \limkappa \frac{A_d(\kappa)}{\kappa} = \limkappa \left( \frac{1}{d} - \frac{\kappa^2}{d^2 (d+2)} + O(\kappa^4) \right) = \frac{1}{d}.
    \end{equation*}
    On the other hand, the denominator converges to $0^+$ due to \eqref{a_d_grenzwert}. Hence, it follows for $l=0$
    \begin{equation*}
        \limkappa \frac{ \left( \cfrac{|\Sp|}{a_d(\kappa)} - 1 \right) \psi_d(\beta) }{\sqrt{2 \text{KL}(\kappa, 0)}}
        = - \limkappa \frac{\psi_d(\beta)}{\sqrt{\cfrac{2 \text{KL}(\kappa,0)}{\alpha_d^2(\kappa)}}} = 0.
    \end{equation*}
    Consider $l=1$. We have
    \begin{equation*}
        \limkappa \frac{ \kappa }{\sqrt{2 \text{KL}(\kappa, 0)}}
        = \limkappa \frac{1}{\sqrt{ \cfrac{2 \left[ A_d(\kappa) \kappa -\log a_d(\kappa) + \log |\Sp| \right]}{\kappa^2} }}
        = \sqrt{d}.
    \end{equation*}
    De L'Hospital's rule and \eqref{A_d_ableitung} yield
    \begin{align*}
        \limkappa \frac{2 \left[ A_d(\kappa) \kappa -\log a_d(\kappa) + \log |\Sp| \right]}{\kappa^2}
        &= \limkappa A_d'(\kappa) \\[0.25cm]
        &= \limkappa \left( 1 - A_d^2(\kappa) - \frac{d-1}{\kappa} A_d(\kappa) \right) = \frac{1}{d}.
    \end{align*}
    Together with $\nu_d(1) = d$ it follows that
    \begin{equation*}
        \limkappa \frac{ \kappa }{\sqrt{2 \text{KL}(\kappa, 0)}}  \frac{|\Sp|}{a_d(\kappa)} \frac{c_{1, \, d}(\beta)}{\nu_d(1)} P_1^d(\skalarprodukt{\theta}{b}) = \frac{c_{1, \, d}(\beta)}{\sqrt{\nu_d(1)}} P_1^d(\skalarprodukt{\theta}{b}), \quad b \in \Sp.
    \end{equation*}
    \\
    Consider $l \geq 2$. Then we immediately have
    \begin{equation*}
        \limkappa \frac{ \kappa^l }{\sqrt{2 \text{KL}(\kappa, 0)}} = \limkappa \frac{ \kappa }{\sqrt{2 \text{KL}(\kappa, 0)}} \kappa^{l-1} = 0,
    \end{equation*}
    because the first factor is bounded. For the following, we define the function series
    \begin{equation*}
        S_N: (0, R) \times \Sp \to \R; \quad
        (\kappa, b) \mapsto \sum_{l=2}^N \frac{1}{l!} \frac{\kappa^l}{\sqrt{2 \text{KL}(\kappa, 0)}} \sum_{j=0}^\beta c_{j, \, d}(\beta) \Delta_j(l) P_j^d(\skalarprodukt{\theta}{b})
    \end{equation*}
    for $N \in \N$, $N \geq 2$. For fixed $\kappa \in (0, R)$, the series is continuous on $\Sp$. Furthermore, it converges uniformly on $(0, R)$ as $N \to\infty$ and for fixed $b \in \Sp$, since for $\kappa \in (0, R)$ we have
    \begin{equation*}
        \frac{ \kappa^l }{\sqrt{2 \text{KL}(\kappa, 0)}} = \frac{ \kappa }{\sqrt{2 \text{KL}(\kappa, 0)}} \kappa^{l-1} \leq \max \{ C, R \}^l
    \end{equation*}
    with a constant $C > 0$ bounding $\kappa/\sqrt{2 \text{KL}(\kappa, 0)}$ from above. Due to $|\, P_j^d \,| \leq 1$, see Proposition \ref{legendre_schranke}, and $| \Delta_j(l) | \leq |\Sp| / |\Sp[d-2]|$ for $j = 1, \ldots, \beta$, $l \in \N_0$, it then follows
    \begin{align*}
        \left| \sum_{l=2}^\infty \frac{1}{l!} \frac{\kappa^l}{\sqrt{2 \text{KL}(\kappa, 0)}} \sum_{j=0}^\beta c_{j, \, d}(\beta) \Delta_j(l) P_j^d(\skalarprodukt{\theta}{b}) \right|
        \leq \beta \max_{j=1,\ldots,\beta} |c_{j, \, d}(\beta)| \frac{|\Sp|}{|\Sp[d-2]|} \sum_{l=2}^\infty \frac{\max \{ C, R \}^l}{l!},
    \end{align*}
    and the last series converges. The uniform convergence yields for $b \in \Sp$
    \begin{align*}
        \limkappa \frac{\gamma_\kappa(b)}{\sqrt{2 \text{KL}(\kappa, 0)}}
        &= \limkappa \left( \frac{ \left( \cfrac{|\Sp|}{a_d(\kappa)} - 1 \right) \psi_d(\beta) }{\sqrt{2 \text{KL}(\kappa, 0)}} + \frac{ \kappa }{\sqrt{2 \text{KL}(\kappa, 0)}}  \frac{|\Sp|}{a_d(\kappa)} \frac{c_{1, \, d}(\beta)}{\nu_d(1)} P_1^d(\skalarprodukt{\theta}{b}) \right) \\[0.25cm]
        &\hspace{0.5cm} + \limkappa \frac{|\Sp[d-2]|}{a_d(\kappa)} \sum_{l=2}^\infty \frac{1}{l!} \frac{\kappa^l}{\sqrt{2 \text{KL}(\kappa, 0)}} \sum_{j=0}^\beta c_{j, \, d}(\beta) \Delta_j(l) P_j^d(\skalarprodukt{\theta}{b}) \\[0.25cm]
        &= 0 + \frac{c_{1, \, d}(\beta)}{\sqrt{\nu_d(1)}} P_1^d(\skalarprodukt{\theta}{b}) + 0
        = \frac{c_{1, \, d}(\beta)}{\sqrt{\nu_d(1)}} P_1^d(\skalarprodukt{\theta}{b}).
    \end{align*}
    We may see here, especially, that $\gamma_\kappa(\cdot)/\sqrt{2 \text{KL}(\kappa, 0)}$ converges pointwise to $c_{1, \, d}(\beta)/\sqrt{\nu_d(1)} P_1^d(\skalarprodukt{\theta}{\cdot})$ for $\kappa \to0^+$. As a matter of fact, the convergence is even uniform on $\Sp$, because $|\, P_j^d \,| \leq 1$. Thus we can finally conclude
    \begin{align*}
        \limkappa \frac{\spheremax \gamma_\kappa^2(b)}{2 \textsc{KL}(\kappa, 0)}
        &= \limkappa \spheremax \left( \frac{\gamma_\kappa(b)}{\sqrt{2 \text{KL}(\kappa, 0)}} \right)^2
        = \spheremax \left( \limkappa \frac{\gamma_\kappa(b)}{\sqrt{2 \text{KL}(\kappa, 0)}} \right)^2 \\[0.25cm]
        &= \frac{(c_{1, \, d}(\beta))^2}{\nu_d(1)} \spheremax |\, P_1^d(\skalarprodukt{\theta}{b}) \,|^2 = \lambda_1 \nu_d(1).
    \end{align*}
    Here, the last equality is due to \eqref{eq:eigenvalues} and the fact that $\spheremax |\, P_1^d(\skalarprodukt{\theta}{b}) \,|^2 = 1$, see Remark \ref{legendre_schranke}.
    \item
    In the case of a Watson alternative, we are not going to perform such a detailed proof as before since the relevant quantities are quite similar to those in i) and, otherwise, many arguments would be repeated. Recall that the density of a Watson distribution with fixed parameter $\theta \in \Sp$ is given by
    \begin{equation*}
        f(x \, | \, \kappa) = \frac{1}{d_d(\kappa)} \exp(\kappa \skalarprodukt{\theta}{x}), \quad x \in \Sp.
    \end{equation*}
    Let $\kappa > 0$. First of all, the Kullback--Leibler information number satisfies
    \begin{align*}
        \text{KL}(\kappa, 0)
        = \kappa \frac{|\Sp[d-2]|}{d_d(\kappa)} \int_{\Sp} t^2 \exp(\kappa t^2) (1-t^2)^{(d-3)/2} \, \text{d}t - \log d_d(\kappa) + \log |\Sp|.
    \end{align*}
    By the definition of the Kummer function
    \begin{align*}
        M(3/2, d/2 + 1, \kappa)
        &= \frac{1}{B(3/2, (d-1)/2)} \int_{-1}^1 t^2 e^{\kappa t^2} (1-t^2)^{(d-3)/2} \, \text{d}t \\[0.25cm]
        &= \frac{d \, \Gamma(d/2)}{\sqrt{\pi} \, \Gamma((d-1)/2)} \int_{-1}^1 t^2 e^{\kappa t^2} (1-t^2)^{(d-3)/2} \, \text{d}t,
    \end{align*}
    so that after a short calculation it follows that
    \begin{align*}
        \text{KL}(\kappa, 0)
        &= \kappa \frac{|\Sp[d-2]|}{d_d(\kappa)} \frac{\sqrt{\pi} \, \Gamma((d-1)/2)}{d \, \Gamma(d/2)} M(3/2, d/2 + 1, \kappa) - \log d_d(\kappa) + \log |\Sp| \\[0.25cm]
        &= D_d(\kappa) \kappa - \log d_d(\kappa) +  \log |\Sp|.
    \end{align*}
    For $\gamma_\kappa$ the exact same arguments as in i) yield the formula
    \begin{equation*}
        \gamma_\kappa(b) = \frac{|\Sp[d-2]|}{d_d(\kappa)} \sum_{l=0}^\infty \frac{\kappa^l}{l!} \sum_{j=0}^\beta c_{j, \, d}(\beta) \Delta_j(2l) P_j^d(\skalarprodukt{\theta}{b}) - \psi_d(\beta), \quad b \in \Sp.
    \end{equation*}
    Therefore,
    \begin{equation*}
        \begin{aligned}
            \gamma_\kappa(b)
            &= \left( \frac{|\Sp|}{d_d(\kappa)} - 1 \right) \psi_d(\beta)
            + \kappa \frac{|\Sp[d-2]|}{d_d(\kappa)} \sum_{j=0}^\beta c_{j, \, d}(\beta) \Delta_j(2) P_j^d(\skalarprodukt{\theta}{b}) \\[0.25cm]
            &\hspace{0.5cm} + \frac{|\Sp[d-2]|}{d_d(\kappa)} \sum_{l=2}^\infty \frac{\kappa^l}{l!} \sum_{j=0}^\beta c_{j, \, d}(\beta) \Delta_j(2l) P_j^d(\skalarprodukt{\theta}{b}) \\[0.25cm]
            &=
            \left( \frac{|\Sp|}{d_d(\kappa)} - 1 \right) \psi_d(\beta)
            + \kappa \frac{|\Sp|}{d_d(\kappa)} \left[ \frac{\psi_d(\beta)}{d} + \frac{c_{2, \, d}(\beta)}{\nu_d(2)} \frac{d-1}{d} P_2^d(\skalarprodukt{\theta}{b}) \right] \\[0.25cm]
            &\hspace{0.5cm} + \frac{|\Sp[d-2]|}{d_d(\kappa)} \sum_{l=2}^\infty \frac{\kappa^l}{l!} \sum_{j=0}^\beta c_{j, \, d}(\beta) \Delta_j(2l) P_j^d(\skalarprodukt{\theta}{b}).
        \end{aligned}
    \end{equation*}
    The last equality is due to \eqref{Delta_0_1_2}. Again we distinguish the cases $l=0$, $l=1$ and $l \geq 2$ and assume that $\kappa > 0$ is sufficiently small. Consider $l=0$ and define $\delta_d(\kappa) = \left( \frac{|\Sp|}{d_d(\kappa)} - 1 \right)$. Then $\delta_d(\kappa) = \left( \frac{1}{M(1/2, d/2, \kappa)} - 1 \right) < 0$ by \eqref{d_d_negativ}. Thus we calculate
    \begingroup
    \allowdisplaybreaks
    \begin{align*}
        \limkappa \frac{2 \text{KL}(\kappa,0)}{\left(1 - \cfrac{1}{M(1/2, d/2, \kappa)} \right)^2}
        &= \limkappa \frac{2 \left[ D_d(\kappa) \kappa -\log d_d(\kappa) + \log |\Sp| \right]}{\left(1 - \cfrac{1}{M(1/2, d/2, \kappa)} \right)^2} \\[0.25cm]
        &= \limkappa \frac{D_d'(\kappa) \kappa}{\cfrac{M'(1/2, d/2, \kappa)}{M^2(1/2, d/2, \kappa)} \cfrac{M(1/2, d/2, \kappa) - 1}{M(1/2, d/2, \kappa)}} \\[0.25cm]
        &= \limkappa M^2(1/2, d/2, \kappa) \frac{D_d'(\kappa) \kappa}{D_d(\kappa) (M(1/2, d/2, \kappa) - 1)} \\[0.25cm]
        &= \limkappa \frac{D_d'(\kappa)}{D_d(\kappa)} \limkappa \frac{\kappa}{M(1/2, d/2, \kappa) - 1} \\[0.25cm]
        &= \limkappa \frac{D_d'(\kappa)}{D_d(\kappa)} \limkappa \frac{1}{M'(1/2, d/2, \kappa)}.
    \end{align*}
    \endgroup
    Here, the second and the last equality are due to de L'Hospital's rule, and the third and fourth equality follow from \eqref{log_d_d_ableitung} and \eqref{kummer_grenzwert}, respectively. By \eqref{kummer_ableitung}, \eqref{kummer_grenzwert}, \eqref{D_d_kappa_klein} and \eqref{D_d_ableitung_grenzwert}, it follows that
    \begin{equation*}
        \limkappa \frac{ \left( \cfrac{|\Sp|}{d_d(\kappa)} - 1 \right) \psi_d(\beta) }{\sqrt{2 \text{KL}(\kappa, 0)}}
        = - \limkappa \frac{ \psi_d(\beta) }{\sqrt{\cfrac{2 \text{KL}(\kappa,0)}{\left(1 - \cfrac{1}{M(1/2, d/2, \kappa)} \right)^2}}}
        = - \frac{\psi_d(\beta)}{\sqrt{2}} \sqrt{\frac{d+2}{d-1}}.
    \end{equation*}
    If $l=1$, then
    \begin{equation*}
        \limkappa \frac{ \kappa }{\sqrt{2 \text{KL}(\kappa, 0)}}
        = \limkappa \frac{1}{\sqrt{ \cfrac{2 \left[ D_d(\kappa) \kappa -\log d_d(\kappa) + \log |\Sp| \right]}{\kappa^2} }}
        = \frac{d}{\sqrt{2}} \sqrt{\frac{d+2}{d-1}},
    \end{equation*}
    since
    \begin{equation*}
        \limkappa \frac{2 \left[ D_d(\kappa) \kappa -\log a_d(\kappa) + \log |\Sp| \right]}{\kappa^2}
        = \limkappa D_d'(\kappa)
        = \frac{2(d-1)}{d^2(d+2)},
    \end{equation*}
    where the last equation is due to \eqref{D_d_ableitung_grenzwert}.
    The expressions for $l \geq 2$ vanish for $\kappa \to0^+$ by the same argument as in i).

    These results yield for $b \in \Sp$
    \begingroup
    \allowdisplaybreaks
    \begin{align*}
        \limkappa \frac{\gamma_\kappa(b)}{\sqrt{2 \text{KL}(\kappa, 0)}}
        &= \limkappa \frac{ \left( \cfrac{|\Sp|}{d_d(\kappa)} - 1 \right) \psi_d(\beta) }{\sqrt{2 \text{KL}(\kappa, 0)}} \\[0.25cm]
        &\hspace{0.5cm} + \limkappa \frac{ \kappa }{\sqrt{2 \text{KL}(\kappa, 0)}}  \frac{|\Sp|}{d_d(\kappa)} \left[ \frac{\psi_d(\beta)}{d} + \frac{c_{2, \, d}(\beta)}{\nu_d(2)} \frac{d-1}{d} P_2^d(\skalarprodukt{\theta}{b}) \right]  \\[0.25cm]
        &\hspace{0.5cm} + \limkappa \frac{|\Sp[d-2]|}{d_d(\kappa)} \sum_{l=2}^\infty \frac{1}{l!} \frac{\kappa^l}{\sqrt{2 \text{KL}(\kappa, 0)}} \sum_{j=0}^\beta c_{j, \, d}(\beta) \Delta_j(2l) P_j^d(\skalarprodukt{\theta}{b}) \\[0.25cm]
        &= - \frac{\psi_d(\beta)}{\sqrt{2}} \sqrt{\frac{d+2}{d-1}} + \frac{d}{\sqrt{2}} \sqrt{\frac{d+2}{d-1}} \left[ \frac{\psi_d(\beta)}{d} + \frac{c_{2, \, d}(\beta)}{\nu_d(2)} \frac{d-1}{d} P_2^d(\skalarprodukt{\theta}{b}) \right] + 0 \\[0.25cm]
        &= \sqrt{\frac{(d+2)(d-1)}{2}} \frac{c_{2, \, d}(\beta)}{\nu_d(2)} P_2^d(\skalarprodukt{\theta}{b})
        = \frac{c_{2, \, d}(\beta)}{\sqrt{\nu_d(2)}} P_2^d(\skalarprodukt{\theta}{b}).
    \end{align*}
    \endgroup
    Using \eqref{eq:eigenvalues} once again, it follows that
     \begin{align*}
        \limkappa \frac{\spheremax \gamma_\kappa^2(b)}{2 \textsc{KL}(\kappa, 0)}
        &= \limkappa \spheremax \left( \frac{\gamma_\kappa(b)}{\sqrt{2 \text{KL}(\kappa, 0)}} \right)^2
        = \spheremax \left( \limkappa \frac{\gamma_\kappa(b)}{\sqrt{2 \text{KL}(\kappa, 0)}} \right)^2 \\[0.25cm]
        &= \frac{(c_{2, \, d}(\beta))^2}{\nu_d(2)} \spheremax |\, P_2^d(\skalarprodukt{\theta}{b}) \,|^2 = \lambda_2 \nu_d(2).
    \end{align*}
    \item
    Recall the density
    \begin{equation*}
        f(x \, | \, \kappa) = \frac{1}{|\Sp|} \left( 1 + \kappa P_m^d(\skalarprodukt{\theta}{x}) \right), \quad x \in \Sp,
    \end{equation*}
     of the Legendre polynomial alternative with fixed $m \in \N$ and $\theta \in \Sp$.
    We calculate for $\kappa \in [0,1]$
    \begin{align*}
        \text{KL}(\kappa, 0)
        &= \frac{1}{|\Sp|} \int_{\Sp} \log \left( 1 + \kappa P_m^d(\skalarprodukt{\theta}{x}) \right) \left( 1 + \kappa P_m^d(\skalarprodukt{\theta}{x}) \right) \, \text{d}\sigma(x) \\[0.25cm]
        &= \frac{|\Sp[d-2]|}{|\Sp|} \int_{-1}^1 \log \left( 1 + \kappa P_m^d(t) \right) \left( 1 + \kappa P_m^d(t) \right) (1-t^2)^{(d-3)/2} \, \text{d}t.
    \end{align*}
    A Taylor expansion of order 1 of $t \mapsto \log(1+t)$ around $0$ yields
    \begin{align*}
        \text{KL}(\kappa, 0)
        &= \frac{|\Sp[d-2]|}{|\Sp|} \int_{-1}^1 \left( \kappa P_m^d(t) - \frac{\left( \kappa P_m^d(t) \right)^2}{2 (1 + \xi_\kappa)^2} \right) \left( 1 + \kappa P_m^d(t) \right) (1-t^2)^{(d-3)/2} \, \text{d}t \\[0.25cm]
        &= \frac{|\Sp[d-2]|}{|\Sp|} \left[ \kappa \left\langle P_m^{\, d}, P_0^{\, d} \right\rangle + \kappa^2 \left\langle P_m^{\, d}, P_m^{\, d} \right\rangle - \frac{\kappa^2}{2 (1 + \xi_\kappa)^2} \left\langle P_m^{\, d}, P_m^{\, d} \right\rangle - \frac{\kappa^3}{2 (1 + \xi_\kappa)^2} \left\langle \left( P_m^{\, d} \right)^2, P_m^{\, d} \right\rangle  \right] \\[0.25cm]
        &= \frac{\kappa^2}{\nu_d(m)} \left( 1 - \frac{1}{2 (1 + \xi_\kappa)^2} \right) -  \frac{\kappa^3}{2(1+\xi_\kappa)^2} \frac{|\Sp[d-2]|}{|\Sp|} \int_{-1}^1 \left( P_m^{\, d}(t) \right)^3 (1-t^2)^{(d-3)/2} \, \text{d}t.
    \end{align*}
    Here, $\xi_\kappa$ is an intermediate point, which satisfies $|\xi_\kappa| \leq \kappa |\, P_m^d(t) \,| \leq \kappa$ for $t \in [-1,1]$.
    Therefore, we have
    \begin{equation*}
        \text{KL}(\kappa, 0) = \frac{\kappa^2}{\nu_d(m)} \left( 1 - \frac{1}{2 (1 + \xi_\kappa)^2} \right) + O(\kappa^3), \quad \kappa \to0^+.
    \end{equation*}
    For $\gamma_\kappa$ we compute for $b \in \Sp$
    \begingroup
    \allowdisplaybreaks
    \begin{align*}
        \E_\kappa (\skalarprodukt{b}{U})^{\, \beta}
        &= \frac{1}{|\Sp|} \int_{\Sp} (\skalarprodukt{b}{x})^{\, \beta} \left( 1 + \kappa P_m^d(\skalarprodukt{\theta}{x}) \right) \, d\sigma(x) \\[0.25cm]
        &= \psi_d(\beta) + \kappa P_m^d(\skalarprodukt{\theta}{b}) \frac{|\Sp[d-2]|}{|\Sp|} \int_{-1}^1 P_m^{\, d}(t) t^{\, \beta} (1-t^2)^{(d-3)/2} \, dt \\[0.25cm]
        &= \psi_d(\beta) + \kappa P_m^d(\skalarprodukt{\theta}{b}) \frac{|\Sp[d-2]|}{|\Sp|} \sum_{j=0}^\beta c_{j, \, d}(\beta) \left\langle P_j^{\, d}, P_m^{\, d} \right\rangle \\[0.25cm]
        &= \psi_d(\beta) + \kappa \frac{c_{m, \, d}(\beta)}{\nu_d(m)} P_m^d(\skalarprodukt{\theta}{b}),
    \end{align*}
    \endgroup
    where the second and third equality are due to \eqref{eq:ewert} and \eqref{potenzen_legendre}, respectively.
    Hence,
    \begin{equation*}
        \spheremax \gamma_\kappa^2(b) = \kappa^2 \lambda_m \spheremax |\, P_m^d(\skalarprodukt{\theta}{b}) \,|^2 = \kappa^2 \lambda_m.
    \end{equation*}
    In view of these results and the fact, that $\limkappa \xi_\kappa = 0$, we conclude
    \begin{align*}
        \limkappa \frac{\spheremax \gamma_\kappa^2(b)}{2 \textsc{KL}(\kappa, 0)}
        &= \limkappa \frac{\kappa^2 \lambda_m}{\cfrac{2 \kappa^2}{\nu_d(m)} \left( 1 - \cfrac{1}{2 (1 + \xi_\kappa)^2} \right) + O(\kappa^3)} \\[0.25cm]
        &= \limkappa \cfrac{\lambda_m}{\cfrac{2}{\nu_d(m)} \left( 1 - \cfrac{1}{2 (1 + \xi_\kappa)^2} \right) + O(\kappa)}
        = \frac{\lambda_m}{\cfrac{2}{\nu_d(m)} \left( 1 - \cfrac{1}{2} \right)}
        = \lambda_m \nu_d(m).
    \end{align*}
\end{enumerate}
\end{proof}

\section{Proofs of main results} \label{sec:AppC}

\textbf{Proof of Theorem \ref{thm:asy_emp}.}

\begin{proof}
The proof uses the methods presented in the proof of Theorem 2.1 in \cite{BH:1991}, although the covariance structure is a generalization. Putting $W(b)=(b^\top U_1)^\beta-\psi_d(\beta),\,b\in\Sp,$ we have (using $x^\beta-y^\beta=(x-y)\sum_{j=0}^{\beta-1}x^jy^{\beta-1-j}$, $x,y\in\R$, $\beta\in\N$) with the Cauchy--Schwarz inequality
\begin{equation*}
|W(b)-W(c)|\le \beta \|b-c\|,\quad b,c\in\Sp,
\end{equation*}
and a direct application of the CLT in Banach spaces, see \cite{AG:1980}, Corollary 7.17, yields the claim. The Corollary is applicable, since the metric space $\left(\Sp,\|\cdot\|\right)$ clearly satisfies the stated entropy condition. To finish the proof it suffices to calculate the covariance structure $E(W(b)W(c))$ by taking advantage of Lemma \ref{lem:mean}.
\end{proof}

\newpage
\noindent
\textbf{Proof of Proposition \ref{K_beta_properties}.}

\begin{proof}
\begin{enumerate}[label=\roman*)]
    \item Let $\beta \in \N$ and $k \in \N_0$. Since $\rho_\beta(b,c) = Q(\skalarprodukt{b}{c})$ for all $b,c \in \Sp$, it follows by the Funk--Hecke-Theorem \ref{funk-hecke} for $x \in \Sp$ and $\phi \in \mathcal{H}_k(\Sp)$
    \begin{align*}
        K_\beta \phi(x)
        = \frac{1}{|\Sp|} \int_{\Sp} \rho_\beta(\omega, x) \phi(\omega) \, \text{d}\sigma(\omega)
        = \frac{1}{|\Sp|} \int_{\Sp} Q(\skalarprodukt{\omega}{x}) \phi(\omega) \, \text{d}\sigma(\omega)
        = \lambda_k \phi(x)
    \end{align*}
    with
    \begin{equation*}
        \lambda_k = \frac{|\Sp[d-2]|}{|\Sp|} \int_{-1}^1 P_k^{\, d}(t) Q(t) (1-t^2)^{(d-3)/2} dt,
    \end{equation*}
    where the expression for $\lambda_k$ is derived from the Funk--Hecke-Theorem. By Proposition \ref{legendre_orthogonal} we have the following orthogonality property of the Legendre polynomials w.r.t the scalar product in \eqref{legendre_scalarproduct}
    \begin{equation*}
        \left\langle P_k^d, P_j^d \right\rangle = \delta_{kj} \frac{|\Sp|}{\nu_d(k) \, |\Sp[d-2]|}, \quad \forall k, j \in \N_0.
    \end{equation*}
    In view of \eqref{eta} and \eqref{eq:covk} we have
    \begin{equation*}
        Q(t) = \sum_{j=0}^\beta \frac{(c_{j,\, d}(\beta))^2}{\nu_d(j)} P_j^{\, d}(t) - \psi_d^2(\beta), \quad t \in [-1,1].
    \end{equation*}
    We calculate for $k \leq \beta$ with $P_0^d\equiv 1$
    \begingroup
    \allowdisplaybreaks
    \begin{align*}
        \int_{-1}^1 P_k^{\, d}(t) Q(t) (1-t^2)^{(d-3)/2}
        &= \sum_{j=0}^\beta \frac{(c_{j,\, d}(\beta))^2}{\nu_d(j)} \left\langle P_k^d, P_j^d \right\rangle
        - \int_{-1}^1 P_k^{\, d}(t) \psi_d^2(\beta) (1-t^2)^{(d-3)/2} \text{d}t \\[0.25cm]
        &= \frac{|\Sp|}{|\Sp[d-2]|} \sum_{j=0}^\beta \left( \frac{c_{j,\, d}(\beta)}{\nu_d(j)} \right)^2 \delta_{kj}
        - \psi_d^2(\beta) \left\langle P_k^d, P_0^d \right\rangle \\[0.25cm]
        &= \frac{|\Sp|}{|\Sp[d-2]|} \left[ \left( \frac{c_{k,\, d}(\beta)}{\nu_d(k)} \right)^2 - \delta_{k0} \psi_d^2(\beta) \right].
    \end{align*}
    \endgroup
    Thereby we obtain
    \begin{equation*}
        \lambda_k
        = \frac{|\Sp[d-2]|}{|\Sp|} \int_{-1}^1 P_k^{\, d}(t) Q(t) (1-t^2)^{(d-3)/2} \text{d}t
        = \left( \frac{c_{k,\, d}(\beta)}{\nu_d(k)} \right)^2 - \delta_{k0} \psi_d^2(\beta), \quad k \leq \beta.
    \end{equation*}
    Evidently is $\lambda_k = 0$ for $k > \beta$ by the orthogonality of the Legendre polynomials. Consider the case $k=0$. With $\nu_d(0)=1$ we already know that
    \begin{equation*}
        \lambda_k = (c_{0,\, d}(\beta))^2 - \psi_d^2(\beta) = (c_{0,\, d}(\beta) - \psi_d(\beta))(c_{0,\, d}(\beta) + \psi_d(\beta)).
    \end{equation*}
    If $\beta$ is odd, then $\psi_d(\beta) = c_{0,\, d}(\beta) = 0$ due to \eqref{eta} and Proposition \ref{c_j,d(m) allgemein}. Thus let $\beta$ be even. Then the Funk--Hecke-Theorem for $1 \in \mathcal{H}_0(\Sp)$, equation \eqref{potenzen_legendre}, and the orthogonality of the Legendre polynomials yield
    \begin{align*}
        \psi_d(\beta)
        &= \E(\skalarprodukt{b}{U})^{\, \beta}
        = \frac{1}{|\Sp|} \int_{\Sp} (\skalarprodukt{b}{\omega})^{\, \beta}  \, \text{d}\sigma(\omega)
        = \frac{|\Sp[d-2]|}{|\Sp|} \int_{-1}^1 t^{\, \beta} (1-t^2)^{(d-3)/2} \, \text{d}t \\[0.25cm]
        &= \frac{|\Sp[d-2]|}{|\Sp|} \sum_{j=0}^\beta c_{j,\, d}(\beta) \int_{-1}^1 P_j^{\, d}(t) (1-t^2)^{(d-3)/2} \, \text{d}t
        = \frac{|\Sp[d-2]|}{|\Sp|} \sum_{j=0}^\beta c_{j,\, d}(\beta) \left\langle P_j^d, P_0^d \right\rangle \\[0.25cm]
        &= \frac{|\Sp[d-2]|}{|\Sp|} \sum_{j=0}^\beta c_{j,\, d}(\beta) \frac{\delta_{j \, 0} |\Sp|}{\nu_d(0) \, |\Sp[d-2]|} = c_{0,\, d}(\beta).
    \end{align*}
    \item
    Let $\{\phi_{k,1},\ldots,\phi_{k,\nu_d(k)}\}$ be an orthonormal basis of $\harmonicspace$. Then $\bigcup_{k=0}^\infty \{\phi_{k,1},\ldots,\phi_{k,\nu_d(k)}\}$ is an orthonormal basis of $L^2(\Sp, \text{d}\sigma)$ and we get the unique series expansion for $f \in L^2(\Sp, \text{d}\sigma)$
    \begin{equation} \label{f_harmonic_expansion}
        f \overset{L^2}{=} \sum_{k=0}^\infty \Psi_k,
    \end{equation}
    with $\Psi_k = \sum_{j=1}^{\nu_d(k)} \left\langle f, \phi_{k,j} \right\rangle_{L^2} \phi_{k,j} \in \harmonicspace$, compare with \eqref{L^2_harmonic_expansion}.
    Set $a_{k, j} = \left\langle f, \phi_{k,j} \right\rangle_{L^2}$ for all $j \in \{ 1,\ldots,\nu_d(k) \}$, $k \in \N_0$. It follows for $x \in \Sp$
    \begin{align*}
        K_\beta f(x)
        &= \frac{1}{|\Sp|} \int_{\Sp} \rho_\beta(\omega, x) f(\omega) \, \text{d}\sigma(\omega)
        = \frac{1}{|\Sp|} \left\langle \, \rho_{\beta}(\cdot, x), f \right\rangle_{L^2} \\[0.25cm]
        &= \frac{1}{|\Sp|} \sum_{k=0}^\infty \sum_{j=1}^{\nu_d(k)} a_{k,j} \left\langle \, \rho_{\beta}(\cdot, x), \phi_{k,j} \right\rangle_{L^2}
        \overset{i)}{=} \sum_{k=0}^\infty \sum_{j=1}^{\nu_d(k)} a_{k,j} \lambda_k \phi_{k,j}(x) \mathds{1}_{\{k \, \leq \, \beta\}} \\[0.25cm]
        &= \sum_{k=1}^\beta \lambda_k \sum_{j=1}^{\nu_d(k)} a_{k,j} \phi_{k,j}(x).
    \end{align*}
    Hence, $K_\beta$ is a finite-rank operator, and it has the representation
    \begin{equation} \label{K_beta_finite_rank}
        K_\beta f(x) = \sum_{k=1}^\beta \lambda_k \sum_{j=1}^{\nu_d(k)} a_{k,j} \phi_{k,j}(x), \quad x \in \Sp.
    \end{equation}
    \item
    Let $\lambda \in \C$, $ \lambda \notin \{ 0 \} \cup \{ \lambda_k \; | \; k = 1,\ldots,\beta \}$, and consider the equation $\lambda f - K_\beta f = 0$ für $f \in L^2(\Sp, \text{d}\sigma)$ as in \eqref{f_harmonic_expansion}. Observe for $i \in \{ 1,\ldots,\nu_d(m) \}$, $m \in \{ 1,\ldots,\beta \}$
    \begin{align*}
        \lambda \left\langle f, \phi_{m,i} \right\rangle_{L^2}
        &= \left\langle K_\beta f, \phi_{m,i} \right\rangle_{L^2}
        = \sum_{k=1}^\beta \lambda_k \sum_{j=1}^{\nu_d(k)} a_{k,j} \left\langle \phi_{k,j}, \phi_{m,i} \right\rangle_{L^2} \\
        &= \sum_{k=1}^\beta \lambda_k \sum_{j=1}^{\nu_d(k)} a_{k,j} \delta_{mk} \delta_{ij}
        = \sum_{k=1}^\beta \lambda_k \delta_{mk} a_{k,i}
        = \lambda_m a_{m,i} = \lambda_m \left\langle f, \phi_{m,i} \right\rangle_{L^2}.
    \end{align*}
    Since $\lambda \neq \lambda_m$, it immediately follows that $\left\langle f, \phi_{m,i} \right\rangle_{L^2} = 0$. Hence, $\left\langle f, \phi_{m,i} \right\rangle_{L^2} = 0$ for all $i \in \{ 1,\ldots,\nu_d(m) \}$, $m \in \{ 1,\ldots,\beta \}$ and therewith $K_\beta f = 0$. We conclude $\lambda f = 0$ and, since $\lambda \neq 0$, it follows that $f = 0$. This, in turn, means that $\lambda$ is a resolvent point.
    \item With $f \in L^2(\Sp, \text{d}\sigma)$ like in \eqref{f_harmonic_expansion} and by \eqref{K_beta_finite_rank} we conclude
    \begin{equation*}
        \skalarproduktLebesgue{K_\beta f}{f}
        = \sum_{k=1}^\beta \lambda_k \sum_{j=1}^{\nu_d(k)} a_{k,j} \skalarproduktLebesgue{\phi_{k,j}}{f}
        = \sum_{k=1}^\beta \lambda_k \sum_{j=1}^{\nu_d(k)} a_{k,j}^2 \geq 0,
    \end{equation*}
    since all eigenvalues are non-negative.
\end{enumerate}
\end{proof}

\noindent
\textbf{Proof of Proposition \ref{pro:AltD}.}

\begin{proof}
We prove the claim only for the case of $\beta$ odd as the other case can be done analogously. Since a centred Gaussian process is defined by the covariance kernel, we merely have to show that the covariance kernels of $Z_\beta(\cdot)$ and the process on the right-hand side, denoted by $Y(\cdot)$, coincide. Clearly, after some calculations we first see that $Y(\cdot)$ is a centred Gaussian process. Furthermore,
\begin{align*}
        \E Y(b) Y(c)
        &= |\Sp| \sum_{\substack{ k,m=1 \\ k,m \textrm{ odd}}}^\beta \sqrt{\lambda_k \lambda_m} \sum_{j=1}^{\nu_d(k)} \sum_{i=1}^{\nu_d(m)} \phi_{k,j}(b) \, \phi_{m,i}(c) \, \E N_{k,j} N_{m,i} \\
        &= |\Sp| \sum_{\substack{ k,m=1 \\ k,m \textrm{ odd}}}^\beta \sqrt{\lambda_k \lambda_m} \sum_{j=1}^{\nu_d(k)} \sum_{i=1}^{\nu_d(m)} \phi_{k,j}(b) \, \phi_{m,i}(c) \, \delta_{km} \delta_{ij} \\
        &= |\Sp| \sum_{\substack{ k=1 \\ k \textrm{ odd}}}^\beta \lambda_k \sum_{j=1}^{\nu_d(k)} \phi_{k,j}(b) \, \phi_{k,j}(c)
        = \rho_\beta(b,c),
\end{align*}
where the last equality follows from Mercer's theorem, see \cite{SCH:2010}, Theorem 2.10, which is applicable due to our previously obtained results. Note that the eigenvalues in \eqref{eq:eigenvalues} equal zero for even indices, if $\beta$ is odd, for the occurring constants $c_{k,\, d}(\beta)$ are zero in these cases. Compare also with Proposition \ref{c_j,d(m) allgemein}.
\end{proof}

\noindent
\textbf{Proof of Theorem \ref{thm:ben_alt}.}

\begin{proof}
Lemma \ref{lem:DQVK} yields by Le Cams first Lemma, see \cite{LV:2017} Proposition 5.2.1, that $\mathbb{P}^{(n)}$ and $\mathbb{A}^{(n)}$ are mutually contiguous. Straightforward calculations show for $b\in\Sp$ under $\mathbb{P}^{(n)}$
\begin{enumerate}[label=\roman*)]
\item $\displaystyle\lim_{n\rightarrow\infty}\mbox{Cov}\left(Z_{n,\beta}(b),\log L_n\right)=S^*_\beta(b)$,

\item for $l\in\N$, $a_1,\ldots,a_l\in\Sp$ the distribution of $(Z_{n,\beta}(a_1),\ldots,Z_{n,\beta}(a_l),\log L_n)^\top$ converges to
    \begin{equation*}
    \mbox{N}_{l+1}\left(\left(0,\ldots,0,-\frac{\tau^2}{2}\right)^\top,\left(\begin{array}{cc}\Sigma_{a_1,\ldots,a_l} & z \\ z^\top & \tau^2 \end{array}\right)\right),
    \end{equation*}
    where $\Sigma_{a_1,\ldots,a_l}=\left(\rho_\beta\left(a_{j_1},a_{j_2}\right)\right)_{1\le j_1,j_2\le l}$ is a $l\times l$-matrix,  $z=\left(S^*_\beta(a_1),\ldots,S^*_\beta(a_l)\right)^\top$ is a $l$-dimensional vector, and $\tau$ is defined in Lemma \ref{lem:DQVK}.
\end{enumerate}
Le Cam's third lemma shows that, under $\mathbb{A}^{(n)}$ the finite-dimensional distributions of $Z_{n,\beta}$ converge weakly to the corresponding distributions of the shifted Gaussian process $Z_\beta+S^*_\beta$. Since $Z_{n,\beta}$ is tight under $\mathbb{P}^{(n)}$ and $\mathbb{A}^{(n)}$ is contiguous to $\mathbb{P}^{(n)}$ we have tightness of $Z_{n,\beta}$ and $Z_{n,\beta}\cd Z_\beta+S^*_\beta$ under $\mathbb{A}^{(n)}$.
\end{proof}

\noindent
\textbf{Proof of Theorem \ref{T_n,beta_bahadur_steigung}.}

\begin{proof}
For the first statement, we use \cite{BAH:1971}, Theorem 7.2, and verify the two conditions therein for the case of the approximate Bahadur slope, compare with \cite{BAH:1967}, Section 6. First of all, due to the strong law of large numbers we have
\begin{equation*}
    \frac{1}{n} \standardsumme (\skalarprodukt{b}{U_j})^{\, \beta} - \psi_d(\beta) \overset{\mathcal{\Pb_\kappa} \textrm{-f.s.}}{\rightarrow} \gamma_\kappa(b), \quad b \in \Sp.
\end{equation*}
Hence we immediately obtain
\begin{equation*}
    \frac{\widetilde{T}_{n, \, \beta}}{\sqrt{n}} \overset{\mathcal{\Pb_\kappa}}{\rightarrow} \spheremax \gamma_\kappa^2(b), \quad \kappa > 0.
\end{equation*}
Under $H_0$, Corollary \ref{cor:H0} yields
\begin{equation*}
    \widetilde{T}_{n, \, \beta} = \sqrt{T_{n, \, \beta}} \weakconv \spheremax | \, Z_\beta(b) \, |.
\end{equation*}
With $F(t) = \Pb( \, \max_{b \in \Sp} | \, Z_\beta(b) \, | < t)$, $t \in \R$, an application of \cite{LT:2002}, Corollary 3.2, gives
\begin{align*}
     \lim_{t \to\infty} \frac{\log(1-F(t))}{t^2}
     &= \lim_{t \to\infty} \frac{\log \Pb\left( \, \max_{b \in \Sp} | \, Z_\beta(b) \, | \geq t \right)}{t^2} \\
     &= \lim_{t \to\infty} \frac{\log \Pb\left( \, \| \, Z_\beta(\cdot) \, \|_\infty \geq t \right)}{t^2} = - \frac{1}{2 \spheremax \rho_\beta(b,b)}.
\end{align*}
Thus, the approximate Bahadur slope is
\begin{equation*}
    c_{\widetilde{T}_\beta}^{\; a}(\kappa)
    = \frac{\spheremax \gamma_\kappa^2(b)}{\spheremax \rho_\beta(b,b)}, \quad \kappa > 0.
\end{equation*}
For the second statement, we utilize again Theorem 7.2 in \cite{BAH:1971} and prove the second condition for sufficiently small $\kappa > 0$ with the help of \cite{RAO:1972}, Lemma 2.1 and Lemma 2.2. Let us consider, once again, the $C(\Sp, \R)$-valued random elements $W_j(b) = (\skalarprodukt{b}{U_j})^{\, \beta} - \psi_d(\beta)$, $b \in \Sp$, which are centered under $H_0$ and due to their compact support fulfill the integrability condition of \cite{RAO:1972}, Lemma 2.1. Hence, with $F_n(t)=\Pb_0\left( \widetilde{T}_{n, \, \beta} < t \right)$, $t \in \R$, we obtain for sufficiently small $\epsilon > 0$
\begin{align*}
    \lim_{n \to\infty} \frac{\log(1-F_n(\sqrt{n}\epsilon))}{n}
    &= \lim_{n \to\infty} \frac{\log \Pb_0\left( \widetilde{T}_{n, \, \beta} \geq \sqrt{n} \epsilon \right)}{n} \\
    &= \lim_{n \to\infty} \frac{\log \Pb_0\left( \, \left\| \frac{1}{n} \sum_{j=1}^n W_j(\cdot) \, \right\|_\infty \geq \epsilon \right)}{n}
    = -\frac{\epsilon^2}{2 \underset{b \in \Sp}{\sup} \E W_1^2(b)} + o(\epsilon^2).
\end{align*}
According to the proof of Theorem \ref{thm:asy_emp}, $W_1(\cdot)$ defines the covariance kernel $\rho_\beta$ of the Gaussian process $Z_\beta(\cdot)$, so that indeed
\begin{equation*}
    \lim_{n \to\infty} \frac{\log(1-F_n(\sqrt{n}\epsilon))}{n}
     = -\frac{\epsilon^2}{2 \spheremax \rho_\beta(b,b)} + o(\epsilon^2)
\end{equation*}
for sufficiently small $\epsilon > 0$. Due to the assumed condition \eqref{density_L1_conv}, it follows that
\begin{align*}
    \left| \gamma_\kappa(b) \right|
    &= \left| \int_{\Sp} (\skalarprodukt{b}{x})^{\, \beta} f(x \, | \, \kappa) \, d\sigma(x) - \psi_d(\beta) \right|
    = \left| \int_{\Sp} (\skalarprodukt{b}{x})^{\, \beta} \left( f(x \, | \, \kappa) - f(x \, | \, 0) \right) \, d\sigma(x) \right|  \\
    &\leq || f(\cdot \, | \, \kappa) - f(\cdot \, | \, 0) ||_{L^1}
    \overset{\kappa \rightarrow 0^+}{\longrightarrow} 0.
\end{align*}
Lebesgue's dominated convergence theorem shows that $\gamma_\kappa \in C(\Sp; \R)$ for each $\kappa > 0$, and the last calculation justifies the uniform convergence of
$\gamma_\kappa$ against $0$ on $\Sp$ for $\kappa \rightarrow 0^+$. This implies $\limkappa \spheremax \gamma_\kappa^2(b) = 0$. Thus, the exact Bahadur slope is
\begin{equation*}
    c_{\widetilde{T}_\beta}(\kappa)
    = \frac{\spheremax \gamma_\kappa^2(b)}{\spheremax \rho_\beta(b,b)} + o \left( \spheremax \gamma_\kappa^2(b) \right)
\end{equation*}
for sufficiently small $\kappa > 0$. Finally, notice that for each $b \in \Sp$, we have
\begin{align*}
    \rho_\beta(b,b)
    &= \eta_\beta(b,b) - \psi_d^2(\beta)
    = \sum_{j=0}^\beta \frac{(c_{j,\, d}(\beta))^2}{\nu_d(j)} P_j^{\, d}(\skalarprodukt{b}{b}) - \psi_d^2(\beta)
    = \sum_{j=0}^\beta \frac{(c_{j,\, d}(\beta))^2}{\nu_d(j)} - \psi_d^2(\beta) \\
    &= \frac{(c_{0, \, d}(\beta))^2}{\nu_d(0)} + \sum_{j=1}^\beta \lambda_j \nu_d(j) - \psi_d^2(\beta)
    = \sum_{j=1}^\beta \lambda_j \nu_d(j).
\end{align*}
Here, the second and third equality follow from \eqref{eta} and Remark \ref{legendre_schranke}, respectively.
\end{proof}

\end{appendix}

\vspace{5mm}
\noindent
J. Borodavka, \\
Steinbuch Centre for Computing, \\
Karlsruhe Institute of Technology (KIT), \\
Zirkel 2, D-76131 Karlsruhe. \\
E-mail: \href{mailto:Jaroslav.Borodavka@kit.edu}{Jaroslav.Borodavka@kit.edu} \\[3mm]
B. Ebner, \\
Institute of Stochastics, \\
Karlsruhe Institute of Technology (KIT), \\
Englerstr. 2, D-76128 Karlsruhe. \\
E-mail: \href{mailto:Bruno.Ebner@kit.edu}{Bruno.Ebner@kit.edu} 
\end{document}